\documentclass{amsart}

\usepackage[T1]{fontenc}
\usepackage{enumerate, amsmath, amsfonts, amssymb, amsthm, mathrsfs, wasysym, graphics, graphicx, xcolor, url, hyperref, hypcap, shuffle, xargs, multicol, overpic, pdflscape, multirow, hvfloat, minibox, accents, array, multido, xifthen, a4wide, ae, aecompl, blkarray, pifont, mathtools, etoolbox, dsfont}
\usepackage{marginnote}
\hypersetup{colorlinks=true, citecolor=darkblue, linkcolor=darkblue}
\usepackage[all]{xy}
\usepackage[bottom]{footmisc}
\usepackage{tikz}
\usetikzlibrary{trees, decorations, decorations.markings, shapes, arrows, matrix, calc, fit, intersections, patterns, angles}
\usepackage[external]{forest}
\graphicspath{{figures/}{figures/nodes/}}
\makeatletter\def\input@path{{figures/}}\makeatother
\usepackage{caption}
\usepackage[noabbrev,capitalise]{cleveref}
\usepackage[export]{adjustbox}
\usepackage{ulem}

\newtheorem{theorem}{Theorem}
\newtheorem{corollary}[theorem]{Corollary}
\newtheorem{proposition}[theorem]{Proposition}
\newtheorem{lemma}[theorem]{Lemma}

\newtheorem*{theorem*}{Theorem}

\theoremstyle{definition}

\newtheorem{example}[theorem]{Example}

\newtheorem{problem}[theorem]{Problem}

\crefname{notation}{Notation}{Notations}
\crefname{problem}{Problem}{Problems}
 
\newcommand{\R}{\mathbb{R}} 
\newcommand{\HH}{\mathbb{H}} 
\newcommand{\K}{\mathbb{K}} 
\renewcommand{\c}[1]{\mathcal{#1}} 
\renewcommand{\b}[1]{\boldsymbol{#1}} 
\newcommand{\bb}[1]{\mathbb{#1}} 
\newcommand{\f}[1]{\mathfrak{#1}} 

\newcommand{\set}[2]{\left\{ #1 \;\middle|\; #2 \right\}} 
\newcommand{\bigset}[2]{\big\{ #1 \;\big|\; #2 \big\}} 
\newcommand{\ssm}{\smallsetminus} 
\newcommand{\dotprod}[2]{\left\langle \, #1 \; \middle| \; #2 \, \right\rangle} 
\newcommand{\one}{\b{1}} 
\newcommand{\eqdef}{\mbox{\,\raisebox{0.2ex}{\scriptsize\ensuremath{\mathrm:}}\ensuremath{=}\,}} 
\newcommand{\simplex}{\b{\triangle}} 

\DeclareMathOperator{\conv}{conv} 

\newcommand{\ie}{\textit{i.e.}~} 
\newcommand{\eg}{\textit{e.g.}~} 
\newcommand{\viceversa}{\textit{vice versa}} 
\newcommand{\aka}{\textit{a.k.a.}~} 
\definecolor{darkblue}{rgb}{0,0,0.7} 
\definecolor{green}{RGB}{57,181,74} 
\definecolor{violet}{RGB}{147,39,143} 
\newcommand{\darkblue}{\color{darkblue}} 
\newcommand{\defn}[1]{\textsl{\darkblue #1}} 
\newcommand{\para}[1]{\bigskip\noindent\uline{#1.}} 

\usepackage{todonotes}
\newcommand{\vincent}[1]{} 

\newcommand{\meet}{\wedge} 
\newcommand{\join}{\vee} 
\newcommand{\bigMeet}{\bigwedge} 
\newcommand{\bigJoin}{\bigvee} 
\newcommandx{\projDown}[1][1={}]{\smash{\pi_\downarrow^{#1}}} 
\newcommandx{\projUp}[1][1={}]{\smash{\pi^\uparrow_{#1}}} 
\newcommand{\con}{\mathrm{con}} 

\newcommandx{\RO}[1][1=D]{\mathcal{R}_{#1}} 
\newcommandx{\AR}[1][1=D]{\mathcal{AR}_{#1}} 
\newcommandx{\restrictionMap}[1][1={D,D'}]{\phi_{#1}} 
\newcommandx{\inverseRestrictionMap}[1][1={D,D'}]{\psi_{#1}} 

\newcommand{\rope}{\rho} 
\newcommand{\diagram}{\delta} 
\newcommand{\up}{\triangle} 
\newcommand{\down}{\bigtriangledown} 

\newcommand{\decorationDown}{\mho} 
\newcommand{\decorationUp}{\Omega} 

\newcommandx{\PR}[1][1={\c{H},\polytope{B}}]{\mathcal{R}_{#1}} 
\newcommandx{\Arrang}[1][1=D]{\mathcal{H}_{#1}} 
\newcommandx{\Fan}[1][1=D]{\mathcal{F}_{#1}} 
\newcommand{\polytope}[1]{\mathds{#1}} 
\newcommandx{\Zono}[1][1=D]{\polytope{Z}_{#1}} 
\newcommandx{\Asso}[1][1=n]{\polytope{A}_{#1}} 
\newcommandx{\shardPolytope}[1][1=\rope]{\polytope{SP}_{#1}} 
\newcommandx{\quotientope}[1][1=\equiv]{\polytope{Q}_{#1}} 
\newcommand{\shard}{\Sigma} 

\makeatletter
\def\l@part{\@tocline{1}{8pt}{0pc}{}{}}
\def\l@section{\@tocline{1}{4pt}{0pc}{}{}}
\makeatother
\let\oldtocpart=\tocpart
\renewcommand{\tocpart}[2]{\sc\large\oldtocpart{#1}{#2}}
\let\oldtocsection=\tocsection
\renewcommand{\tocsection}[2]{\bf\oldtocsection{#1}{#2}}
\let\oldtocsubsubsection=\tocsubsubsection
\renewcommand{\tocsubsubsection}[2]{\quad\oldtocsubsubsection{#1}{#2}}

\title[Acyclic reorientation lattices and their lattice quotients]{Acyclic reorientation lattices \\ and their lattice quotients}

\thanks{The author was a CNRS researcher at \'Ecole Polytechnique when this work was done, and was partially supported by the French ANR grant CAPPS~17\,CE40\,0018 and CHARMS~19\,CE40\,0017.}

\author{Vincent Pilaud}
\address{Universitat de Barcelona}
\email{vincent.pilaud@ub.edu}
\urladdr{\url{https://www.ub.edu/comb/vincentpilaud/}}

\begin{document}

\begin{abstract}
We prove that the acyclic reorientation poset of a directed acyclic graph~$D$ is a lattice if and only if the transitive reduction of any induced subgraph of~$D$ is a forest.
We then show that the acyclic reorientation lattice is always congruence normal, semidistributive (thus congruence uniform) if and only if~$D$ is filled, and distributive if and only if~$D$ is a forest.
When the acyclic reorientation lattice is semidistributive, we introduce the ropes of~$D$ that encode the join irreducibles acyclic reorientations and exploit this combinatorial model in three directions.
First, we describe the canonical join and meet representations of acyclic reorientations in terms of non-crossing rope diagrams.
Second, we describe the congruences of the acyclic reorientation lattice in terms of lower ideals of a natural subrope order.
Third, we use Minkowski sums of shard polytopes of ropes to construct a quotientope for any congruence of the acyclic reorientation lattice.
\end{abstract}

\vspace*{.1cm}
\maketitle

\tableofcontents
\vspace*{-.5cm}


\addtocontents{toc}{\protect\setcounter{tocdepth}{1}}

\section*{Introduction and overview}
\label{sec:intro}

Fix a (finite and simple) directed graph~$D \eqdef (V, A)$.
A \defn{reorientation} of~$D$ is a directed graph with the same underlying undirected graph as~$D$.
It can be encoded by its set of reversed arcs with respect to~$D$.
The \defn{reorientation lattice}~$\RO$ is the boolean lattice formed by all reorientations of~$D$ ordered by inclusion of reversed sets (we denote this order by~$\le$).
Its minimal element is~$D$, its maximal element is the reverse~$\bar D$ of~$D$, its cover relations are given by flipping a single arc, and it is clearly self-dual under reversing all arcs.

\medskip
Assume now that~$D$ is a (finite and simple) directed acyclic graph.
The \defn{acyclic reorientation poset}~$\AR$ is the subposet of~$\RO$ induced by acyclic reorientations of~$D$.
Its minimal and maximal elements are still~$D$ and~$\bar D$, its cover relations are still given by flipping a single arc, and it is still self-dual under reversing all arcs.
For instance, the acyclic reorientation poset of any directed forest is a boolean lattice, and the acyclic reorientation poset of a tournament is isomorphic to the weak order on permutations.
Some examples are illustrated in \cref{fig:acyclicReorientationPosets}.

\medskip
These acyclic reorientations posets and the underlying acyclic orientation flip graphs have been extensively studied, in particular for counting~\cite{Stanley-acyclicOrientations, Lass}, traversing~\cite{SavageSquireWest, PruesseRuskey}, and generating~\cite{Squire, BarbosaSzwarcfiter} all acyclic orientations of a graph.
This paper considers these acyclic reorientation posets from a lattice theoretic perspective: after characterizing the directed acyclic graphs~$D$ for which~$\AR$ is a lattice, we explore lattice properties of~$\AR$, in particular the combinatorics and geometry of the lattice quotients of~$\AR$ when it turns out to be semidistributive.

\begin{figure}[b]
	\centerline{
		\begin{tabular}{c@{\quad}c@{\quad}c@{\quad}c}
			weak order lattice &
			boolean lattice & 
			another lattice & 
			not a lattice
			\\[.1cm]
			\includegraphics[scale=.9]{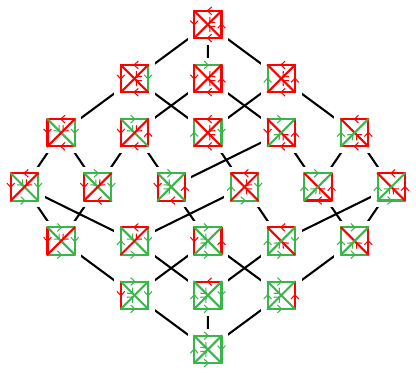} &
			\includegraphics[scale=.9]{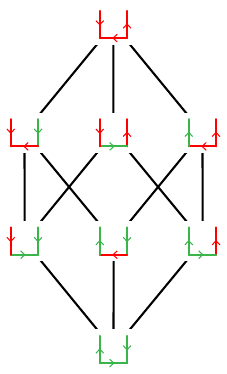} & 
			\includegraphics[scale=.9]{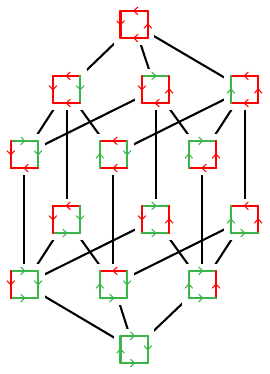} & 
			\includegraphics[scale=.9]{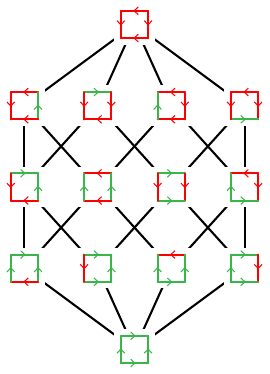}
		\end{tabular}
	}
	\caption{Some acyclic reorientation posets. The first three are lattices while the fourth is not. The first is the weak order on permutations since $D$ is a tournament, the second is boolean since $D$ is a forest. The green arcs agree with the reference orientation, while the red arcs are reversed.}
	\label{fig:acyclicReorientationPosets}
\end{figure}


\subsection*{Acyclic reorientation lattices}

Recall that the \defn{transitive reduction} (resp.~\defn{transitive closure}) of~$D$ is the directed graph obtained by deleting from (resp.~adding to) $D$ all arcs whose endpoints are connected by a directed path in~$D$ of length at least~$2$.
These operations clearly play an important role for acyclic reorientations: for instance, note that an arc in an acyclic reorientation~$E$ of~$D$ is flippable if and only if it belongs to the transitive reduction of~$E$.

\medskip
In this paper, we say that $D$ is \defn{vertebrate} when the transitive reduction of any induced subgraph of~$D$ is a forest.
For instance, any forest and any tournament is vertebrate.
Note that it is important to check all induced subgraphs of~$D$: there are directed acyclic graphs whose transitive reduction is a forest, but containing an induced subgraph whose transitive reduction is not a forest.
\vincent{In fact, the square with a central vertex is an example of a graph~$G$ such that for any orientation~$D$ of~$G$, the acyclic reorientation poset~$\AR$ is not a lattice.}
Our starting observation is the following result illustrated in \cref{fig:acyclicReorientationPosets}.

\begin{theorem}
\label{thm:characterizationLattice}
The acyclic reorientation poset~$\AR$ is a lattice if and only if $D$ is vertebrate.
\end{theorem}

We will actually provide two proofs of \cref{thm:characterizationLattice}.
Our first proof in \cref{sec:characterizationLattice} will describe the join and meet operations in the acyclic reorientation lattice of a vertebrate directed acyclic graph.
Our second proof in \cref{subsec:congruenceNormality} will show that the acyclic reorientation lattice of a vertebrate directed acyclic graph can be obtained from the acyclic reorientation lattice of its transitive reduction by a sequence of convex doublings in the sense of~\cite{Day}.
 

\subsection*{Restriction maps}

The natural restriction maps between acyclic reorientation posets provide an important tool in some proofs of this paper.
Consider two directed acyclic graphs~$D \eqdef (V, A)$ and~$D' \eqdef (V, A')$ on the same vertex set~$V$ with~$A \supseteq A'$.
Since~$A \supseteq A'$, any (acyclic) reorientation of~$D$ restricts to an (acyclic) reorientation of~$D'$.
The restriction map~$\restrictionMap : \AR \to \AR[D']$ is surjective and order preserving.
See \cref{fig:restrictionMap} for examples.

\begin{figure}[b]
	\centerline{
		\begin{tabular}{c@{\quad}c@{\quad}c@{\quad}c}
			strongly pathful &
			pathful &
			weakly pathful &
			nothing
			\\[.1cm]
			\includegraphics[scale=.9]{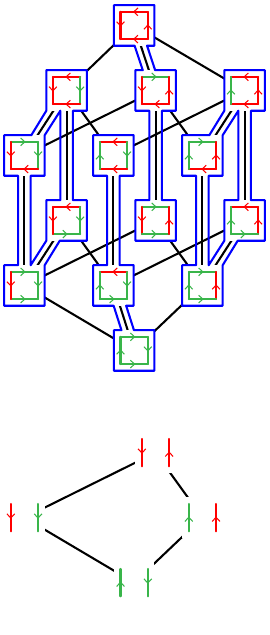} &
			\includegraphics[scale=.9]{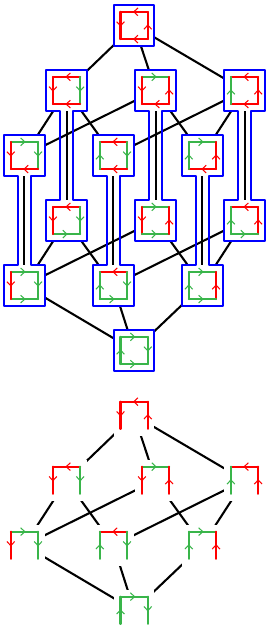} &
			\includegraphics[scale=.9]{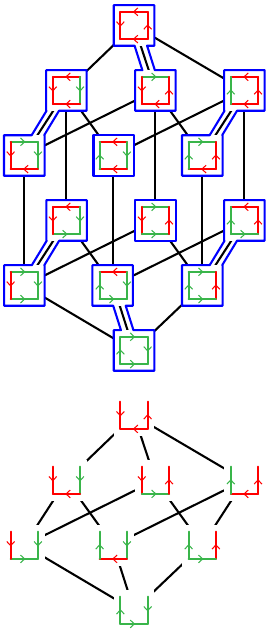} &
			\includegraphics[scale=.9]{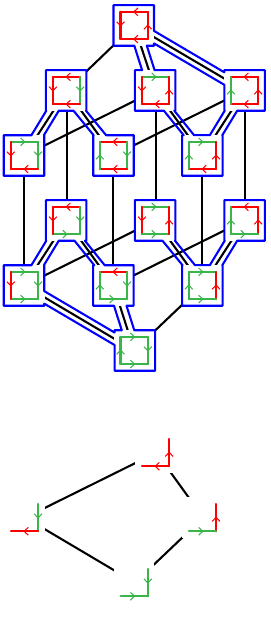}
		\end{tabular}
	}
	\caption{Restriction maps on acyclic reorientations. The fibers are represented as blue bubbles. The first subgraph of~$D$ is strongly pathful, the second is not strongly pathful but pathful, the third is not pathful but weakly pathful, the fourth is not weakly pathful.}
	\label{fig:restrictionMap}
\end{figure}

\medskip
\enlargethispage{1.1cm}
Assuming that both~$D$ and~$D'$ are vertebrate, we characterize some relevant lattice properties of this restriction map~$\restrictionMap$.
We say that~$D'$ is 
\begin{itemize}
\item \defn{weakly pathful} in~$D$ if along any directed path in~$D$ whose endpoints are connected by an arc in~$D'$, at most one arc does not belong to~$D'$,
\item \defn{pathful} in~$D$ if any directed path in~$D$ joining the endpoints of an arc in~$D'$ is \mbox{contained~in~$D'$,}
\item \defn{strongly pathful} in~$D$ if any directed path in~$D$ joining the endpoints of a directed path in~$D'$ is contained in~$D'$.
\end{itemize}
Note that strongly pathful implies pathful, and pathful implies weakly pathful, but that both reverse implications are wrong.
The next statement is proved in \cref{sec:restrictionMaps} and \mbox{illustrated in~\cref{fig:restrictionMap}}.

\begin{theorem}
\label{thm:latticeMap}
For two vertebrate directed acyclic graphs~$D \eqdef (V, A)$ and~$D' \eqdef (V, A')$ with~${A \supseteq A'}$,
\begin{itemize}
\item all fibers of~$\restrictionMap$ are intervals if and only if~$D'$ is weakly pathful in~$D$,
\item $\restrictionMap$ is a lattice quotient map if and only if~$D'$ is pathful in~$D$,
\item $\restrictionMap$ restricts to a lattice isomorphism from a lower (or upper) interval of~$\AR$ to~$\AR[D']$ if and only if~$D'$ is strongly pathful in~$D$.
\end{itemize}
\end{theorem}

Specializing \cref{thm:latticeMap} in the situation when~$D$ is a tournament, we obtain in \cref{exm:latticeQuotientsWeakOrder} a bijection between the directed acyclic graphs~$D'$ whose acyclic reorientation poset~$\AR[D']$ is a lattice quotient of the weak order on~$\f{S}_n$ and the non-nesting partitions of~$[n] \eqdef \{1, \dots, n\}$, which are counted by the Catalan number~$C_n \eqdef \frac{1}{n+1} \binom{2n}{n}$.


\subsection*{Lattice properties}

We assume now that $D$ is vertebrate and discuss some properties of its acyclic reorientation lattice~$\AR$.
We refer to \cref{sec:latticeProperties} for the definitions and characterizations of the classical notions of distributivity, semidistributivity, congruence normality, and congruence uniformity of lattices.
We say that~$D$ is \defn{filled} when for any directed path~$\pi$ in~$D$, if the arc joining the endpoints of~$\pi$ belongs to~$D$, then all arcs joining any two vertices of~$\pi$ also belong~to~$D$.
For instance, any forest and any tournament is filled.
The following statement is proved in \cref{sec:latticeProperties} and illustrated in \cref{fig:distributiveSemidistributive}.

\begin{theorem}
\label{thm:propertiesAcyclicReorientationLattice}
When~$D$ is vertebrate, the acyclic reorientation lattice~$\AR$ is
\begin{itemize}
\item distributive if and only if $D$ is a forest,
\item semidistributive is and only if $D$ is filled,
\item always congruence normal (\aka constructible by convex doubling),
\item congruence uniform (\aka constructible by interval doubling) if and only if~$D$ is filled.
\end{itemize}
\end{theorem}

\begin{figure}[h]
	\centerline{
		\begin{tabular}{c@{\qquad}c@{\qquad}c}
			distributive &
			semidistributive &
			not semidistributive
			\\[.1cm]
			\includegraphics[scale=.9]{acyclicReorientationPoset1} &
			\includegraphics[scale=.9]{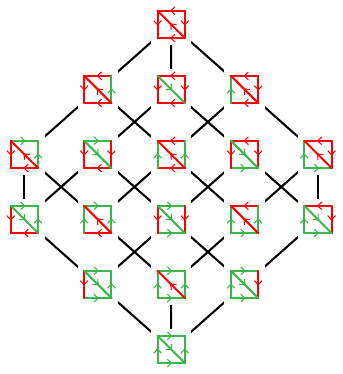} &
			\includegraphics[scale=.9]{acyclicReorientationPoset5}
		\end{tabular}
	}
	\caption{Some acyclic reorientation lattices. The first is distributive, the second is not distributive but semidistributive, the third is not semidistributive. They are all congruence normal, hence the first two are also congruence uniform.}
	\label{fig:distributiveSemidistributive}
\end{figure}

Note that our proof of the congruence normality is based on doubling of order convex sets~\cite{Day}, and thus provides an alternative proof of \cref{thm:characterizationLattice}.

\medskip
The remaining of the paper focusses on the situation when~$D$ is vertebrate and filled, which we abbreviate into \defn{skeletal}.
As for the lattice property, we also provide two proofs of semidistributivity.
Our first proof in \cref{subsec:semidistributivity} will enable us to describe the canonical join and meet representations in the acyclic reorientation lattice of a skeletal directed acyclic graph.
Our second proof in \cref{subsec:congruenceNormality} will show that the acyclic reorientation lattice of a skeletal directed acyclic graph can be obtained from the acyclic reorientation lattice of its transitive reduction by a sequence of interval doublings in the sense of~\cite{Day}.
All the results of the remaining sections exploit the join irreducible elements of the acyclic reorientation lattice~$\AR$ to describe all its elements, its congruences and its quotients when~$D$ is skeletal.
Our approach is based on a convenient combinatorial model for join irreducibles of~$\AR$, extending the arcs of N.~Reading~\cite{Reading-arcDiagrams}, which provides simple combinatorial descriptions of the compatibility relation and the forcing order among join irreducibles, as we discuss next.


\subsection*{Ropes}

Assume that~$D$ is skeletal, so that its acyclic reorientation poset~$\AR$ is a congruence uniform lattice.
Generalizing the arcs of N.~Reading~\cite{Reading-arcDiagrams}, we introduce in \cref{subsec:ropes} some combinatorial gadgets, that we call the \defn{ropes} of~$D$, to encode the join (or meet) irreducible elements of~$\AR$.
We use these ropes to describe
\begin{itemize}
\item the canonical join complex of~$\AR$ (whose faces are the canonical join representations of~$\AR$) in terms of non-crossing rope diagrams of~$D$ in \cref{subsec:ropeDiagrams},
\item the canonical complex of~$\AR$ (whose faces are in bijection with intervals of~$\AR$) in terms of rope bidiagrams of~$D$ in \cref{subsec:bidiagrams},
\item the forcing order among join irreducibles of~$\AR$ (whose lower ideals correspond to lattice quotients of~$\AR$) in terms of subropes in~$D$ in \cref{subsec:subropes}.
\end{itemize}
The subrope order enables us to describe and manipulate all congruences of the acyclic reorientation lattice~$\AR$.
For instance, the minimal and maximal elements of the classes of a congruence~$\equiv$ correspond to non-crossing rope diagrams contained in the subrope ideal associated to~$\equiv$.

\medskip
Using ropes, we also introduce and explore in \cref{subsec:coherentCongruencesPrincipalCongruences} some particularly relevant congruences of~$\AR$: the principal congruences corresponding to principal ideals of the subrope order, and the coherent congruences generalizing the sylvester~\cite{HivertNovelliThibon-algebraBinarySearchTrees}, Cambrian~\cite{Reading-CambrianLattices}, and permutree~\cite{PilaudPons-permutrees} congruences of the weak order on permutations.
For the coherent congruences, we provide analogues of the classical properties of the sylvester congruence: we describe each coherent congruence as the transitive closure of certain allowed arc flips, we describe the minimal and maximal acyclic reorientations in the congruence classes in terms of avoidance of certain patterns, and we discuss the partial acyclic reorientations encoding the elements and the intervals of the corresponding quotient generalizing~\cite{ChatelPilaudPons}.


\subsection*{Quotientopes}

As originally observed by C.~Greene~\cite{Greene} (see also~\cite[Lem.~7.1]{GreeneZaslavsky}), the Hasse diagram of the acyclic reorientation poset~$\AR$ can be interpreted geometrically as 
\begin{itemize}
\item the dual graph of the \defn{graphical fan}~$\Fan$, defined by the \defn{graphical arrangement} of~$D$ containing the hyperplanes~$\bigset{\b{x} \in \R^V}{x_u = x_v}$ for all arcs~$(u,v) \in D$, oriented in the linear direction~$\b{\omega}_D \eqdef \sum_{(u,v) \in A} \b{e}_v - \b{e}_u$, or
\item the graph of the \defn{graphical zonotope}~$\Zono$, defined as the Minkowski sum of all segments~$[\b{e}_u, \b{e}_v]$ for all arcs~$(u,v)$ of~$D$, oriented in the linear direction~$\b{\omega}_D$.
\end{itemize}
Note that the graphical fan and the graphical zonotope are dual to each other, and that their codimension is the number of connected components of~$D$.
For instance, the graphical arrangements and graphical zonotopes corresponding to the acyclic reorientation posets of \cref{fig:acyclicReorientationPosets} are illustrated in \cref{fig:graphicalArrangementsIntro,fig:graphicalZonotopesIntro}.

\begin{figure}[b]
	\centerline{
			\includegraphics[scale=.5]{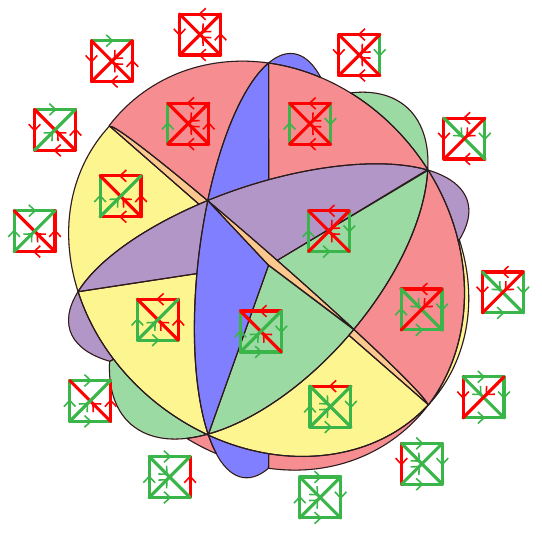}
			\includegraphics[scale=.5]{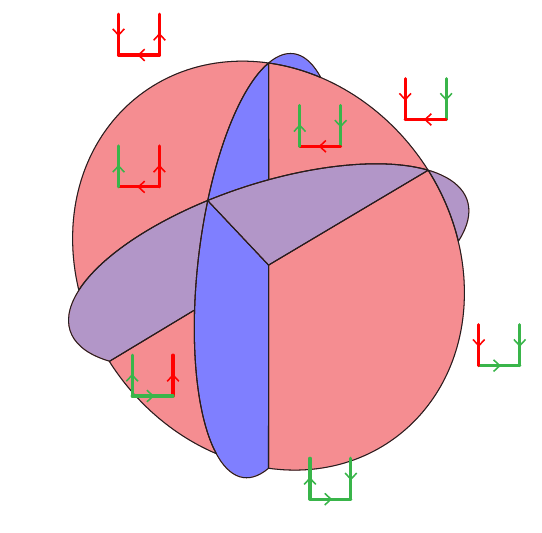}
			\includegraphics[scale=.5]{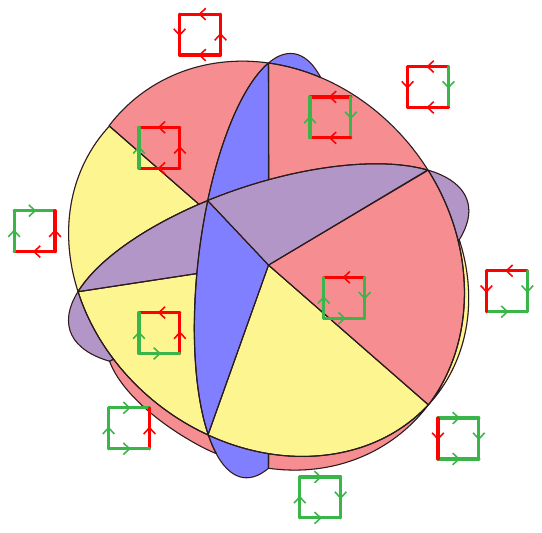}
			\includegraphics[scale=.5]{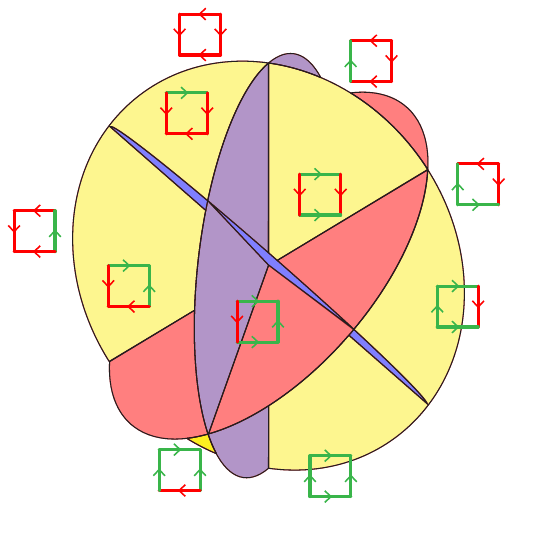}
	}
	\caption{The graphical arrangements corresponding to the acyclic reorientation posets of \cref{fig:acyclicReorientationPosets}. The first is the classical braid arrangement. The regions are labeled by the corresponding acyclic reorientations. The hyperplanes are colored according to the corresponding arc. The perspective is chosen so that the minimal reorientation appears at the bottom of the picture.}
	\label{fig:graphicalArrangementsIntro}
\end{figure}

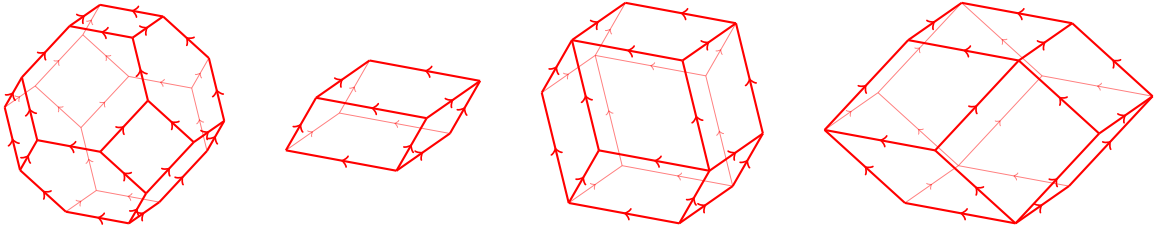
\begin{figure}
	\centerline{

\begin{tikzpicture}%
	[x={(-0.366215cm, -0.789554cm)},
	y={(0.235950cm, -0.590693cm)},
	z={(0.900119cm, -0.166391cm)},
	scale=.570000,
	back/.style={very thin, opacity=0.5},
	edge/.style={color=red, thick, decoration={markings, mark=at position 0.5 with {\arrow{>}}}},
	facet/.style={fill=red,fill opacity=0},
	vertex/.style={}]
%
%

\coordinate (0.00000, 0.00000, 0.00000) at (0.00000, 0.00000, 0.00000);
\coordinate (0.00000, 0.00000, 1.63299) at (0.00000, 0.00000, 1.63299);
\coordinate (0.00000, 1.41421, -0.81650) at (0.00000, 1.41421, -0.81650);
\coordinate (4.00000, 2.82843, 1.63299) at (4.00000, 2.82843, 1.63299);
\coordinate (0.00000, 1.41421, 2.44949) at (0.00000, 1.41421, 2.44949);
\coordinate (0.00000, 2.82843, 0.00000) at (0.00000, 2.82843, 0.00000);
\coordinate (0.00000, 2.82843, 1.63299) at (0.00000, 2.82843, 1.63299);
\coordinate (1.33333, -0.94281, 0.00000) at (1.33333, -0.94281, 0.00000);
\coordinate (1.33333, -0.94281, 1.63299) at (1.33333, -0.94281, 1.63299);
\coordinate (4.00000, 2.82843, 0.00000) at (4.00000, 2.82843, 0.00000);
\coordinate (4.00000, 1.41421, 2.44949) at (4.00000, 1.41421, 2.44949);
\coordinate (4.00000, 1.41421, -0.81650) at (4.00000, 1.41421, -0.81650);
\coordinate (1.33333, 1.88562, -1.63299) at (1.33333, 1.88562, -1.63299);
\coordinate (4.00000, 0.00000, 1.63299) at (4.00000, 0.00000, 1.63299);
\coordinate (4.00000, 0.00000, 0.00000) at (4.00000, 0.00000, 0.00000);
\coordinate (1.33333, 1.88562, 3.26599) at (1.33333, 1.88562, 3.26599);
\coordinate (1.33333, 3.29983, -0.81650) at (1.33333, 3.29983, -0.81650);
\coordinate (2.66667, 3.77124, 1.63299) at (2.66667, 3.77124, 1.63299);
\coordinate (1.33333, 3.29983, 2.44949) at (1.33333, 3.29983, 2.44949);
\coordinate (2.66667, -0.47140, -0.81650) at (2.66667, -0.47140, -0.81650);
\coordinate (2.66667, 3.77124, 0.00000) at (2.66667, 3.77124, 0.00000);
\coordinate (2.66667, -0.47140, 2.44949) at (2.66667, -0.47140, 2.44949);
\coordinate (2.66667, 0.94281, -1.63299) at (2.66667, 0.94281, -1.63299);
\coordinate (2.66667, 0.94281, 3.26599) at (2.66667, 0.94281, 3.26599);
\draw[edge,postaction={decorate},back] (0.00000, 1.41421, -0.81650) -- (0.00000, 0.00000, 0.00000);
\draw[edge,postaction={decorate},back] (0.00000, 2.82843, 0.00000) -- (0.00000, 1.41421, -0.81650);
\draw[edge,postaction={decorate},back] (1.33333, 1.88562, -1.63299) -- (0.00000, 1.41421, -0.81650);
\draw[edge,postaction={decorate},back] (0.00000, 2.82843, 1.63299) -- (0.00000, 1.41421, 2.44949);
\draw[edge,postaction={decorate},back] (0.00000, 2.82843, 1.63299) -- (0.00000, 2.82843, 0.00000);
\draw[edge,postaction={decorate},back] (1.33333, 3.29983, -0.81650) -- (0.00000, 2.82843, 0.00000);
\draw[edge,postaction={decorate},back] (1.33333, 3.29983, 2.44949) -- (0.00000, 2.82843, 1.63299);
\draw[edge,postaction={decorate},back] (4.00000, 2.82843, 0.00000) -- (2.66667, 3.77124, 0.00000);
\draw[edge,postaction={decorate},back] (1.33333, 3.29983, -0.81650) -- (1.33333, 1.88562, -1.63299);
\draw[edge,postaction={decorate},back] (2.66667, 0.94281, -1.63299) -- (1.33333, 1.88562, -1.63299);
\draw[edge,postaction={decorate},back] (2.66667, 3.77124, 0.00000) -- (1.33333, 3.29983, -0.81650);
\draw[edge,postaction={decorate},back] (2.66667, 3.77124, 1.63299) -- (2.66667, 3.77124, 0.00000);
\node[vertex] at (1.33333, 1.88562, -1.63299)     {};
\node[vertex] at (1.33333, 3.29983, -0.81650)     {};
\node[vertex] at (2.66667, 3.77124, 0.00000)     {};
\node[vertex] at (0.00000, 1.41421, -0.81650)     {};
\node[vertex] at (0.00000, 2.82843, 0.00000)     {};
\node[vertex] at (0.00000, 2.82843, 1.63299)     {};
\fill[facet] (2.66667, 0.94281, -1.63299) -- (4.00000, 1.41421, -0.81650) -- (4.00000, 0.00000, 0.00000) -- (2.66667, -0.47140, -0.81650) -- cycle {};
\fill[facet] (1.33333, -0.94281, 1.63299) -- (0.00000, 0.00000, 1.63299) -- (0.00000, 0.00000, 0.00000) -- (1.33333, -0.94281, 0.00000) -- cycle {};
\fill[facet] (2.66667, 0.94281, 3.26599) -- (1.33333, 1.88562, 3.26599) -- (0.00000, 1.41421, 2.44949) -- (0.00000, 0.00000, 1.63299) -- (1.33333, -0.94281, 1.63299) -- (2.66667, -0.47140, 2.44949) -- cycle {};
\fill[facet] (4.00000, 0.00000, 0.00000) -- (4.00000, 0.00000, 1.63299) -- (2.66667, -0.47140, 2.44949) -- (1.33333, -0.94281, 1.63299) -- (1.33333, -0.94281, 0.00000) -- (2.66667, -0.47140, -0.81650) -- cycle {};
\fill[facet] (2.66667, 0.94281, 3.26599) -- (4.00000, 1.41421, 2.44949) -- (4.00000, 0.00000, 1.63299) -- (2.66667, -0.47140, 2.44949) -- cycle {};
\fill[facet] (1.33333, 1.88562, 3.26599) -- (2.66667, 0.94281, 3.26599) -- (4.00000, 1.41421, 2.44949) -- (4.00000, 2.82843, 1.63299) -- (2.66667, 3.77124, 1.63299) -- (1.33333, 3.29983, 2.44949) -- cycle {};
\fill[facet] (4.00000, 0.00000, 0.00000) -- (4.00000, 1.41421, -0.81650) -- (4.00000, 2.82843, 0.00000) -- (4.00000, 2.82843, 1.63299) -- (4.00000, 1.41421, 2.44949) -- (4.00000, 0.00000, 1.63299) -- cycle {};
\draw[edge,postaction={decorate}] (0.00000, 0.00000, 1.63299) -- (0.00000, 0.00000, 0.00000);
\draw[edge,postaction={decorate}] (1.33333, -0.94281, 0.00000) -- (0.00000, 0.00000, 0.00000);
\draw[edge,postaction={decorate}] (0.00000, 1.41421, 2.44949) -- (0.00000, 0.00000, 1.63299);
\draw[edge,postaction={decorate}] (1.33333, -0.94281, 1.63299) -- (0.00000, 0.00000, 1.63299);
\draw[edge,postaction={decorate}] (4.00000, 2.82843, 1.63299) -- (4.00000, 2.82843, 0.00000);
\draw[edge,postaction={decorate}] (4.00000, 2.82843, 1.63299) -- (4.00000, 1.41421, 2.44949);
\draw[edge,postaction={decorate}] (4.00000, 2.82843, 1.63299) -- (2.66667, 3.77124, 1.63299);
\draw[edge,postaction={decorate}] (1.33333, 1.88562, 3.26599) -- (0.00000, 1.41421, 2.44949);
\draw[edge,postaction={decorate}] (1.33333, -0.94281, 1.63299) -- (1.33333, -0.94281, 0.00000);
\draw[edge,postaction={decorate}] (2.66667, -0.47140, -0.81650) -- (1.33333, -0.94281, 0.00000);
\draw[edge,postaction={decorate}] (2.66667, -0.47140, 2.44949) -- (1.33333, -0.94281, 1.63299);
\draw[edge,postaction={decorate}] (4.00000, 2.82843, 0.00000) -- (4.00000, 1.41421, -0.81650);
\draw[edge,postaction={decorate}] (4.00000, 1.41421, 2.44949) -- (4.00000, 0.00000, 1.63299);
\draw[edge,postaction={decorate}] (4.00000, 1.41421, 2.44949) -- (2.66667, 0.94281, 3.26599);
\draw[edge,postaction={decorate}] (4.00000, 1.41421, -0.81650) -- (4.00000, 0.00000, 0.00000);
\draw[edge,postaction={decorate}] (4.00000, 1.41421, -0.81650) -- (2.66667, 0.94281, -1.63299);
\draw[edge,postaction={decorate}] (4.00000, 0.00000, 1.63299) -- (4.00000, 0.00000, 0.00000);
\draw[edge,postaction={decorate}] (4.00000, 0.00000, 1.63299) -- (2.66667, -0.47140, 2.44949);
\draw[edge,postaction={decorate}] (4.00000, 0.00000, 0.00000) -- (2.66667, -0.47140, -0.81650);
\draw[edge,postaction={decorate}] (1.33333, 3.29983, 2.44949) -- (1.33333, 1.88562, 3.26599);
\draw[edge,postaction={decorate}] (2.66667, 0.94281, 3.26599) -- (1.33333, 1.88562, 3.26599);
\draw[edge,postaction={decorate}] (2.66667, 3.77124, 1.63299) -- (1.33333, 3.29983, 2.44949);
\draw[edge,postaction={decorate}] (2.66667, 0.94281, -1.63299) -- (2.66667, -0.47140, -0.81650);
\draw[edge,postaction={decorate}] (2.66667, 0.94281, 3.26599) -- (2.66667, -0.47140, 2.44949);
\node[vertex] at (0.00000, 0.00000, 0.00000)     {};
\node[vertex] at (0.00000, 0.00000, 1.63299)     {};
\node[vertex] at (4.00000, 2.82843, 1.63299)     {};
\node[vertex] at (0.00000, 1.41421, 2.44949)     {};
\node[vertex] at (1.33333, -0.94281, 0.00000)     {};
\node[vertex] at (1.33333, -0.94281, 1.63299)     {};
\node[vertex] at (4.00000, 2.82843, 0.00000)     {};
\node[vertex] at (4.00000, 1.41421, 2.44949)     {};
\node[vertex] at (4.00000, 1.41421, -0.81650)     {};
\node[vertex] at (4.00000, 0.00000, 1.63299)     {};
\node[vertex] at (4.00000, 0.00000, 0.00000)     {};
\node[vertex] at (1.33333, 1.88562, 3.26599)     {};
\node[vertex] at (2.66667, 3.77124, 1.63299)     {};
\node[vertex] at (1.33333, 3.29983, 2.44949)     {};
\node[vertex] at (2.66667, -0.47140, -0.81650)     {};
\node[vertex] at (2.66667, -0.47140, 2.44949)     {};
\node[vertex] at (2.66667, 0.94281, -1.63299)     {};
\node[vertex] at (2.66667, 0.94281, 3.26599)     {};
\end{tikzpicture}
		\quad
		\raisebox{.7cm}{

\begin{tikzpicture}%
	[x={(-0.366215cm, -0.789554cm)},
	y={(0.235950cm, -0.590693cm)},
	z={(0.900119cm, -0.166391cm)},
	scale=1.000000,
	back/.style={very thin, opacity=0.5},
	edge/.style={color=red, thick, decoration={markings, mark=at position 0.5 with {\arrow{>}}}},
	facet/.style={fill=red,fill opacity=0},
	vertex/.style={}]
%
%

\coordinate (0.00000, 0.00000, 0.00000) at (0.00000, 0.00000, 0.00000);
\coordinate (0.00000, 0.00000, 1.63299) at (0.00000, 0.00000, 1.63299);
\coordinate (0.00000, 1.41421, -0.81650) at (0.00000, 1.41421, -0.81650);
\coordinate (0.00000, 1.41421, 0.81650) at (0.00000, 1.41421, 0.81650);
\coordinate (1.33333, -0.94281, 0.00000) at (1.33333, -0.94281, 0.00000);
\coordinate (1.33333, -0.94281, 1.63299) at (1.33333, -0.94281, 1.63299);
\coordinate (1.33333, 0.47140, -0.81650) at (1.33333, 0.47140, -0.81650);
\coordinate (1.33333, 0.47140, 0.81650) at (1.33333, 0.47140, 0.81650);
\draw[edge,postaction={decorate},back] (0.00000, 1.41421, -0.81650) -- (0.00000, 0.00000, 0.00000);
\draw[edge,postaction={decorate},back] (0.00000, 1.41421, 0.81650) -- (0.00000, 1.41421, -0.81650);
\draw[edge,postaction={decorate},back] (1.33333, 0.47140, -0.81650) -- (0.00000, 1.41421, -0.81650);
\node[vertex] at (0.00000, 1.41421, -0.81650)     {};
\fill[facet] (1.33333, -0.94281, 1.63299) -- (0.00000, 0.00000, 1.63299) -- (0.00000, 0.00000, 0.00000) -- (1.33333, -0.94281, 0.00000) -- cycle {};
\fill[facet] (1.33333, 0.47140, 0.81650) -- (0.00000, 1.41421, 0.81650) -- (0.00000, 0.00000, 1.63299) -- (1.33333, -0.94281, 1.63299) -- cycle {};
\fill[facet] (1.33333, 0.47140, 0.81650) -- (1.33333, -0.94281, 1.63299) -- (1.33333, -0.94281, 0.00000) -- (1.33333, 0.47140, -0.81650) -- cycle {};
\draw[edge,postaction={decorate}] (0.00000, 0.00000, 1.63299) -- (0.00000, 0.00000, 0.00000);
\draw[edge,postaction={decorate}] (1.33333, -0.94281, 0.00000) -- (0.00000, 0.00000, 0.00000);
\draw[edge,postaction={decorate}] (0.00000, 1.41421, 0.81650) -- (0.00000, 0.00000, 1.63299);
\draw[edge,postaction={decorate}] (1.33333, -0.94281, 1.63299) -- (0.00000, 0.00000, 1.63299);
\draw[edge,postaction={decorate}] (1.33333, 0.47140, 0.81650) -- (0.00000, 1.41421, 0.81650);
\draw[edge,postaction={decorate}] (1.33333, -0.94281, 1.63299) -- (1.33333, -0.94281, 0.00000);
\draw[edge,postaction={decorate}] (1.33333, 0.47140, -0.81650) -- (1.33333, -0.94281, 0.00000);
\draw[edge,postaction={decorate}] (1.33333, 0.47140, 0.81650) -- (1.33333, -0.94281, 1.63299);
\draw[edge,postaction={decorate}] (1.33333, 0.47140, 0.81650) -- (1.33333, 0.47140, -0.81650);
\node[vertex] at (0.00000, 0.00000, 0.00000)     {};
\node[vertex] at (0.00000, 0.00000, 1.63299)     {};
\node[vertex] at (0.00000, 1.41421, 0.81650)     {};
\node[vertex] at (1.33333, -0.94281, 0.00000)     {};
\node[vertex] at (1.33333, -0.94281, 1.63299)     {};
\node[vertex] at (1.33333, 0.47140, -0.81650)     {};
\node[vertex] at (1.33333, 0.47140, 0.81650)     {};
\end{tikzpicture}}
		\quad

\begin{tikzpicture}%
	[x={(-0.366215cm, -0.789554cm)},
	y={(0.235950cm, -0.590693cm)},
	z={(0.900119cm, -0.166391cm)},
	scale=1.000000,
	back/.style={very thin, opacity=0.5},
	edge/.style={color=red, thick, decoration={markings, mark=at position 0.5 with {\arrow{>}}}},
	facet/.style={fill=red,fill opacity=0},
	vertex/.style={}]
%
%

\coordinate (0.00000, 0.00000, 0.00000) at (0.00000, 0.00000, 0.00000);
\coordinate (0.00000, 0.00000, 1.63299) at (0.00000, 0.00000, 1.63299);
\coordinate (0.00000, 1.41421, -0.81650) at (0.00000, 1.41421, -0.81650);
\coordinate (0.00000, 1.41421, 0.81650) at (0.00000, 1.41421, 0.81650);
\coordinate (1.33333, -0.94281, 0.00000) at (1.33333, -0.94281, 0.00000);
\coordinate (1.33333, -0.94281, 1.63299) at (1.33333, -0.94281, 1.63299);
\coordinate (1.33333, 0.47140, -0.81650) at (1.33333, 0.47140, -0.81650);
\coordinate (2.66667, 0.94281, 1.63299) at (2.66667, 0.94281, 1.63299);
\coordinate (1.33333, 0.47140, 2.44949) at (1.33333, 0.47140, 2.44949);
\coordinate (1.33333, 1.88562, 0.00000) at (1.33333, 1.88562, 0.00000);
\coordinate (1.33333, 1.88562, 1.63299) at (1.33333, 1.88562, 1.63299);
\coordinate (2.66667, -0.47140, 0.81650) at (2.66667, -0.47140, 0.81650);
\coordinate (2.66667, -0.47140, 2.44949) at (2.66667, -0.47140, 2.44949);
\coordinate (2.66667, 0.94281, 0.00000) at (2.66667, 0.94281, 0.00000);
\draw[edge,back,postaction={decorate}] (0.00000, 1.41421, -0.81650) -- (0.00000, 0.00000, 0.00000);
\draw[edge,back,postaction={decorate}] (0.00000, 1.41421, 0.81650) -- (0.00000, 0.00000, 1.63299);
\draw[edge,back,postaction={decorate}] (0.00000, 1.41421, 0.81650) -- (0.00000, 1.41421, -0.81650);
\draw[edge,back,postaction={decorate}] (1.33333, 0.47140, -0.81650) -- (0.00000, 1.41421, -0.81650);
\draw[edge,back,postaction={decorate}] (1.33333, 1.88562, 0.00000) -- (0.00000, 1.41421, -0.81650);
\draw[edge,back,postaction={decorate}] (1.33333, 1.88562, 1.63299) -- (0.00000, 1.41421, 0.81650);
\draw[edge,back,postaction={decorate}] (1.33333, 1.88562, 1.63299) -- (1.33333, 1.88562, 0.00000);
\draw[edge,back,postaction={decorate}] (2.66667, 0.94281, 0.00000) -- (1.33333, 1.88562, 0.00000);
\node[vertex] at (0.00000, 1.41421, -0.81650)     {};
\node[vertex] at (1.33333, 1.88562, 0.00000)     {};
\node[vertex] at (0.00000, 1.41421, 0.81650)     {};
\fill[facet] (2.66667, 0.94281, 0.00000) -- (1.33333, 0.47140, -0.81650) -- (1.33333, -0.94281, 0.00000) -- (2.66667, -0.47140, 0.81650) -- cycle {};
\fill[facet] (1.33333, -0.94281, 1.63299) -- (0.00000, 0.00000, 1.63299) -- (0.00000, 0.00000, 0.00000) -- (1.33333, -0.94281, 0.00000) -- cycle {};
\fill[facet] (2.66667, -0.47140, 2.44949) -- (1.33333, -0.94281, 1.63299) -- (0.00000, 0.00000, 1.63299) -- (1.33333, 0.47140, 2.44949) -- cycle {};
\fill[facet] (2.66667, -0.47140, 2.44949) -- (1.33333, -0.94281, 1.63299) -- (1.33333, -0.94281, 0.00000) -- (2.66667, -0.47140, 0.81650) -- cycle {};
\fill[facet] (1.33333, 0.47140, 2.44949) -- (2.66667, -0.47140, 2.44949) -- (2.66667, 0.94281, 1.63299) -- (1.33333, 1.88562, 1.63299) -- cycle {};
\fill[facet] (2.66667, -0.47140, 0.81650) -- (2.66667, 0.94281, 0.00000) -- (2.66667, 0.94281, 1.63299) -- (2.66667, -0.47140, 2.44949) -- cycle {};
\draw[edge,postaction={decorate}] (0.00000, 0.00000, 1.63299) -- (0.00000, 0.00000, 0.00000);
\draw[edge,postaction={decorate}] (1.33333, -0.94281, 0.00000) -- (0.00000, 0.00000, 0.00000);
\draw[edge,postaction={decorate}] (1.33333, -0.94281, 1.63299) -- (0.00000, 0.00000, 1.63299);
\draw[edge,postaction={decorate}] (1.33333, 0.47140, 2.44949) -- (0.00000, 0.00000, 1.63299);
\draw[edge,postaction={decorate}] (1.33333, -0.94281, 1.63299) -- (1.33333, -0.94281, 0.00000);
\draw[edge,postaction={decorate}] (1.33333, 0.47140, -0.81650) -- (1.33333, -0.94281, 0.00000);
\draw[edge,postaction={decorate}] (2.66667, -0.47140, 0.81650) -- (1.33333, -0.94281, 0.00000);
\draw[edge,postaction={decorate}] (2.66667, -0.47140, 2.44949) -- (1.33333, -0.94281, 1.63299);
\draw[edge,postaction={decorate}] (2.66667, 0.94281, 0.00000) -- (1.33333, 0.47140, -0.81650);
\draw[edge,postaction={decorate}] (2.66667, 0.94281, 1.63299) -- (1.33333, 1.88562, 1.63299);
\draw[edge,postaction={decorate}] (2.66667, 0.94281, 1.63299) -- (2.66667, -0.47140, 2.44949);
\draw[edge,postaction={decorate}] (2.66667, 0.94281, 1.63299) -- (2.66667, 0.94281, 0.00000);
\draw[edge,postaction={decorate}] (1.33333, 1.88562, 1.63299) -- (1.33333, 0.47140, 2.44949);
\draw[edge,postaction={decorate}] (2.66667, -0.47140, 2.44949) -- (1.33333, 0.47140, 2.44949);
\draw[edge,postaction={decorate}] (2.66667, -0.47140, 2.44949) -- (2.66667, -0.47140, 0.81650);
\draw[edge,postaction={decorate}] (2.66667, 0.94281, 0.00000) -- (2.66667, -0.47140, 0.81650);
\node[vertex] at (0.00000, 0.00000, 0.00000)     {};
\node[vertex] at (0.00000, 0.00000, 1.63299)     {};
\node[vertex] at (1.33333, -0.94281, 0.00000)     {};
\node[vertex] at (1.33333, -0.94281, 1.63299)     {};
\node[vertex] at (1.33333, 0.47140, -0.81650)     {};
\node[vertex] at (2.66667, 0.94281, 1.63299)     {};
\node[vertex] at (1.33333, 0.47140, 2.44949)     {};
\node[vertex] at (1.33333, 1.88562, 1.63299)     {};
\node[vertex] at (2.66667, -0.47140, 0.81650)     {};
\node[vertex] at (2.66667, -0.47140, 2.44949)     {};
\node[vertex] at (2.66667, 0.94281, 0.00000)     {};
\end{tikzpicture}
		\quad

\begin{tikzpicture}%
	[x={(-0.366215cm, -0.789554cm)},
	y={(0.235950cm, -0.590693cm)},
	z={(0.900119cm, -0.166391cm)},
	scale=1.000000,
	back/.style={very thin, opacity=0.5},
	edge/.style={color=red, thick, decoration={markings, mark=at position 0.5 with {\arrow{>}}}},
	facet/.style={fill=red,fill opacity=0},
	vertex/.style={}]
%
%

\coordinate (0.00000, 0.00000, 0.00000) at (0.00000, 0.00000, 0.00000);
\coordinate (0.00000, 0.00000, 1.63299) at (0.00000, 0.00000, 1.63299);
\coordinate (0.00000, 1.41421, 0.81650) at (0.00000, 1.41421, 0.81650);
\coordinate (0.00000, 1.41421, 2.44949) at (0.00000, 1.41421, 2.44949);
\coordinate (1.33333, -0.94281, 0.00000) at (1.33333, -0.94281, 0.00000);
\coordinate (1.33333, -0.94281, 1.63299) at (1.33333, -0.94281, 1.63299);
\coordinate (1.33333, 0.47140, -0.81650) at (1.33333, 0.47140, -0.81650);
\coordinate (2.66667, 0.94281, 1.63299) at (2.66667, 0.94281, 1.63299);
\coordinate (1.33333, 0.47140, 2.44949) at (1.33333, 0.47140, 2.44949);
\coordinate (1.33333, 1.88562, 0.00000) at (1.33333, 1.88562, 0.00000);
\coordinate (1.33333, 1.88562, 1.63299) at (1.33333, 1.88562, 1.63299);
\coordinate (2.66667, -0.47140, -0.81650) at (2.66667, -0.47140, -0.81650);
\coordinate (2.66667, -0.47140, 0.81650) at (2.66667, -0.47140, 0.81650);
\coordinate (2.66667, 0.94281, 0.00000) at (2.66667, 0.94281, 0.00000);
\draw[edge,back,postaction={decorate}] (0.00000, 1.41421, 0.81650) -- (0.00000, 0.00000, 0.00000);
\draw[edge,back,postaction={decorate}] (1.33333, 0.47140, -0.81650) -- (0.00000, 0.00000, 0.00000);
\draw[edge,back,postaction={decorate}] (0.00000, 1.41421, 2.44949) -- (0.00000, 1.41421, 0.81650);
\draw[edge,back,postaction={decorate}] (1.33333, 1.88562, 0.00000) -- (0.00000, 1.41421, 0.81650);
\draw[edge,back,postaction={decorate}] (1.33333, 1.88562, 0.00000) -- (1.33333, 0.47140, -0.81650);
\draw[edge,back,postaction={decorate}] (2.66667, -0.47140, -0.81650) -- (1.33333, 0.47140, -0.81650);
\draw[edge,back,postaction={decorate}] (1.33333, 1.88562, 1.63299) -- (1.33333, 1.88562, 0.00000);
\draw[edge,back,postaction={decorate}] (2.66667, 0.94281, 0.00000) -- (1.33333, 1.88562, 0.00000);
\node[vertex] at (1.33333, 0.47140, -0.81650)     {};
\node[vertex] at (1.33333, 1.88562, 0.00000)     {};
\node[vertex] at (0.00000, 1.41421, 0.81650)     {};
\fill[facet] (1.33333, -0.94281, 1.63299) -- (0.00000, 0.00000, 1.63299) -- (0.00000, 0.00000, 0.00000) -- (1.33333, -0.94281, 0.00000) -- cycle {};
\fill[facet] (1.33333, 0.47140, 2.44949) -- (0.00000, 1.41421, 2.44949) -- (0.00000, 0.00000, 1.63299) -- (1.33333, -0.94281, 1.63299) -- cycle {};
\fill[facet] (2.66667, -0.47140, 0.81650) -- (1.33333, -0.94281, 1.63299) -- (1.33333, -0.94281, 0.00000) -- (2.66667, -0.47140, -0.81650) -- cycle {};
\fill[facet] (2.66667, 0.94281, 1.63299) -- (2.66667, -0.47140, 0.81650) -- (1.33333, -0.94281, 1.63299) -- (1.33333, 0.47140, 2.44949) -- cycle {};
\fill[facet] (2.66667, 0.94281, 1.63299) -- (1.33333, 1.88562, 1.63299) -- (0.00000, 1.41421, 2.44949) -- (1.33333, 0.47140, 2.44949) -- cycle {};
\fill[facet] (2.66667, -0.47140, -0.81650) -- (2.66667, 0.94281, 0.00000) -- (2.66667, 0.94281, 1.63299) -- (2.66667, -0.47140, 0.81650) -- cycle {};
\draw[edge,postaction={decorate}] (0.00000, 0.00000, 1.63299) -- (0.00000, 0.00000, 0.00000);
\draw[edge,postaction={decorate}] (1.33333, -0.94281, 0.00000) -- (0.00000, 0.00000, 0.00000);
\draw[edge,postaction={decorate}] (0.00000, 1.41421, 2.44949) -- (0.00000, 0.00000, 1.63299);
\draw[edge,postaction={decorate}] (1.33333, -0.94281, 1.63299) -- (0.00000, 0.00000, 1.63299);
\draw[edge,postaction={decorate}] (1.33333, 0.47140, 2.44949) -- (0.00000, 1.41421, 2.44949);
\draw[edge,postaction={decorate}] (1.33333, 1.88562, 1.63299) -- (0.00000, 1.41421, 2.44949);
\draw[edge,postaction={decorate}] (1.33333, -0.94281, 1.63299) -- (1.33333, -0.94281, 0.00000);
\draw[edge,postaction={decorate}] (2.66667, -0.47140, -0.81650) -- (1.33333, -0.94281, 0.00000);
\draw[edge,postaction={decorate}] (1.33333, 0.47140, 2.44949) -- (1.33333, -0.94281, 1.63299);
\draw[edge,postaction={decorate}] (2.66667, -0.47140, 0.81650) -- (1.33333, -0.94281, 1.63299) ;
\draw[edge,postaction={decorate}] (2.66667, 0.94281, 1.63299) -- (1.33333, 0.47140, 2.44949);
\draw[edge,postaction={decorate}] (2.66667, 0.94281, 1.63299) -- (1.33333, 1.88562, 1.63299);
\draw[edge,postaction={decorate}] (2.66667, 0.94281, 1.63299) -- (2.66667, -0.47140, 0.81650);
\draw[edge,postaction={decorate}] (2.66667, 0.94281, 1.63299) -- (2.66667, 0.94281, 0.00000);
\draw[edge,postaction={decorate}] (2.66667, -0.47140, 0.81650) -- (2.66667, -0.47140, -0.81650);
\draw[edge,postaction={decorate}] (2.66667, 0.94281, 0.00000) -- (2.66667, -0.47140, -0.81650);
\node[vertex] at (0.00000, 0.00000, 0.00000)     {};
\node[vertex] at (0.00000, 0.00000, 1.63299)     {};
\node[vertex] at (0.00000, 1.41421, 2.44949)     {};
\node[vertex] at (1.33333, -0.94281, 0.00000)     {};
\node[vertex] at (1.33333, -0.94281, 1.63299)     {};
\node[vertex] at (2.66667, 0.94281, 1.63299)     {};
\node[vertex] at (1.33333, 0.47140, 2.44949)     {};
\node[vertex] at (1.33333, 1.88562, 1.63299)     {};
\node[vertex] at (2.66667, -0.47140, -0.81650)     {};
\node[vertex] at (2.66667, -0.47140, 0.81650)     {};
\node[vertex] at (2.66667, 0.94281, 0.00000)     {};
\end{tikzpicture}
	}
	\caption{Linear orientations of the graphs of the graphical zonotopes corresponding to the acyclic reorientation posets of \cref{fig:acyclicReorientationPosets}. The first is the classical permutahedron (which has been rescaled to fit the size of the others). The perspective is chosen so that the minimal reorientation appears at the bottom of the picture.}
	\label{fig:graphicalZonotopesIntro}
\end{figure}

\medskip
Assume now that~$D$ is skeletal.
As proved by N.~Reading~\cite{Reading-HopfAlgebras}, any congruence~$\equiv$ of the acyclic reorientation lattice~$\AR$ defines a quotient fan, obtained 
\begin{itemize}
\item either from the graphical fan of~$D$ by glueing regions corresponding to acyclic reorientations of~$D$ that belong to the same $\equiv$-class,
\item or from the shards associated to the join irreducibles of~$\AR$ uncontracted by~$\equiv$.
\end{itemize}
In \cref{subsec:graphicalZonotopeAssociahedraShardPolytopesQuotientopes}, we construct polytopal realizations of all quotient fans, mimicking the approaches of~\cite{PadrolPilaudRitter}.
Some of the resulting quotientopes are illustrated in \cref{fig:graphicalZonotopesAssociahedraIntro}.
The following statement is proved in~\cref{thm:MinkowskiSumAssociahedra,thm:MinkowskiSumShardPolytopes}.

\begin{theorem}
\label{thm:quotientopes}
When~$D$ is skeletal, the quotient fan of any congruence of the acyclic reorientation lattice~$\AR$ is the normal fan of
\begin{itemize}
\item a Minkowski sum of associahedra of~\cite{HohlwegLange}, and
\item a Minkowski sum of shard polytopes of~\cite{PadrolPilaudRitter}.
\end{itemize}
\end{theorem}

\enlargethispage{1cm}
We also conjecture that the quotient fan of any coherent congruence of~$\AR$ can be realized by deleting inequalities in the facet description of the graphical zonotope of~$D$, generalizing the classical constructions of the associahedra and permutreehedra~\cite{ShniderSternberg, Loday, HohlwegLange,PilaudPons-permutrees, AlbertinPilaudRitter}.
In fact, for the Cambrian congruences, the quotientope defined by this inequality description seems to always coincide with the quotientope described as a Minkowski sum of shard polytopes in \cref{thm:quotientopes}.
In this paper, we just give a simple proof of this statement for the sylvester congruence, which is illustrated in \cref{fig:graphicalZonotopesAssociahedraIntro}.
Note that this contruction fails for congruences beyond the coherent congruences, as already discussed in~\cite{AlbertinPilaudRitter} for congruences of the weak order.

\begin{figure}[!h]
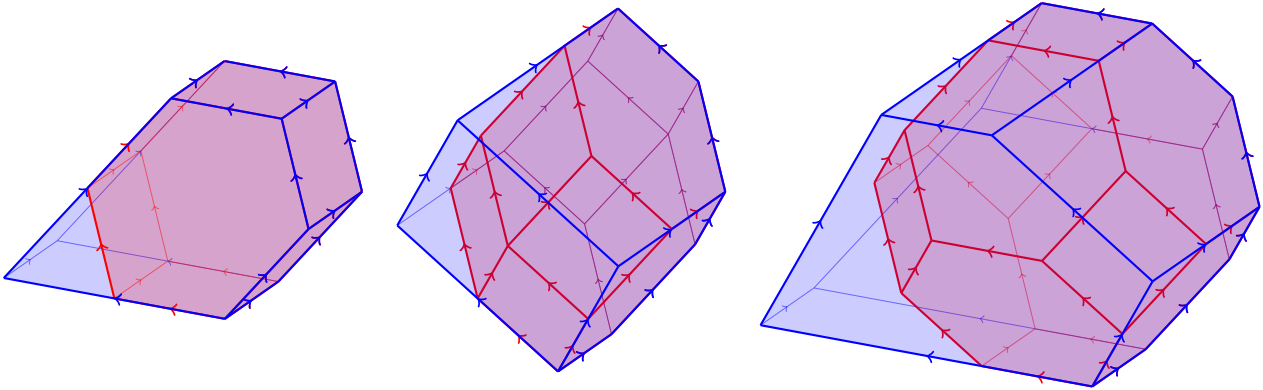

	\centerline{

\begin{tikzpicture}%
	[x={(-0.366215cm, -0.789554cm)},
	y={(0.235950cm, -0.590693cm)},
	z={(0.900119cm, -0.166391cm)},
	scale=1.000000,
	back/.style={very thin, opacity=0.5},
	edgeZono/.style={color=red, thick, decoration={markings, mark=at position 0.5 with {\arrow{>}}}},
	edgeAsso/.style={color=blue, thick, decoration={markings, mark=at position 0.5 with {\arrow{>}}}},
	facetZono/.style={fill=red,fill opacity=.2},
	facetAsso/.style={fill=blue,fill opacity=.2},
	vertex/.style={}]
%
%

\coordinate (0.00000, 0.00000, 0.00000) at (0.00000, 0.00000, 0.00000);
\coordinate (0.00000, 0.00000, 1.63299) at (0.00000, 0.00000, 1.63299);
\coordinate (1.33333, -0.94281, 0.00000) at (1.33333, -0.94281, 0.00000);
\coordinate (1.33333, -0.94281, 1.63299) at (1.33333, -0.94281, 1.63299);
\coordinate (4.00000, 0.00000, 1.63299) at (4.00000, 0.00000, 1.63299);
\coordinate (1.33333, 0.47140, 2.44949) at (1.33333, 0.47140, 2.44949);
\coordinate (4.00000, 0.00000, -1.63299) at (4.00000, 0.00000, -1.63299);
\coordinate (2.66667, -0.47140, 2.44949) at (2.66667, -0.47140, 2.44949);
\coordinate (2.66667, 0.94281, -1.63299) at (2.66667, 0.94281, -1.63299);
\coordinate (2.66667, 0.94281, 1.63299) at (2.66667, 0.94281, 1.63299);
\draw[edgeAsso,postaction={decorate},back] (2.66667, 0.94281, -1.63299) -- (0.00000, 0.00000, 0.00000);
\draw[edgeAsso,postaction={decorate},back] (4.00000, 0.00000, -1.63299) -- (2.66667, 0.94281, -1.63299);
\draw[edgeAsso,postaction={decorate},back] (2.66667, 0.94281, 1.63299) -- (2.66667, 0.94281, -1.63299);
\node[vertex] at (2.66667, 0.94281, -1.63299)     {};
\fill[facetAsso] (1.33333, -0.94281, 1.63299) -- (0.00000, 0.00000, 1.63299) -- (0.00000, 0.00000, 0.00000) -- (1.33333, -0.94281, 0.00000) -- cycle {};
\fill[facetAsso] (2.66667, -0.47140, 2.44949) -- (1.33333, -0.94281, 1.63299) -- (0.00000, 0.00000, 1.63299) -- (1.33333, 0.47140, 2.44949) -- cycle {};
\fill[facetAsso] (1.33333, 0.47140, 2.44949) -- (2.66667, 0.94281, 1.63299) -- (4.00000, 0.00000, 1.63299) -- (2.66667, -0.47140, 2.44949) -- cycle {};
\fill[facetAsso] (4.00000, 0.00000, 1.63299) -- (2.66667, -0.47140, 2.44949) -- (1.33333, -0.94281, 1.63299) -- (1.33333, -0.94281, 0.00000) -- (4.00000, 0.00000, -1.63299) -- cycle {};
\coordinate (0.00000, 0.00000, 0.00000) at (0.00000, 0.00000, 0.00000);
\coordinate (0.00000, 0.00000, 1.63299) at (0.00000, 0.00000, 1.63299);
\coordinate (1.33333, -0.94281, 0.00000) at (1.33333, -0.94281, 0.00000);
\coordinate (1.33333, -0.94281, 1.63299) at (1.33333, -0.94281, 1.63299);
\coordinate (1.33333, 0.47140, -0.81650) at (1.33333, 0.47140, -0.81650);
\coordinate (4.00000, 0.00000, 1.63299) at (4.00000, 0.00000, 1.63299);
\coordinate (1.33333, 0.47140, 2.44949) at (1.33333, 0.47140, 2.44949);
\coordinate (2.66667, -0.47140, -0.81650) at (2.66667, -0.47140, -0.81650);
\coordinate (4.00000, 0.00000, 0.00000) at (4.00000, 0.00000, 0.00000);
\coordinate (2.66667, -0.47140, 2.44949) at (2.66667, -0.47140, 2.44949);
\coordinate (2.66667, 0.94281, 0.00000) at (2.66667, 0.94281, 0.00000);
\coordinate (2.66667, 0.94281, 1.63299) at (2.66667, 0.94281, 1.63299);
\draw[edgeZono,postaction={decorate},back] (1.33333, 0.47140, -0.81650) -- (0.00000, 0.00000, 0.00000);
\draw[edgeZono,postaction={decorate},back] (2.66667, -0.47140, -0.81650) -- (1.33333, 0.47140, -0.81650);
\draw[edgeZono,postaction={decorate},back] (2.66667, 0.94281, 0.00000) -- (1.33333, 0.47140, -0.81650);
\draw[edgeZono,postaction={decorate},back] (4.00000, 0.00000, 0.00000) -- (2.66667, 0.94281, 0.00000);
\draw[edgeZono,postaction={decorate},back] (2.66667, 0.94281, 1.63299) -- (2.66667, 0.94281, 0.00000);
\node[vertex] at (1.33333, 0.47140, -0.81650)     {};
\node[vertex] at (2.66667, 0.94281, 0.00000)     {};
\fill[facetZono] (1.33333, -0.94281, 1.63299) -- (0.00000, 0.00000, 1.63299) -- (0.00000, 0.00000, 0.00000) -- (1.33333, -0.94281, 0.00000) -- cycle {};
\fill[facetZono] (2.66667, -0.47140, 2.44949) -- (1.33333, -0.94281, 1.63299) -- (0.00000, 0.00000, 1.63299) -- (1.33333, 0.47140, 2.44949) -- cycle {};
\fill[facetZono] (1.33333, 0.47140, 2.44949) -- (2.66667, 0.94281, 1.63299) -- (4.00000, 0.00000, 1.63299) -- (2.66667, -0.47140, 2.44949) -- cycle {};
\fill[facetZono] (4.00000, 0.00000, 1.63299) -- (2.66667, -0.47140, 2.44949) -- (1.33333, -0.94281, 1.63299) -- (1.33333, -0.94281, 0.00000) -- (2.66667, -0.47140, -0.81650) -- (4.00000, 0.00000, 0.00000) -- cycle {};
\draw[edgeZono,postaction={decorate}] (0.00000, 0.00000, 1.63299) -- (0.00000, 0.00000, 0.00000);
\draw[edgeZono,postaction={decorate}] (1.33333, -0.94281, 0.00000) -- (0.00000, 0.00000, 0.00000);
\draw[edgeZono,postaction={decorate}] (1.33333, -0.94281, 1.63299) -- (0.00000, 0.00000, 1.63299);
\draw[edgeZono,postaction={decorate}] (1.33333, 0.47140, 2.44949) -- (0.00000, 0.00000, 1.63299);
\draw[edgeZono,postaction={decorate}] (1.33333, -0.94281, 1.63299) -- (1.33333, -0.94281, 0.00000);
\draw[edgeZono,postaction={decorate}] (2.66667, -0.47140, -0.81650) -- (1.33333, -0.94281, 0.00000);
\draw[edgeZono,postaction={decorate}] (2.66667, -0.47140, 2.44949) -- (1.33333, -0.94281, 1.63299);
\draw[edgeZono,postaction={decorate}] (4.00000, 0.00000, 1.63299) -- (4.00000, 0.00000, 0.00000);
\draw[edgeZono,postaction={decorate}] (4.00000, 0.00000, 1.63299) -- (2.66667, -0.47140, 2.44949);
\draw[edgeZono,postaction={decorate}] (4.00000, 0.00000, 1.63299) -- (2.66667, 0.94281, 1.63299);
\draw[edgeZono,postaction={decorate}] (2.66667, -0.47140, 2.44949) -- (1.33333, 0.47140, 2.44949);
\draw[edgeZono,postaction={decorate}] (2.66667, 0.94281, 1.63299) -- (1.33333, 0.47140, 2.44949);
\draw[edgeZono,postaction={decorate}] (4.00000, 0.00000, 0.00000) -- (2.66667, -0.47140, -0.81650);
\node[vertex] at (0.00000, 0.00000, 0.00000)     {};
\node[vertex] at (0.00000, 0.00000, 1.63299)     {};
\node[vertex] at (1.33333, -0.94281, 0.00000)     {};
\node[vertex] at (1.33333, -0.94281, 1.63299)     {};
\node[vertex] at (4.00000, 0.00000, 1.63299)     {};
\node[vertex] at (1.33333, 0.47140, 2.44949)     {};
\node[vertex] at (2.66667, -0.47140, -0.81650)     {};
\node[vertex] at (4.00000, 0.00000, 0.00000)     {};
\node[vertex] at (2.66667, -0.47140, 2.44949)     {};
\node[vertex] at (2.66667, 0.94281, 1.63299)     {};
\draw[edgeAsso,postaction={decorate}] (0.00000, 0.00000, 1.63299) -- (0.00000, 0.00000, 0.00000);
\draw[edgeAsso,postaction={decorate}] (1.33333, -0.94281, 0.00000) -- (0.00000, 0.00000, 0.00000);
\draw[edgeAsso,postaction={decorate}] (1.33333, -0.94281, 1.63299) -- (0.00000, 0.00000, 1.63299);
\draw[edgeAsso,postaction={decorate}] (1.33333, 0.47140, 2.44949) -- (0.00000, 0.00000, 1.63299);
\draw[edgeAsso,postaction={decorate}] (1.33333, -0.94281, 1.63299) -- (1.33333, -0.94281, 0.00000);
\draw[edgeAsso,postaction={decorate}] (4.00000, 0.00000, -1.63299) -- (1.33333, -0.94281, 0.00000);
\draw[edgeAsso,postaction={decorate}] (2.66667, -0.47140, 2.44949) -- (1.33333, -0.94281, 1.63299);
\draw[edgeAsso,postaction={decorate}] (4.00000, 0.00000, 1.63299) -- (4.00000, 0.00000, -1.63299);
\draw[edgeAsso,postaction={decorate}] (4.00000, 0.00000, 1.63299) -- (2.66667, -0.47140, 2.44949);
\draw[edgeAsso,postaction={decorate}] (4.00000, 0.00000, 1.63299) -- (2.66667, 0.94281, 1.63299);
\draw[edgeAsso,postaction={decorate}] (2.66667, -0.47140, 2.44949) -- (1.33333, 0.47140, 2.44949);
\draw[edgeAsso,postaction={decorate}] (2.66667, 0.94281, 1.63299) -- (1.33333, 0.47140, 2.44949);
\node[vertex] at (0.00000, 0.00000, 0.00000)     {};
\node[vertex] at (0.00000, 0.00000, 1.63299)     {};
\node[vertex] at (1.33333, -0.94281, 0.00000)     {};
\node[vertex] at (1.33333, -0.94281, 1.63299)     {};
\node[vertex] at (4.00000, 0.00000, 1.63299)     {};
\node[vertex] at (1.33333, 0.47140, 2.44949)     {};
\node[vertex] at (4.00000, 0.00000, -1.63299)     {};
\node[vertex] at (2.66667, -0.47140, 2.44949)     {};
\node[vertex] at (2.66667, 0.94281, 1.63299)     {};
\end{tikzpicture}
			\raisebox{-.7cm}{\input{zonotopeAssociahedron4Intro}}
			\raisebox{-.9cm}{\input{zonotopeAssociahedron7Intro}}
	}
	\caption{The graphical zonotopes (red) and the associahedra (blue) for the acyclic reorientation lattices of \cref{fig:canonicalJoinComplexes}. All associahedra are obtained by deleting inequalities in the facet descriptions of the corresponding graphical zonotope. The perspective is chosen so that the minimal reorientation appears at the bottom of the picture.}
	\label{fig:graphicalZonotopesAssociahedraIntro}
\end{figure}


\subsection*{Posets of regions}

Finally, we want to discuss the connections of our results to the posets of regions of arbitrary hyperplane arrangements introduced by A.~Bj\"orner, P.~Edelman and G.~Ziegler in~\cite{Edelman, BjornerEdelmanZiegler}.
For a central hyperplane arrangement~$\c{H}$ and a base region~$\polytope{B}$ of~$\c{H}$, the \defn{poset of regions}~$\PR$ is the partial order on all regions of~$\c{H}$ defined by inclusion of the sets of hyperplanes separating each region from~$\polytope{B}$.
It was proved in~\cite{BjornerEdelmanZiegler} that
\begin{itemize}
\item the base region~$\polytope{B}$ is simplicial when the poset of regions~$\PR$ is a lattice,
\item the poset of regions~$\PR$ is a lattice when~$\c{H}$ is simplicial,
\item the poset of regions~$\PR$ is a lattice when~$\c{H}$ is supersolvable and~$\polytope{B}$ is a canonical base region of~$\c{H}$ in the sense of~\cite{BjornerEdelmanZiegler}.
\end{itemize}
Moreover, N.~Reading showed in~\cite{Reading-PosetRegionsChapter} that the poset of regions~$\PR$ is a congruence uniform lattice if and only if~$\c{H}$ is \defn{tight} with respect to~$\polytope{B}$, meaning that for each region~$\polytope{R}$ of~$\c{H}$, every pair of upper (resp.~lower) facets of~$\polytope{R}$ with respect to~$\polytope{B}$ intersects in a codimension~$2$ face.

\medskip
In view of these properties, it is relevant to characterize the directed acyclic graphs~$D$ whose graphical arrangements are simplicial, tight, or supersolvable.
Recall that a \defn{chord} of an undirected cycle~$C$ is an edge joining two non-consecutive vertices of~$C$.
An undirected graph~$G$ is \defn{chordal} (resp.~\defn{chordful}) if for any cycle~$C$ of length at least~$4$ contained in~$G$, at least one chord (resp.~all chords) of~$C$ also belongs to~$G$.
The directed graph~$D$ is chordal (resp.~chordful) if its underlying undirected graph is.
Note that chordful graphs are also known as block graphs in the literature.
Observe that chordful implies skeletal, and skeletal implies chordal, but none of the reverse directions holds.
For instance, any forest and any tournament is chordful, skeletal and chordal.
The first point of the next statement is proved in \cref{prop:simplicialGraphicalFan}, the second follows from~\cref{prop:congruenceUniform}, and the last was proved in~\cite[Coro.~4.10]{Stanley-IntroductionHyperplaneArrangements}.
It is illustrated in \cref{fig:simplicialTightSupersolvable}.

\begin{theorem}
The graphical arrangement of~$D$ is
\begin{itemize}
\item simplicial if and only if $D$ is chordful,
\item tight if and only if~$D$ is skeletal,
\item supersolvable if and only if $D$ is chordal.
\end{itemize}
\end{theorem}

\begin{figure}[h]
	\centerline{
		\begin{tabular}{c@{\quad}c@{\quad}c@{\quad}c} 
			simplicial &
			tight but not simplicial &
			supersolvable but not tight &
			not supersolvable
			\\
			\includegraphics[scale=.9]{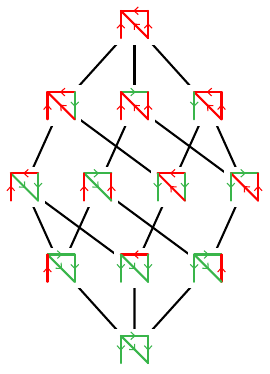} &
			\includegraphics[scale=.9]{acyclicReorientationPoset3} &
			\includegraphics[scale=.9]{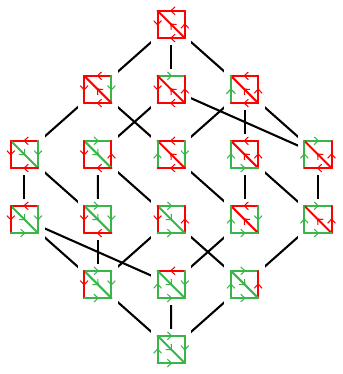} &
			\includegraphics[scale=.9]{acyclicReorientationPoset5}
		\end{tabular}
	}
	\caption{Some acyclic reorientation lattices. The first is simplicial, the second is not simplicial but tight, and the third is not tight but supersolvable, and the last is not supersolvable.}
	\vspace{-.4cm}
	\label{fig:simplicialTightSupersolvable}
\end{figure}

\enlargethispage{.5cm}
Conversely, it is natural to wonder to what extent the results of this paper can be transported to the poset of regions of arbitrary hyperplane arrangements.
In \cref{sec:posetRegions}, we translate the condition of \cref{thm:characterizationLattice} to natural equivalent geometric conditions on the hyperplane arrangement~$\c{H}$ and the base region~$\polytope{B}$.
We show that these conditions are necessary, but not sufficient, for the poset of regions~$\PR$ to be a lattice.


\subsection*{Open problems}

We close this overview by observing that the paper opens many combinatorial and geometric research directions.
We tried to underline some of the particularly puzzling questions in \cref{pb:regularCoverGraph,,pb:partialReorientationsAcyclic,,pb:TamariRegular,,pb:allCambrianRegular,,pb:allCambrianSameGraphs,,pb:descriptionElementPosetsCoherentCongruences,,pb:descriptionPartialReorientations,,pb:HamiltonianQuotients,,pb:allCambrianSameFaceLattices,,pb:deformationConesQuotientopes,,pb:vertexDescriptionAssociahedron,,pb:CambrianAssociahedraRemovahedra,,pb:coherentQuotientopesRemovahedra,,pb:coherentQuotientopesRemovahedra}.

\addtocontents{toc}{ \vspace{.1cm} }
\addtocontents{toc}{\protect\setcounter{tocdepth}{2}}


\section{Characterization of acyclic reorientation lattices}
\label{sec:characterizationLattice}

In this section, we show \cref{thm:characterizationLattice} and provide a characterization of the sets of arcs reversed in the acyclic reorientations of~$D$ and explicit formulas for the join and meet operations in the case where the acyclic reorientation poset~$\AR$ is a lattice.

\medskip
We start with an obvious necessary condition for~$\AR$ to be a lattice.
We give a self-contained proof although it is just a specialization of~\cite[Thm.~3.1]{BjornerEdelmanZiegler}.

\begin{lemma}
\label{lem:transitiveReductionForest}
If $\AR$ is a lattice, then the transitive reduction of~$D$ is a forest.
\end{lemma}

\begin{proof}
Assume that the transitive reduction of~$D$ contains a (undirected) cycle~$C$.
Choose an arbitrary orientation on~$C$, and let~$F$ denote the forward arcs along~$C$ and~$B$ denote the backward arcs along~$C$.
For~$f \in F$, denote by~$D_f$ the acyclic reorientation of~$D$ obtained by reversing~$f$ (it is indeed acyclic since~$f$ belongs to the transitive reduction of~$D$).
For~$b \in B$, denote by~$\bar D_b$ the acyclic reorientation of~$D$ obtained by reversing all arcs but~$b$ (it is indeed acyclic since~$b$ belongs to the transitive reduction of~$D$).
Note that~$D_f \le \bar D_b$ for any~$f \in F$ and~$b \in B$.
Consider now any reorientation~$E$ of~$D$ such that~$D_f \le E \le \bar D_b$ for all~$f \in F$ and~$b \in B$.
Then all arcs in~$F$ are reversed in~$E$ (because $D_f \le E$ for all~$f \in F$) while none of the arcs in~$B$ are reversed in~$E$ (because~$E \le \bar D_b$ for all~$b \in B$).
It follows that $C$ is a directed cycle in~$E$, so that~$\set{D_f}{f \in F}$ has no join (and $\set{\bar D_b}{b \in B}$ has no meet) in~$\AR$.
\end{proof}

\begin{corollary}
\label{coro:vertebrate}
If $\AR$ is a lattice, then $D$ is vertebrate.
\end{corollary}

\begin{proof}
Fix a subset~$U \subseteq V$ and let~$D_U$ denote the directed subgraph of~$D$ induced by~$U$ and~$D^U$ denote the directed acyclic graph obtained from~$D$ by deleting all arcs joining two vertices in~$U$.
Fix an acyclic reorientation~$E$ of~$D^U$ in which all arcs incident to~$U$ are pointing towards~$U$.
Then the set of acyclic reorientations of~$D$ that agree with~$E$ on~$D^U$ is an interval of the acyclic reorientation poset isomorphic to the acyclic reorientation poset of~$D_U$.
Since an interval of a lattice is a lattice, it follows that the transitive reduction of~$D_U$ is a forest by \cref{lem:transitiveReductionForest}.
\end{proof}

We now assume that $D$ is vertebrate and we will show that the acyclic reorientation poset~$\AR$ is a lattice, and describe the join and meet operations.

\medskip
It is classical that a subset~$B$ of~$\smash{\binom{[n]}{2}}$ is the inversion set of a permutation of~$[n]$ if and only if both~$B$ and~$\smash{\binom{[n]}{2}} \ssm B$ are transitive.
This generalizes to the following characterization of the reversed sets of the acyclic reorientations of~$D$.
We say that a subset~$B \subseteq A$ of arcs of~$D$ is 
\begin{itemize}
\item \defn{closed} if all arcs of~$A$ in the transitive closure of~$B$ belong to~$B$, 
\item \defn{coclosed} if its complement~$A \ssm B$ is closed, and 
\item \defn{biclosed} if it is both closed and coclosed.
\end{itemize}

\begin{proposition}
\label{prop:biclosed}
If~$D$ is vertebrate, a subset~$B$ of~$A$ is biclosed if and only if its reorientation~is~acyclic.
\end{proposition}

\begin{proof}
Consider the reorientation~$E$ of~$D$ obtained by reversing the arcs of~$B$.

\medskip
If~$B$ is not closed, then $A$ contains an arc in the transitive closure of~$B$ but not in~$B$.
This arc together with (the reverse of) any path in~$B$ joining its endpoints clearly forms a directed cycle in~$E$.
By symmetry, we conclude that if $E$ is acyclic, then~$B$ is biclosed.

\medskip
Conversely, if~$E$ is not acyclic, then it contains a directed cycle~$C$ with vertex set~$U$.
As any chord in a directed cycle defines a smaller directed cycle, we can assume that~$C$ is induced.
As the subgraph of~$D$ induced by~$U$ is a (not necessarily directed) cycle, its transitive reduction can only be a path by assumption on~$D$.
In other words, there exists an arc~$c$ of~$C$ such that either~$c$ is reversed while $C \ssm \{c\}$ is not, or~$C \ssm \{c\}$ is reversed while~$c$ is not.
This ensures that~$B$ is not biclosed, as it is not coclosed in the former case, and not closed in the later case.
\end{proof}

Note that when $D$ is not vertebrate, any set whose reorientation is acyclic is still biclosed, but the converse fails.
For instance, in the last example of \cref{fig:acyclicReorientationPosets}, each of the two directed paths from the source to the sink of~$D$ forms a biclosed set whose reorientation is not acyclic.

\medskip
With \cref{prop:biclosed} at hand, we are now ready to show a refined version of the non-trivial direction of \cref{thm:characterizationLattice}.
For the weak order on permutations, it is well-known that, for any permutations~$\pi_1, \dots, \pi_k$ of~$[n]$, the inversion set of $\pi_1 \join \dots \join \pi_k$ (resp.~of $\pi_1 \meet \dots \meet \pi_k$) is the transitive closure (resp.~the complement of the transitive closure) of the inversion sets (resp.~of the complements of the inversion sets) of~$\pi_1, \dots, \pi_k$.
This generalizes for vertebrate directed acyclic graphs as follows.

\begin{theorem}
\label{thm:joinMeet}
If $D$ is vertebrate, then the acyclic reorientation poset is a lattice, where the join (resp.~meet) of a set of acyclic reorientations~$E_1, \dots, E_k$ of~$D$ is obtained by reversing all arcs of~$A$ that belong (resp.~do not belong) to the transitive closure of the arcs reversed (resp.~not reversed) in at least one of the reorientations~$E_1, \dots, E_k$.
\end{theorem}

\begin{proof}
Note that it suffices to prove the statement for the join since the acyclic reorientation poset is self-dual under reversing all arcs.

\medskip
Let~$B$ denote the transitive closure of the arcs reversed in at least one of the reorientations~$E_1, \dots, E_k$.
It is clearly closed, let us show that it is as well coclosed.
Assume by means of contradiction that $A$ contains an arc~$a'$ which is in the transitive closure of~$A \ssm B$ and in~$B$.
By definition, the endpoints~$u$ and~$v$ of~$a'$ are therefore connected by
\begin{itemize}
\item a directed path~$\pi = a_1, \dots, a_\ell$ of arcs in~$A \ssm B$, and
\item a directed path~$\pi' = a'_1, \dots, a'_{\ell'}$ of arcs reversed in at least one of the reorientations~$E_1, \dots, E_k$.
\end{itemize}
Note that we have both~$\ell > 1$ since~$B$ is closed, and~$\ell' > 1$ since all~$E_1, \dots, E_k$ are acyclic.
Moreover, we can assume without loss of generality that~$\ell+\ell'$ is minimal among all pairs of such paths sharing their endpoints.
This minimality assumption implies that
\begin{itemize}
\item these two paths do not share inner vertices, and
\item there is no arc from an inner vertex of one path to an inner vertex of the other path.
\end{itemize}
It follows that all arcs of~$\pi$ and~$\pi'$ belong to the transitive reduction of the restriction of~$D$ to the union of the vertex sets of~$\pi$ and~$\pi'$.
This contradicts our assumption on~$D$.

\medskip
We conclude that~$B$ is biclosed, and it is by definition the smallest biclosed subset of~$A$ containing all arcs reversed in at least one of the reorientations~$E_1, \dots, E_k$.
By \cref{prop:biclosed}, we conclude that the reorientation~$E$ obtained by reversing~$B$ is the join of the reorientations~$E_1, \dots, E_k$.
\end{proof}

\begin{proof}[Proof of \cref{thm:characterizationLattice}]
One direction is given by \cref{coro:vertebrate}, the other by \cref{thm:joinMeet}.
\end{proof}

An alternative proof will follow later from \cref{prop:congruenceNormal}.
The advantage of the proof of this section is that it provides explicit descriptions of the join and meet operations in the acyclic reorientation lattices.

\medskip
Note that assuming \cref{thm:characterizationLattice}, the characterization of \cref{prop:biclosed} and the description of the join and the meet operations of \cref{thm:joinMeet} can be seen as specializations of~\cite[Sec.~5]{BjornerEdelmanZiegler}.
Here, we used them to establish the characterization of \cref{thm:characterizationLattice}.


\section{Restriction maps}
\label{sec:restrictionMaps}

Consider now two directed acyclic graphs~$D \eqdef (V, A)$ and~$D' \eqdef (V, A')$ on the same vertex set with~${A \supseteq A'}$.
At the moment, we do not require that~$D$ and $D'$ be vertebrate.
We consider the restriction map~$\restrictionMap : \AR \to \AR[D']$ from acyclic reorientations of~$D$ to acyclic reorientations of~$D'$, that we simply denote by~$\restrictionMap[]$ throughout this section as there is no ambiguity.
Four different restriction maps are illustrated in \cref{fig:restrictionMap}.
We start by an elementary observation.

\begin{lemma}
\label{lem:restrictionMap}
The restriction map~$\restrictionMap[]$ is surjective and order preserving.
\end{lemma}

\begin{proof}
Consider an acyclic reorientation~$E'$ of~$D'$.
Since~$E'$ is acyclic, there exists a total order~$\prec$ on~$V$ such that all arcs of~$E'$ are increasing for~$\prec$.
It defines an acyclic reorientation~$E$ of~$D$ where all arcs are increasing for~$\prec$.
Clearly, we have $\restrictionMap[](E) = E'$.
This proves that $\restrictionMap[]$ is surjective.

\medskip
Observe now that for an acyclic reorientation~$E$ of~$D$, the arcs reversed in~$\restrictionMap[](E)$ are the arcs reversed in~$E$ that belong to~$D'$.
Since the order among acyclic reorientations is defined by the inclusion of sets of reversed arcs, this immediately implies that $\restrictionMap[]$ is order preserving.
\end{proof}

We now consider the fibers of~$\restrictionMap[]$.
It immediately follows from \cref{lem:restrictionMap} that each fiber~$F$ is order convex (\ie $x \le y \le z$ and~$x,z \in F$ implies~$y \in F$), but they might fail to be intervals as illustrated in \cref{fig:restrictionMap}.
We now characterize the acyclic reorientations of~$D'$ whose fibers under the restriction map~$\restrictionMap[]$ admit a minimal or maximal element.
A classical result of A.~Bj\"orner and M.~Wachs~\cite{BjornerWachs} states that the set of linear extensions of a poset~$\prec$ on~$[n]$ admits a minimal (resp.~maximal) element under the weak order on permutations if and only if~$i \succ k$ implies $i \succ j$ or $j \succ k$ (resp.~$i \prec k$ implies $i \prec j$ or $j \prec k$) for any~$1 \le i < j < k \le n$.
This generalizes as follows.

\begin{proposition}
\label{prop:minimalMaximalInFiber}
Consider an acyclic reorientation~$E'$ of~$D'$ and let~$E$ denote the reorientation of~$D$ where an arc~$(u,v)$ is reversed (resp.~not reversed) if there is a directed path in~$E'$ joining~$v$ to~$u$ (resp.~$u$ to~$v$).
Then the following assertions are equivalent:
\begin{enumerate}[(i)]
\item the reorientation~$E$ is acyclic,
\item the fiber of~$E'$ under the restriction map~$\restrictionMap[]$ admits a minimal (resp.~maximal) element (then, this element is~$E$),
\item any directed cycle formed by arcs of~$E'$ and of~$D \ssm D'$ contains at least one arc~$(u,v)$ of~$D \ssm D'$ such that there is a directed path in~$E'$ joining~$v$ to~$u$ (resp~$u$ to~$v$).
\end{enumerate}
\end{proposition}

\begin{proof}
Note that it suffices to prove the result for minimal elements since the acyclic reorientation poset is self-dual under reversing all arcs.

\medskip
\noindent
\uline{(i) $\Rightarrow$ (ii).}
Observe that $E$ agrees with~$E'$ on~$D'$ and that all arcs reversed in~$E$ are reversed in any acyclic reorientation in the fiber of~$E'$.
Therefore, if~$E$ is acyclic, it is the minimal element of the fiber of~$E'$ under~$\restrictionMap[]$.

\medskip
\noindent
\uline{(ii) $\Rightarrow$ (iii).}
Suppose that the fiber of~$E'$ under~$\restrictionMap[]$ admits a minimal element~$M$.
Consider a directed cycle~$C$ formed by arcs of~$E'$ and of~$D \ssm D'$.
Let~$a \eqdef (u,v)$ be an arc of~$C$ which belongs to~$D \ssm D'$ so that there is no directed path in~$E'$ joining~$v$ to~$u$.
Consider the reorientation~$E'_a$ of the directed acyclic graph~$D'_a \eqdef (V, A' \cup\{a\})$ that agrees with~$E'$ on~$A'$ and where~$a$ is not reversed.
Since $E'$ is acyclic and there is no directed path in~$E'$ joining~$v$ to~$u$, the reorientation~$E'_a$ of~$D'_a$ is acyclic, so that it can be completed into an acyclic reorientation~$E_a$ of~$D$ by \cref{lem:restrictionMap}.
By definition, we have $\restrictionMap[](E_a) = E'$ and $a$ is not reversed in~$E_a$.
Since~$M$ is the minimal element of the fiber of~$E'$ under~$\restrictionMap[]$, we have~$M \le E_a$, so that the arc~$a$ is not reversed in~$M$.
Since~$M$ is acyclic, $C$ contains at least one arc~$(u,v)$ of~$D \ssm D'$ such that there is a directed path in~$E'$ joining~$v$~to~$u$.

\medskip
\noindent
\uline{(iii) $\Rightarrow$ (i).}
Assume that~$E$ contains a cycle~$C$.
Up to replacing each reversed arc of~$C$ by a directed path in~$E'$ joining its endpoints, we can assume that all arcs of~$C$ belong to~$E'$ or to~$D \ssm D'$.
Each arc~$(u,v)$ of~$C$ in~$D \ssm D'$ is a non-reversed arc of~$E$ so that there is no directed path in~$E'$ joining~$v$ to~$u$.
Therefore, $E'$ does not fulfill (iii).
\end{proof}

Conversely, observe that any interval can be seen as the fiber of a well-chosen restriction map.
For two acyclic reorientations~$E$ and~$F$ of~$D$, we denote by~$E \cap F$ the directed acyclic graph whose arcs are the common arcs of~$E$ and~$F$.

\begin{proposition}
\label{prop:intervalsAreFibers}
Any interval~$[E^\join, E^\meet] \eqdef \set{E \in \AR}{E^\join \le E \le E^\meet}$ of~$\AR$ is the fiber of~${E^\join \cap E^\meet}$ (resp.~of the transitive reduction of~$E^\join \cap E^\meet$) under the restriction~map to the edges of~$D$ that appear in any direcction in~${E^\join \cap E^\meet}$ (resp.~in the transitive reduction of~$E^\join \cap E^\meet$)
\end{proposition}

\begin{proof}
Observe that an arc is reversed in~$E^\join$ (resp.~unreversed in~$E^\meet$) if and only if it is reversed (resp.~unreversed) in all~$E \in [E^\join, E^\meet]$ if and only if it belongs to and is reversed (resp.~unreversed) in~$E^\join \cap E^\meet$.
The result follows for~$E^\join \cap E^\meet$.
It also holds for the transitive reduction of~$E^\join \cap E^\meet$ since the fiber of an acyclic reorientation and the fiber of its transitive reduction always~coincide.
\end{proof}

In the next statements, we say that $D'$ is 
\begin{itemize}
\item \defn{weakly balanced} in~$D$ if for any simple cycle~$C$ in~$D$, if all backward arcs along~$C$ belong to~$D'$, then either all or all but one forward arcs along~$C$ belong to~$D'$,
\item \defn{balanced} in~$D$ if for any simple cycle~$C$ in~$D$, if all backward arcs along~$C$ belong to~$D'$ and~$C$ has at least two forward arcs, then all forward arcs along~$C$ also belong to~$D'$,
\item \defn{strongly balanced} in~$D$ if for any simple cycle~$C$ in~$D$, if all backward arcs along~$C$ belong to~$D'$, then all forwards arcs along~$C$ also belong to~$D'$.
\end{itemize}
Note that strongly balanced implies balanced and balanced implies weakly balanced, but both reverse implications are wrong.

\medskip
We now characterize the subgraphs~$D'$ for which all fibers under the restriction map~$\restrictionMap[]$ are intervals.

\begin{proposition}
\label{prop:intervalFibers}
The fibers of~$\restrictionMap[]$ are all intervals if and only if $D'$ is weakly balanced in~$D$.
\end{proposition}

\begin{proof}
Note that since the acyclic reorientation poset is self-dual under reversing all arcs, all fibers of~$\restrictionMap[]$ are intervals if and only if all fibers of~$\restrictionMap[]$ admit a minimal element.
We thus focus on minimal elements below.

\medskip
Assume that there is a simple cycle~$C$ in~$D$ with all backward arcs in~$D'$, but with two forward arcs~$a$ and~$b$ not in~$D'$.
By \cref{lem:restrictionMap}, there exists an acyclic reorientation~$E'$ of~$D'$ where all backward arcs along~$C$ are reversed, none of the forward arcs along~$C$ are reversed, and all other arcs incident to~$C$ are pointing toward~$C$.
The cycle~$C$ is formed by arcs of~$E'$ and of~$D \ssm D'$ and contains no arc~$(u,v)$ of~$D \ssm D'$ such that there is a directed path in~$E'$ joining~$v$ to~$u$ (because $a$ and $b$ are both in~$D \ssm D'$, and all arcs in~$E'$ incident to~$C$ are pointing toward~$C$).
We conclude by \cref{prop:minimalMaximalInFiber} that the fiber of~$E'$ under~$\restrictionMap[]$ has no minimal element.

\medskip
Conversely, assume that there is an acyclic reorientation~$E'$ of~$D'$ whose fiber under~$\restrictionMap[]$ has no minimal element.
By \cref{prop:minimalMaximalInFiber}, there is a directed cycle~$C$ formed by arcs of~$E'$ and of~$D \ssm D'$ which contains no arc~$(u,v)$ of~$D \ssm D'$ such that there is a directed path in~$E'$ joining~$v$ to~$u$
The backward arcs along~$C$ all belong to~$D'$ (since they do not belong to~$D \ssm D'$), and we claim that~$C$ contains at least two arcs of~$D \ssm D'$.
Indeed,
\begin{itemize}
\item if $C$ contains no arc in~$D \ssm D'$, then $C$ is a directed cycle in~$E'$, contradicting the acyclicity~of~$E'$,
\item if~$C$ contains only one arc~$a \eqdef (u,v)$ in~$D \ssm D'$, then $C \ssm \{a\}$ forms a directed path in~$E'$ joining~$v$ to~$u$, contradicting our assumption on~$C$.
\end{itemize}
We conclude that~$C$ is a simple cycle with all backward arcs in~$D'$ and at least two forward arcs not in~$D'$.
\end{proof}

Assume from now on that~$D'$ is weakly balanced in~$D$.
We denote by~$\projDown[]$ (resp.~$\projUp[]$) the map from~$\AR$ to~$\AR$ sending an acyclic reorientation~$E$ to the minimal (resp.~maximal) acyclic reorientation~$F$ such that~$\restrictionMap[](E) = \restrictionMap[](F)$.

\begin{proposition}
\label{prop:orderPreservingDownUpProj}
The maps~$\projDown[]$ and~$\projUp[]$ are order preserving if and only if $D'$ is balanced in~$D$.
\end{proposition}

\begin{proof}
Note that it suffices to prove the statement for~$\projDown[]$ since the acyclic reorientation poset is self-dual under reversing all arcs.

\medskip
Assume that there is a simple cycle~$C$ in~$D$ with all backward arcs in~$D'$ and at least two forward arcs~$a$ in~$D'$ and~$b$ not in~$D'$.
Note that since~$D'$ is weakly balanced in~$D$, all forward arcs along~$C$ except~$b$ belong to~$D'$.
By \cref{lem:restrictionMap}, there exists an acyclic reorientation~$E$ of~$D$ where all backward arcs along~$C$ are reversed, none of the forward arcs along~$C$ are reversed except~$b$, and all other arcs incident to~$C$ are pointing toward~$C$.
Let $F$ be the acyclic reorientation of~$D$ obtained by reversing~$a$ (it is indeed acyclic as $C$ is not a directed cycle in~$F$, and all arcs incident to~$C$ are pointing toward~$C$).
By \cref{prop:minimalMaximalInFiber}, $b$ is reversed in~$\projDown[](E)$ but not in~$\projDown[](F)$.
We conclude that~$\projDown[]$ is not order preserving, since~$E \le F$ by construction, while $\projDown[](E) \not\le \projDown[](F)$ because of~$b$.

\medskip
Conversely, assume that $D'$ is balanced in~$D$ and consider two acyclic reorientations~$E$ and~$F$ of~$D$ such that~$E \le F$.
Denoting~$E' \eqdef \restrictionMap[](E)$ and~$F' \eqdef \restrictionMap[](F)$, we have~$E' \le F'$ since $\restrictionMap[]$ is order preserving by \cref{lem:restrictionMap}.
Consider an arc~$(u,v)$ reversed in~$\projDown[](E)$.
If $(u,v)$ belongs to~$D'$, then it is reversed in~$\restrictionMap[](\projDown[](E)) = E'$, therefore as $E' \le F'$, it is reversed in $F' = \restrictionMap[](\projDown[](F))$ and thus in~$\projDown[](F)$.
If $(u,v)$ does not belong to~$D'$, then there is a directed path~$\pi$ joining~$v$ to~$u$ in~$E'$ by \cref{prop:minimalMaximalInFiber}.
Moreover, since $D'$ is balanced in~$D$ and~$(u,v) \notin D'$, all arcs along~$\pi$ are reversed in~$E'$.
Since~$E' \le F'$, all arcs along~$\pi$ are also reversed in~$F'$, so that~$(u,v)$ is also reversed in~$F$ by \cref{prop:minimalMaximalInFiber}.
We conclude that all arcs reversed in~$\projDown[](E)$ are also reversed in~$\projDown[](F)$, so that~$\projDown[](E) \le \projDown[](F)$.
We conclude that~$\projDown[]$ is order preserving.
\end{proof}

We now characterize the subgraphs~$D'$ for which~$\AR[D']$ can be seen as a lower (or upper) interval of~$\AR$, \ie of the form~$[D,E]$ (resp~$[E,\bar D]$) for some~$E \in \AR$.
This will be useful when studying the congruences of congruence uniform acyclic reorientation lattices in \cref{sec:congruences}.

\begin{proposition}
\label{prop:posetIsomInterval}
The map~$\restrictionMap[]$ restricts to a poset isomorphism from a lower (or upper) interval of~$\AR$ to~$\AR[D']$ if and only if~$D'$ is strongly balanced in~$D$.
\end{proposition}

\begin{proof}
Assume first that~$D'$ is strongly balanced in~$D$.
Let~$\inverseRestrictionMap[] : \AR[D'] \to \AR[D]$ denote the map sending an acyclic reorientation~$E'$ of~$D'$ to the acyclic reorientation of~$D$ whose reversed arcs are precisely the reversed arcs of~$D'$ (it is indeed acyclic, otherwise it would contain a simple cycle whose backward arcs all belong to~$D'$ and whose forward arcs cannot all belong to~$D'$ by acyclicity of~$E'$, contradicting the assumption on~$D'$).
It is clear that~$\restrictionMap[]$ and~$\inverseRestrictionMap[]$ are inverse poset isomorphisms from the lower interval~$[D, \inverseRestrictionMap[](\bar D')]$ of~$\AR$ to~$\AR[D']$.

\medskip
Conversely, assume that~$\restrictionMap[]$ restricts to a poset isomorphism from some lower interval~$I$ of~$\AR$ to $\AR[D']$.
Assume that some arc~$a$ of~$D \ssm D'$ is reversed in an acyclic reorientation~$E$ of~$I$.
Consider a saturated chain~$D = F_0 \le \dots \le F_p = E$ in~$\AR$. 
There is~$i \in [p]$ such that the arc~$a$ is flipped from~$F_{i-1}$ to~$F_i$.
Since~$a \notin D'$, we obtain~$\restrictionMap[](F_{i-1}) = \restrictionMap[](F_i)$ while~$F_{i-1}$ and~$F_i$ both belong to~$I$, contradicting our assumption on~$\restrictionMap[]$.
We conclude that no arc of~$D \ssm D'$ can be reversed in an acyclic reorientation of~$E$.
Assume now that there is a cycle~$C$ in~$D$ such that all backward arcs, but not all forward arcs, along~$C$ belong to~$D'$.
Let~$E'$ be an acyclic reorientation of~$D'$ that agrees with~$C$ on~$D'$ (it indeed exists by \cref{lem:restrictionMap} since~$C$ is not completely in~$D'$).
Then any acyclic reorientation~$E$ of~$D$ in the fiber of~$E'$ under~$\restrictionMap[]$ must have at least one arc of~$C \ssm D'$ reversed.
Therefore, the fiber of~$E'$ cannot meet~$I$, a contradiction.
\end{proof}

We are finally ready to prove \cref{thm:latticeMap}, as a specialization of \cref{prop:intervalFibers,prop:orderPreservingDownUpProj,prop:posetIsomInterval} in the case when~$D$ and~$D'$ are vertebrate.
Recall that a map~$\restrictionMap[] : L \to L'$ between two lattices~${(L, \le, \meet, \join)}$ and~${(L', \le', \meet', \join')}$ is a \defn{lattice map} if it respects the join and meet operations, that is ${\restrictionMap[](x \meet y) = \restrictionMap[](x) \meet' \restrictionMap[](y)}$ and~${\restrictionMap[](x \join y) = \restrictionMap[](x) \join' \restrictionMap[](y)}$ for all~$x, y \in L$.
When it is surjective, it is \defn{lattice quotient map}, and $L'$ is a \defn{lattice quotient} of~$L$.
The following characterization of lattice maps is classical.

\begin{proposition}
\label{prop:latticeQuotientMap}
A map~$\restrictionMap[] : L \to L'$ is a lattice map if and only if
\begin{itemize}
\item the fibers of~$\restrictionMap[]$ are intervals of~$L$, and
\item the map~$\projDown$ (resp.~$\projUp$) that send an element~$x$ of~$L$ to the minimal (resp.~maximal) element~$y$ with~$\restrictionMap[](x) = \restrictionMap[](y)$ is order preserving.
\end{itemize}
\end{proposition}

\begin{proof}[Proof of \cref{thm:latticeMap}]
It follows from \cref{prop:intervalFibers,prop:orderPreservingDownUpProj,prop:latticeQuotientMap,prop:posetIsomInterval}, and the immediate observation that the (weakly / strongly) balanced condition is equivalent to the (weakly / strongly) pathful condition when~$D$ is vertebrate.
\end{proof}

\begin{example}
Assume that~$D'$ is a forest.
As already mentioned in the introduction, the acyclic reorientation poset~$\AR[D']$ is then a boolean lattice.
The restriction map~$\restrictionMap[]$ is a lattice map if and only if~$D$ is vertebrate and~$D'$ is a subgraph of the transitive reduction of~$D$.
Therefore, for any vertebrate directed acyclic graph~$D$, any subgraph~$D'$ of the transitive reduction of~$D$ defines a boolean lattice quotient~$\AR[D']$ of~$\AR$.
\end{example}

\begin{example}
\label{exm:latticeQuotientsWeakOrder}
Assume that~$D$ is a tournament, and label the vertices of~$D$ by~$[n]$ so that~${(i,j) \in D}$~for all~$1 \le i < j \le n$.
As already mentioned in the introduction, the acyclic reorientation poset~$\AR$ is then isomorphic to the classical weak order on permutations.
The restriction map $\restrictionMap[]$ is a lattice map (in other words, $\AR[D']$ is a lattice quotient of the weak order) if and only if $(i,\ell) \in D'$ implies $(j,k) \in D'$ for any~$1 \le i \le j < k \le \ell \le n$.
In other words, $D'$ is a lower ideal for the nesting order defined by $(i,\ell) < (j,k)$ for~$1 \le i \le j < k \le \ell \le n$.
Representing this ideal by its generators, we obtain a bijection between acyclic reorientation posets that are lattice quotients of the weak order on~$\f{S}_n$ and non-nested partitions of~$[n]$, which are counted by the Catalan number~$C_n \eqdef \frac{1}{n+1} \binom{2n}{n}$.

\medskip
\enlargethispage{.5cm}
Note that the same graphs already appeared in the work of E.~Bernard and T.~McConville~\cite{BarnardMcConville} concerning lattice maps in the context of graph associahedra.
In particular, when~$D'$ is a lower ideal of the nesting order, there is a triangle of lattice morphisms from the weak order, through the tubing order on~$D'$ \cite{CarrDevadoss, BarnardMcConville}, to the acyclic reorientation lattice of~$D'$.
\end{example}


\section{Properties of acyclic reorientation lattices}
\label{sec:latticeProperties}

In this section, we assume that~$D$ is vertebrate and we study classical lattice properties of the acyclic reorientation lattice~$\AR$, illustrated in \cref{fig:distributiveSemidistributive}.
We refer to~\cite{GratzerWehrung-LatticeTheoryI, GratzerWehrung-LatticeTheoryII} for a detailed reference on these lattice properties and just briefly recall the needed definitions and characterizations of these properties.


\subsection{Join and meet irreducibles}
\label{subsec:joinMeetIrreducibles}

Recall first that an element~$x$ of a lattice~$L$ is \defn{join} (resp.~\defn{meet}) \defn{irreducible} if it covers (resp.~is covered by) a unique element of~$L$ denoted~$x_\star$ (resp.~$x^\star$).
For instance, the join (resp.~meet) irreducibles of the boolean lattice are the singletons (resp.~complements of singletons), and the join (resp.~meet) irreducibles in the weak order on permutations are the permutations with a single descent (resp.~ascent).
These examples generalize as follows.

\begin{proposition}
\label{prop:joinIrreducibles}
The following assertions are equivalent for an acyclic reorientation~$E$ of~$D$:
\begin{enumerate}[(i)]
\item $E$ is join (resp.~meet) irreducible in~$\AR$,
\item the transitive reduction of~$E$ contains a single reversed (resp.~not reversed) arc,
\item there is an arc~$a$ of~$D$ such that~$E$ is a minimal (resp.~maximal) element of the fiber of the reverse of~$a$ (resp.~of~$a$) under the restriction map~$\restrictionMap[a]$ from~$D$ to~$\{a\}$.
\end{enumerate}
\end{proposition}

\begin{proof}
Note that it suffices to prove the statement for join irreducibles since the acyclic reorientation poset is self-dual under reversing all arcs.

\medskip
\noindent
\uline{(i) $\Leftrightarrow$ (ii).}
We already mentioned in the introduction that an arc is flippable in~$E$ if and only if it belongs to the transitive reduction of~$E$.
Therefore, $E$ is join irreducible if and only if exactly one such arc is reversed.

\medskip
\noindent
\uline{(ii) $\Leftrightarrow$ (iii).}
Flipping any arc~$b$ distinct from~$a$ in the transitive reduction of~$E$ yields an acyclic reorientation~$F$ of~$D$ in the same fiber under~$\restrictionMap[a]$, with $E \ge F$ if and only if~$b$ is reversed in~$E$.
Therefore, $a$ is the only reversed arc in the transitive reduction of~$E$ if and only if~$E$ is a minimal element of the fiber of the reverse of~$a$ under~$\restrictionMap[a]$.
\end{proof}

\begin{corollary}
\label{coro:numberJoinMeetIrreducibles}
The number of join (resp.~meet) irreducible elements of~$\AR$ is at least~$|A|$.
\end{corollary}

Note that \cref{prop:joinIrreducibles,coro:numberJoinMeetIrreducibles} hold for any directed acyclic graph~$D$.
We will state a much more precise count of join and meet irreducibles of~$\AR$ when~$D$ is skeletal using ropes in \cref{subsec:ropes}.


\subsection{Distributivity}
\label{subsec:distributivity}

A finite lattice~$(L, \le, \meet, \join)$ is \defn{distributive} if $x \join (y \meet z) = (x \join y) \meet (x \join z)$ (or equivalently $x \meet (y \join z) = (x \meet y) \join (x \meet z)$) for any~$x, y, z \in L$.
The fundamental theorem for distributive lattices affirms that $L$ is distributive if and only if it is isomorphic to the lattice of lower ideals of its join irreducible poset (or equivalently of upper ideals of its meet irreducible poset).
The following statement says that an acyclic reorientation lattice is distributive if and only if it is a boolean lattice.

\begin{proposition}
\label{prop:distributive}
The acyclic reorientation poset~$\AR$ is a distributive lattice if and only if~$D$ is a forest.
\end{proposition}

\begin{proof}
If~$D$ is a forest, all reorientations of~$D$ are acyclic, so that~$\AR$ is a boolean lattice.

\medskip
Conversely, assume that~$D$ is not a forest.
Since~$D$ is vertebrate, its transitive reduction~$R$ is a forest, so that there exists a directed path~$a_1, \dots, a_\ell$ (with $\ell > 2$) in~$R$ and an arc~$a$ in~$D$ with the same endpoints.
Let~$E$ denote the acyclic reorientation of~$D$ obtained by reversing all arcs except~$a_\ell$.
For~$i \in [\ell]$, denote by~$F_i$ the acyclic reorientation of~$D$ obtained by reversing only the arc~$a_i$.
Observe that
\begin{itemize}
\item all the arcs~$a_i$ are reversed in~$\bigJoin_{i \in [\ell]} F_i$, so that all the arcs in their transitive closure are reversed in~$\bigJoin_{i \in [\ell]} F_i$. By \cref{thm:joinMeet}, we obtain that~$a$ is reversed in~$E \meet \bigJoin_{i \in [\ell]} F_i$.
\item $E \meet F_i = F_i$ for~$i \in [\ell-1]$ while and~$E \meet F_\ell = D$, so that~$\bigJoin_{i \in [\ell]} (E \meet F_i) = \bigJoin_{i \in [\ell-1]} F_i$. By \cref{thm:joinMeet}, we obtain that~$a$ is not reversed in~$\bigJoin_{i \in [\ell]} (E \meet F_i)$.
\end{itemize}
Therefore, $E \meet \bigJoin_{i \in [\ell]} F_i \ne \bigJoin_{i \in [\ell]} (E \meet F_i)$ which shows that $\AR$ is not distributive.
\end{proof}


\subsection{Semidistributivity and canonical representations}
\label{subsec:semidistributivity}

A finite lattice~$(L, \le, \meet, \join)$ is \defn{join semidistributive} if $x \join y = x \join z$ implies $x \join (y \meet z) = x \join y$ for any~$x, y, z \in L$.
Equivalently, $L$ is join semidistributive if for any cover relation~$x \lessdot y$ in~$L$, the set~$K_\join(x,y) \eqdef \set{z \in L}{x \join z = y}$ has a unique minimal element~$k_\join(x,y)$.
Note that~$k_\join(x,y)$ is join irreducible.
In particular we define~$\kappa_\join(m) \eqdef k_\join(m, m^\star)$ for a meet irreducible~$m$ of~$L$.
The \defn{meet semidistributivity} and the maps~$K_\meet$, $k_\meet$ and~$\kappa_\meet$ are defined dually.
A lattice~$L$ is \defn{semidistributive} if it is both join and meet semidistributive.
In this case, the maps~$\kappa_\join$ and~$\kappa_\meet$ are inverse bijections between the meet irreducible and the join irreducible elements of~$L$.

\medskip
Our next statement characterizes semidistributivity for acyclic reorientation lattices.
Recall that~$D$ is \defn{filled} when the following equivalent conditions are fulfilled:
\begin{itemize}
\item for any directed path~$\pi$ in~$D$, if the arc joining the endpoints of~$\pi$ belongs to~$D$, then all arcs joining any two vertices of~$\pi$ also belong~to~$D$,
\item the transitive support of any arc~$a$ of~$D$ induces a tournament in~$D$,
\item for any arc~$(u,v)$ in~$D$ and any vertex~$w$ in the transitive support of~$(u,v)$ minus~$\{u,v\}$, both arcs~$(u,w)$ and~$(w,v)$ also belong to~$D$,
\end{itemize}
where the \defn{transitive support} of an arc~$a$ of~$D$ is the set of vertices of~$D$ that appear along a directed path in~$D$ joining the endpoints of~$a$ (or equivalently along the directed path in the transitive reduction of~$D$ joining the endpoints of~$a$).
From now on, we abbreviate vertebrate and filled by \defn{skeletal}.
Note that chordful (meaning that any cycle induces a clique) implies skeletal, and that skeletal implies chordal (meaning that there is no induced cycle of length at least~$4$), but that both reverse implications are wrong.
In particular, any forest and any tournament is skeletal.
In fact, it is not difficult to check that the skeletal directed acyclic graphs are precisely the directed forests on which some directed paths are replaced by tournaments.
Some examples of skeletal directed acyclic graphs and their acyclic reorientation lattices are illustrated in \cref{fig:canonicalJoinComplexes}.

\begin{proposition}
\label{prop:semidistributive}
The acyclic reorientation poset~$\AR$ is a semidistributive lattice if and only if~$D$ is skeletal.
\end{proposition}

\begin{proof}
Since we focus in this paper on self-dual lattices, the notions of join semidistributivity, meet semidistributivity and semidistributivity coincide.
We focus here on join semidistributivity.

\medskip
Assume that there is a directed path with vertices~$v_0, \dots, v_\ell$ in the transitive reduction of~$D$ such that~$(v_0,v_\ell) \in D$ but there is~$0 \le i < j \le \ell$ such that~$(v_i, v_j) \notin D$.
Restricting the path, we can assume that~$(v_0,v_\ell) \in D$ while~$(v_0, v_{\ell-1}) \notin D$ or~$(v_1, v_\ell) \notin D$, say the later for instance.
Let~$X$ denote the acyclic reorientation of~$D$ obtained by reversing all arcs except the arcs~$(v_k, v_\ell)$ that belong to~$D$ (in particular the arc~$(v_0, v_\ell)$), and let~$Y$ denote the reorientation of~$D$ obtained from~$X$ by reversing~$(v_0, v_\ell)$.
For~$i \in~[\ell-1]$, let~$E_i$ denote the reorientation of~$D$ that agrees with~$Y$ except on the arc~$a_i = (v_{i-1}, v_i)$.

We claim that $E_i$ is acyclic.
Assume by means of contradiction that~$E_i$ contains a directed cycle~$C$.
Since~$a_i$ is in the transitive reduction of~$D$, it cannot be the only arc of~$D$ in~$C$.
Therefore, one of the arcs~$(v_k, v_\ell)$ is also in~$C$, so that~$(v_0, v_\ell)$ is also reversed in~$C$.
Since~$k \ne 1$ as we assumed that~$(v_1, v_\ell) \notin D$, the arc~$a_i$ does not suffice to close~$C$.

Consider now~$E \eqdef \bigMeet_{i \in [\ell-1]} E_i$.
Since the arc~$a_i$ and~$a_\ell$ are not reversed in~$E_i$, we obtain that none of the arcs~$a_i$ are reversed in~$E$, so that~$(v_0, v_\ell)$ is not reversed in~$E$, hence~$X \le E$.
We conclude that the set~$\set{F \in \AR}{X \join F = Y}$ contains all~$E_i$ but not~$E$, so that it has no minimal element.

\medskip
Assume now that~$D$ is filled and consider a cover relation~$X \lessdot Y$ in~$\AR$.
Let~$a$ denote the arc reversed from~$X$ to~$Y$.
We say that an arc of~$D$ is \defn{forced} if it is the only arc reversed in~$Y$ along a directed path in~$D$ joining the endpoints of~$a$.
In other words, an arc of~$D$ is forced if its endpoints are connected by a directed path in~$Y$ where~$a$ is the only reversed arc.
Note that by definition, the arc~$a$ is forced while the arcs not reversed in~$Y$ are not forced.

Our assumption that~$D$ is filled implies that any directed path in~$D$ joining the endpoints of~$a$ contains at least one forced arc.
Indeed, let~$v_0, \dots, v_\ell$ denote the vertices along such a path.
Since $a = (v_0, v_\ell)$ and $D$ is filled, all arcs~$(v_i, v_j)$ with~$0 \le i < j \le \ell$ belong to~$D$.
Let~$k \in [\ell]$ be minimal such that the arc $(v_0, v_k)$ is reversed in~$Y$.
Then neither $(v_0, v_{k-1})$ (by minimality of~$k$) nor $(v_k, v_\ell)$ (since~$X$ is acyclic and contains $(v_k, v_0)$ and $(v_0, v_\ell)$) are reversed in~$Y$.
This shows that the arc~$(v_{k-1}, v_k)$ is forced.

Let~$E$ be the reorientation of~$D$ obtained by reversing all forced arcs.
We claim that
\begin{itemize}
\item $E$ is acyclic. Otherwise, it contains a directed cycle~$C$. The arcs in~$E$ are either in~$D$ or forced. Replacing each forced arc in~$C$ by a directed path in~$Y$ where~$a$ is the only reversed arc, and taking eventually a subcycle if the result is not simple, we can assume that~$C$ is formed by the arc~$a$ together with a directed path of arcs in~$D$ joining the endpoints of~$a$. By definition of~$E$, none of the arcs along this path is forced, contradicting our earlier observation.
\item $X \join E = Y$. Indeed, since the arcs not reversed in~$Y$ are not forced, they are also not reversed in~$E$ so that~$E \le Y$. Moreover, since~$a$ is forced, it is reversed in~$E$, so that~$E \not\le X$. Therefore, $X \join E = Y$.
\item $E$ is smaller than any~$F \in \AR$ such that~$X \join F = Y$. Indeed, if~$X \join F = Y$, then all arcs not reversed in~$Y$ are not reversed in~$F$ (because~$F \le Y$), so that~$a$ is reversed in~$F$ (because~$F \not\le X$), so that the forced arcs are reversed in~$F$ (because $F$ is acyclic). Therefore, all arcs reversed in~$E$ are reversed in~$F$, so that~$E \le F$.
\end{itemize}
This shows that~$E$ is the unique minimal element of the set~$\set{F \in \AR}{X \join F = Y}$, which proves that~$\AR$ is join semidistributive, thus semidistributive (by self-duality).
\end{proof}

\begin{figure}
	\centerline{
		\begin{tabular}{c@{\quad}c@{\quad}c}
			\includegraphics[scale=.9]{acyclicReorientationPoset2} &
			\includegraphics[scale=.9]{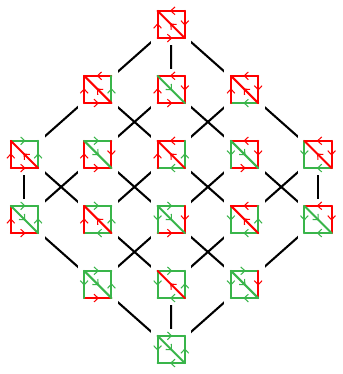} &
			\includegraphics[scale=.9]{acyclicReorientationPoset7}
			\\[.2cm]
			\includegraphics[scale=.9]{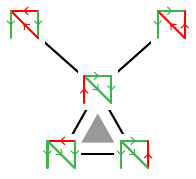} &
			\includegraphics[scale=.9]{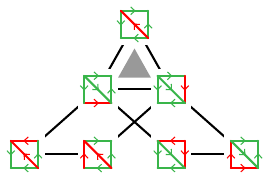} &
			\includegraphics[scale=.9]{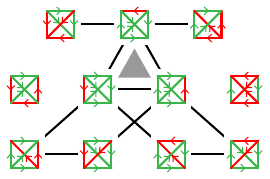}
		\end{tabular}
	}
	\caption{Three semidistributive acyclic reorientation lattices (top) and their canonical join complexes (bottom). The rightmost lattice is isomorphic to the weak order on permutations.}
	\label{fig:canonicalJoinComplexes}
\end{figure}

Semidistributivity enables us to consider canonical representations.
A \defn{join representation} of~${x \in L}$ is a subset~$J \subseteq L$ such that~$x = \bigJoin J$.
Such a representation is \defn{irredundant} if~$x \ne \bigJoin J'$ for any strict subset~$J' \subsetneq J$.
The irredundant join representations of an element~$x \in L$ are ordered by containement of the lower ideals of their elements, \ie~$J \le J'$ if and only if for any~$y \in J$ there exists~$y' \in J'$ such that~$y \le y'$ in~$L$.
The \defn{canonical join representation} of~$x$ is the minimal irredundant join representation of~$x$ for this order when it exists.
Its elements are the \defn{canonical joinands} of~$x$.
The \defn{canonical meet representations} and the \defn{canonical meetands} are defined dually.

\pagebreak
\medskip
A classical result affirms that a finite lattice~$L$ is join (resp.~meet) semidistributive if and only if any element of~$L$ admits a canonical join (resp.~meet) representation~\cite[Thm.~2.24]{FreeseNation}.
Moreover, in a join (resp.~meet) semidistributive lattice, the canonical join (resp.~meet) representation of~$y \in L$ is given by
\[
y = \bigJoin_{x \lessdot y} k_\join(x, y)
\qquad \text{(resp. } 
y = \bigMeet_{y \lessdot z} k_\meet(y, z)
\text{)}
\]
where~$k_\join(x,y)$ is the minimal element of~$K_\join(x,y) \eqdef \set{z \in L}{x \join z = y}$ (resp.~$k_\meet(y,z)$ is the maximal element of~$K_\meet(y,z) \eqdef \set{x \in L}{x \meet y = z}$).

\medskip
Combining this description with \cref{prop:joinIrreducibles,prop:semidistributive}, we obtain the join (resp.~meet) canonical representations in the acyclic reorientation lattice, generalizing the description of~\cite{Reading-arcDiagrams} for the weak order.
An alternative description is presented later in \cref{prop:canonicalJoinMeetComplex} in terms of ropes.

\begin{corollary}
\label{coro:canonicalJoinMeetRepresentation}
Assume that~$D$ is skeletal.
The canonical join (resp.~meet) representation of an acyclic reorientation~$E$ of~$D$ is given by~$E = \bigJoin_a E_a$ (resp.~$E = \bigMeet_a E_a$) where
\begin{itemize}
\item $a$ runs over all arcs of~$D$ reversed (resp.~not reversed) in the transitive reduction of~$E$,
\item $E_a$ is the acyclic reorientation of~$D$ where an arc is reversed (resp.~not reversed) if and only if it is the only arc reversed (resp.~not reversed) in~$E$ along a directed path in~$D$ joining the endpoints of~$a$.
\end{itemize}
\end{corollary}

The \defn{canonical join complex} of a join semidistributive lattice~$L$ is the simplicial complex on join irreducible elements of~$L$ whose faces are the canonical join representations of the elements of~$L$.
For instance, \cref{fig:canonicalJoinComplexes} shows the canonical join complexes for some acyclic reorientation lattices.
It was proved in~\cite{Barnard} that this complex is flag (\ie its minimal non-faces are edges, or equivalently it is the clique complex of its graph) if and only if~$L$ is semidistributive.
The \defn{canonical meet complex} is defined dually.
The \defn{canonical complex} of a semidistributive lattice~$L$ is the simplicial complex on join irreducible elements and meet irreducible elements of~$L$ whose faces are of the form~$J \cup M$ where~$J$ is a canonical join representation, $M$ is a canonical meet representation, and~$\bigJoin J \le \bigMeet M$.
It was proved in~\cite{AlbertinPilaud} that this complex is again flag.
Note that the canonical join (resp.~meet) complex encodes the elements of~$L$, while the canonical complex encodes the intervals of~$L$.
We will describe the canonical join (resp.~meet) complex of~$\AR$ in \cref{coro:canonicalJoinMeetComplex} using non-crossing rope diagrams generalizing~\cite{Reading-arcDiagrams}, and the canonical complex of~$\AR$ in \cref{coro:canonicalComplex} using rope bidiagrams generalizing~\cite{AlbertinPilaud}.


\subsection{Congruence normality and uniformity}
\label{subsec:congruenceNormality}

Recall that a \defn{congruence} of a finite lattice~$(L, \le, \meet, \join)$ is an equivalence relation on~$L$ that respects meets and joins, that is $x \equiv x'$ and $y \equiv y'$ implies $x \join y \equiv x' \join y'$ and $x \meet y \equiv x' \meet y'$.
The \defn{lattice quotient}~$L/{\equiv}$ is the lattice structure on the congruence classes of~$\equiv$, where for any two congruence classes~$X$ and~$Y$, the order is given by~$X \le Y$ if and only if $x \le y$ for some representatives~$x \in X$ and~$y \in Y$, and the join~$X \join Y$ (resp.~meet~$X \meet Y$) is the congruence class of~$x \join y$ (resp.~$x \meet y$) for any representatives~$x \in X$ and~$y \in Y$.
In other words, the projection map sending an element of~$L$ to its congruence class is a lattice map.
Moreover, the lattice quotient~$L/{\equiv}$ is isomorphic to the subposet of~$L$ induced by the minimal elements in their congruence classes.

\medskip
\enlargethispage{.1cm}
The set~$\con(L)$ of all congruences of~$L$, ordered by refinement, forms itself a distributive lattice where the meet is the intersection of relations and the join is the transitive closure of union of relations.
For any $x, y \in L$, there is a unique minimal congruence~$\con(x,y)$ in which~$x \equiv y$.
For a join irreducible element~$j$ of~$L$ (covering a single element~$j_\star$), the congruence~$\con(j_\star,j)$ is join irreducible in the congruence lattice~$\con(L)$.
Similarly, for any meet irreducible element~$m$ of~$L$, the congruence~$\con(m, m^\star)$ is meet irreducible in~$\con(L)$.
The lattice~$L$ is called
\begin{itemize}
\item \defn{congruence normal} if $\con(j_\star, j) \ne \con(m, m^\star)$ for any join irreducible $j$ and meet irreducible $m$ such that~$j \le m$,
\item \defn{congruence uniform} if the map~$j \mapsto \con(j_\star,j)$ (resp.~$m \mapsto \con(m, m^\star)$) is a bijection between the join (resp.~meet) irreducible elements of~$L$ and that of~$\con(L)$.
\end{itemize}
A lattice is congruence uniform if and only if it is congruence normal and semidistributive.

\medskip
In the sequel, we will use an alternative characterization of congruence normality and congruence uniformity in terms of convex and interval doublings in the sense of~\cite{Day}.
Given a poset~$P$ and a subset~$X$ of~$P$, the \defn{doubling} of~$X$ in~$P$ is the poset~$P[X]$ on~$(P \ssm X) \sqcup (X \times \{0,1\})$ defined by:
\begin{itemize}
\item $a \le b$ in~$P[X]$ if~$a, b \notin X$ and~$a \le b$ in~$P$,
\item $(a,i) \le b$ in~$P[X]$ if~$a \in X$, $b \notin X$, $i \in \{0,1\}$, and $a \le b$ in~$P$,
\item $a \le (b,j)$ in~$P[X]$ if~$a \notin X$, $b \in X$, $j \in \{0,1\}$, and $a \le b$ in~$P$,
\item $(a,i) \le (b,j)$ in~$P[X]$ if~$a,b \in X$, $i,j \in \{0,1\}$, and $a \le b$ in~$P$ and~$i \le j$.
\end{itemize}
This construction is illustrated in \cref{fig:doublingLattices}.
It was observed that if~$L$ is a lattice and $C \subseteq L$ is order convex (\ie $x \le y \le z$ and~$x,z \in C$ implies~$y \in C$), then~$L[C]$ is again a lattice.
A lattice is congruence normal (resp.~uniform) if and only if it can be obtained from a distributive lattice by a sequence of doublings of order convex sets (resp.~of intervals).

\begin{figure}
	\centerline{\includegraphics[scale=.9]{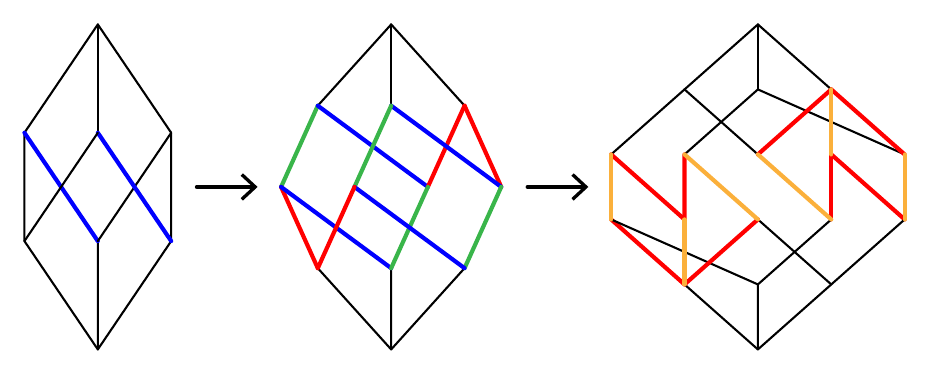}}
	\caption{A sequence of order convex doublings in a lattice. The two blue segments in the first lattice are doubled into four blue segments connected by four green segments in the second lattice, and the four red segments in the second lattice are doubled into height red segments connected by six orange segments in the third lattice. The first step could be decomposed into two interval doublings, the second cannot.}
	\label{fig:doublingLattices}
\end{figure}

\begin{proposition}
\label{prop:congruenceNormal}
The acyclic reorientation poset~$\AR$ is a congruence normal lattice for any vertebrate directed acyclic graph~$D$.
\end{proposition}

\begin{proof}
Order the arcs of~$D$ by~$a \prec b$ if there is a directed path in~$D$ containing~$a$ joining the endpoints of~$b$.
The minimal elements of~$\prec$ are the arcs of the transitive reduction~$R$ of~$D$.
Choose an arbitrary order~$a_1, \dots, a_\ell$ on the arcs of~$D \ssm R$ so that~$a_i \prec a_j$ for~$i < j$.
Let~${R = D_0, D_1, \dots, D_\ell = D}$ be the directed subgraphs of~$D$ obtained by adding the arcs~$a_1, \dots, a_\ell$ one by one.

Let~$i \in [\ell]$.
Let~$X_i$ (resp.~$Y_i$) denote the set of acyclic reorientations of~$D_{i-1}$ which can be completed into an acyclic reorientation of~$D_i$ by adding~$a_i$ (resp.~the reverse of~$a_i$), and denote~${Z_i = X_i \cap Y_i}$.
Clearly, the acyclic reorientation poset~$\AR[D_i]$ is isomorphic to the doubling of~$Z_i$ in~$\AR[D_{i-1}]$.
Moreover, we claim that~$Z_i$ is order convex in~$D_{i-1}$.
This immediately follows from the fact that~$X_i$ (resp.~$Y_i$) is a lower (resp.~upper) ideal of~$\AR[D_{i-1}]$.

To prove this fact, it suffices by symmetry to show that~$Y_i$ is an upper ideal.
Consider an acyclic reorientation~$E$ of~$D_{i-1}$ and the reorientation~$F$ of~$D_i$ that agrees with~$E$ on~$D_{i-1}$ and where~$a_i$ is reversed.
If~$E$ does not belong to~$Y_i$, then $F$ contains a cycle~$C$.
We can assume that~$C$ is induced (as any chord in a directed cycle defines a smaller directed cycle) and we know that~$C$ contains~$a_i$ (because $E$ is acyclic).
Since~$D$ is vertebrate, there exists an arc~$c$ of~$C$ such that either~$c$ is reversed in~$F$ while $C \ssm \{c\}$ is not, or~$C \ssm \{c\}$ is reversed in~$F$ while~$c$ is not.
In the former case, we have~$c = a_i$ (because $a_i$ is reversed and belongs to~$C$) and the arcs of $C \ssm \{c\}$ are not reversed in~$E$.
It follows that for any~$E' \le E$, the arcs of $C \ssm \{c\}$ are not reversed in~$E'$ so that~$E'$ cannot belong to~$Y_i$.
In the later case, we would have~$a_i \prec c$ contradicting our assumption that~$a_h \prec a_i$ for~$h < i$.
We conclude that~$E \notin Y_i$ and $E' \le E$ implies $E' \notin Y_i$, so that~$Y_i$ is an upper ideal of~$\AR[D_{i-1}]$.

To sum up, we obtained a sequence of lattices~${\AR[R] = \AR[D_0], \AR[D_1], \dots, \AR[D_\ell] = \AR}$, where each~$\AR[D_i]$ is isomorphic to the doubling of the order convex set~$Z_i$ in~$\AR[D_{i-1}]$.
Since~$R$ is a forest, $\AR[R]$ is distributive, so that~$\AR$ is congruence normal.
Such a sequence is illustrated in \cref{fig:doublingAcyclicReorientationLattices}.
\begin{figure}
	\centerline{\includegraphics[scale=.9]{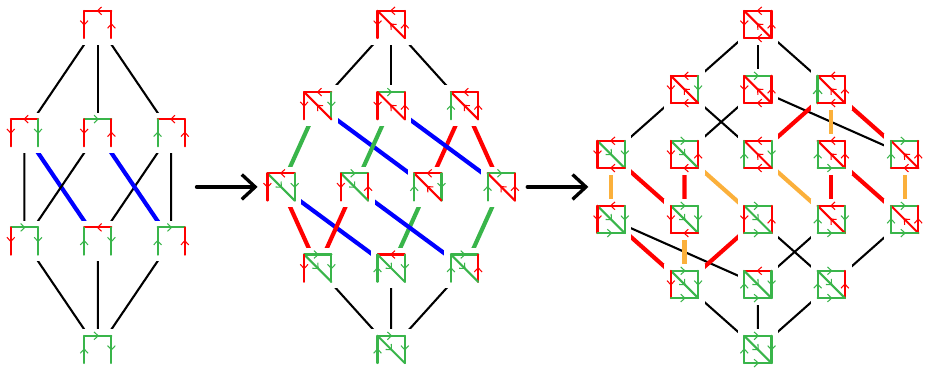}}
	\caption{Doubling convex sets in acyclic reorientation posets. See \cref{fig:doublingLattices} for the explanation of the colors of the cover relations.}
	\label{fig:doublingAcyclicReorientationLattices}
\end{figure}
\end{proof}

Another approach to prove \cref{prop:congruenceNormal} is to
\begin{itemize}
\item order the arcs of~$D$ by inclusion of their transitive supports,
\item label each cover relation~$E \lessdot E'$ in~$\AR$ by the arc of~$D$ flipped from~$E$ to~$E'$.
\end{itemize}
We invite the reader to check that this defines a CN-labeling of~$\AR$ in the sense of N.~Reading~\cite[Thm.~4]{Reading-posetRegions}, and thus implies that~$\AR$ is congruence normal.
The advantage of describing an explicit sequence of order convex doubling is that our proof of \cref{prop:congruenceNormal} actually provides an alternative proof of \cref{thm:characterizationLattice}.

\medskip
We now switch to congruence uniformity.
The following statement is a direct consequence of \cref{prop:semidistributive,prop:congruenceNormal}.
However, we sketch an independent proof based on interval doublings, which thus provides an alternative proof of \cref{prop:semidistributive}.

\begin{proposition}
\label{prop:congruenceUniform}
The acyclic reorientation poset~$\AR$ is a congruence uniform lattice if and only if~$D$ is skeletal.
\end{proposition}

\begin{proof}
Consider the sequence of directed acyclic graphs~$D_i$ constructed in the proof of \cref{prop:congruenceNormal}.
We claim that, while the convex sets~$Z_i$ are not always intervals, they can be partitioned into intervals~$Z_i = I_i^1 \sqcup \dots \sqcup I_i^k$. 
Doubling separately these intervals~$I_i^p$ thus shows that~$\AR$ is congruence uniform.

To see the claim, consider the arc~$a_i$ added at a given step~$i \in [\ell]$.
Since~$D$ is skeletal, the transitive reduction of~$a_i$ induces a tournament~$K_i$ in~$D$.
Since we are adding the arcs of~$D \ssm R$ in an order compatible with~$\prec$, all arcs of~$K_i$ belong to~$D_i$.
Partition the acyclic reorientations of~$Z_i$ according to their restriction to~$K_i \ssm \{a_i\}$.
Since $K_i \ssm \{a_i\}$ is clearly weakly pathful in~$D_i$ (it is actually strongly pathful), this partitions~$Z_i$ into intervals of~$\AR[D_{i-1}]$ by \cref{thm:latticeMap}.
\end{proof}


\pagebreak
\section{Rope diagrams}
\label{sec:ropeDiagrams}

Throughout this section, we assume that~$D$ is skeletal, so that the acyclic reorientation poset~$\AR$ is a congruence uniform lattice by \cref{prop:congruenceUniform}.
We introduce ropes and non-crossing rope diagrams, generalizing the work of N.~Reading in~\cite{Reading-arcDiagrams} on arcs and non-crossing arc diagrams.
In our setting, we prefer the word ``rope'' rather than ``arc'' to avoid the possible confusions with the arcs of the directed graphs.


\subsection{Ropes and irreducibles}
\label{subsec:ropes}

A \defn{rope} of~$D$ is a quadruple~$(u, v, \down, \up)$ where~$(u,v)$ is an arc of~$D$ and~$\down \sqcup \up$ is a partition of the transitive support of~$(u,v)$ minus~$\{u,v\}$ (or equivalently since~$D$ is filled, the vertices~$w$ so that both~$(u,w)$ and~$(w,v)$ belong to~$D$).

\begin{lemma}
\label{lem:numberRopes}
Assume that~$D$ is skeletal. Then the ropes of~$D$ are 
\begin{enumerate}[(i)]
\item counted by~$\sum_{a \in A} 2^{ts(a)-2}$ where~$ts(a)$ denotes the size of the transitive support of~$a$ in~$D$,
\item in bijection with the cliques of~$D$ with at least $2$ vertices.
\end{enumerate}
\end{lemma}

\begin{proof}
First, (i) is immediate since a rope of~$D$ is given by an arc~$(u,v)$ of~$D$ together with a subset~$\up$ of the transitive support of~$(u,v)$ in~$D$ minus~$\{u,v\}$.
For (ii), note that 
\begin{itemize}
\item each rope~$(u, v, \down, \up)$ defines a clique induced by~$\up \cup \{u,v\}$,
\item each clique~$K$ defines a rope~$(u,v, \down, \up)$ where~$u$ and~$v$ are the source and target of~$K$, and~$\up$ (resp.~$\down$) are the vertices of the transitive support of~$(u,v)$ which belong (resp.~do not belong) to~$K$.
\qedhere
\end{itemize}
\end{proof}

For an acyclic reorientation~$E$ of~$D$ and an arc~$(u,v) \in D$, we set
\begin{align*}
\down_{u,v}^E & \eqdef \set{w \in V}{(u,w) \in D \ssm E \text{ and } (w,v) \in D \cap E} \\
\text{and} \qquad
\up_{u,v}^E & \eqdef \set{w \in V}{(u,w) \in D \cap E \text{ and } (w,v) \in D \ssm E},
\end{align*}
and we define
\[
\rope_{u,v}^E \eqdef (u, v, \down_{u,v}^E, \up_{u,v}^E).
\]
We need the following two elementary properties of the sets~$\down_{u,v}^E$ and~$\up_{u,v}^E$.

\begin{lemma}
\label{lem:DeltaNabla}
The sets~$\down_{u,v}^E$ and~$\up_{u,v}^E$ fulfill the following properties:
\begin{enumerate}[(i)]
\item for any distinct vertices~$w \in \down_{u,v}^E \cup \{u,v\}$ and~$w' \in \up_{u,v}^E \cup \{u,v\}$, there is an arc~$(w,w')$ in~$E$, except if~$(w,w') = (u,v) \notin E$ or if~$(w,w') = (v,u) \notin E$,
\item if~$(u,v)$ or~$(v,u)$ appears in the transitive reduction of~$E$, then $\rope_{u,v}^E$ is a rope of~$D$.
\end{enumerate}
\end{lemma}

\begin{proof}
For the first point, observe that $E$ contains arcs from any vertex of~$\down_{u,v}^E$ to both~$u$ and~$v$, and from both~$u$ and~$v$ to any vertex of~$\up_{u,v}^E$.
Therefore, except when~$(w,w') = (u,v) \notin E$ or if~$(w,w') = (v,u) \notin E$, there is a directed path in~$E$ joining~$w$ to~$w'$.
Since~$D$ is filled, and both~$w$ and~$w'$ belong to the transitive support of~$(u,v)$, it follows that~$(w,w')$ is an arc of~$E$.

\medskip
For the second point, assume for instance that~$(u,v)$ appears in the transitive reduction of~$E$ and consider a vertex~$w$ such that both~$(u,w)$ and~$(w,v)$ belong to~$D$.
Since~$(u,v)$ belongs to the transitive reduction of~$E$, either~$(u,w)$ or~$(w,v)$ is reversed in~$E$.
Since~$E$ is acyclic, either~$(u,w)$ or~$(w,v)$ is not reversed in~$E$.
Therefore, $w$ belongs either to~$\up_{u,v}^E$ or to~$\down_{u,v}^E$.
In other words, $\down_{u,v}^E \sqcup \up_{u,v}^E$ is a partition of these vertices and $\rope_\join(E)$ is indeed a rope of~$D$.
\end{proof}

We now connect the ropes of~$D$ with the join and meet irreducibles of~$\AR$:
\begin{itemize}
\item for a join (resp.~meet) irreducible~$I$ of~$\AR$, let~$\rope_\join(I)$ (resp.~$\rope_\meet(I)$) be the rope~$\rope_{u,v}^I$ where~$(u,v)$ is the only arc reversed (resp.~not reversed) in the transitive reduction of~$I$,
\item for a rope~$\rope \eqdef (u, v, \down, \up)$ on~$D$, let~$I_\join(\rope)$ (resp.~$I_\meet(\rope)$) be the reorientation of~$D$ where an arc~$(w,w')$ of~$D$ is reversed (resp.~not reversed) if and only if~$w \in \up \cup \{u\}$ and~$w' \in \down \cup \{v\}$ (resp.~$w \in \down \cup \{u\}$ and~$w' \in \up \cup \{v\}$).
\end{itemize}
For an illustration of these maps, compare \cref{fig:canonicalJoinComplexes}\,(bottom) with \cref{fig:canonicalJoinComplexesRopes}.

\begin{proposition}
\label{prop:bijectionsRopes}
Assume that~$D$ is skeletal.
The two maps~$\rope_\join$ and~$I_\join$ (resp.~$\rope_\meet$ and~$I_\meet$) are inverse bijections between the join (resp.~meet) irreducibles of~$\AR$ and the ropes of~$D$.
\end{proposition}

\begin{proof}
As~$\AR$ is self-dual under reversing all arcs, we focus on join irreducibles.

\medskip
It follows from \cref{lem:DeltaNabla} that $\rope_\join(J)$ is indeed a rope of~$D$ for any join irreducible~$J$ of~$\AR$.

\medskip
Conversely, consider a rope~$\rope \eqdef (u, v, \down, \up)$ of~$D$.
We claim that the reorientation~$I_\join(\rope)$ of~$D$ is acyclic.
Indeed, since~$D$ is filled, $\down \sqcup \up$ covers all vertices that appear along a directed path in~$D$ joining~$u$ to~$v$.
Hence, there exists a total order~$\prec$ on~$V$ so that all arcs of~$D$ are increasing for~$\prec$ and $\down \sqcup \up = \set{w \in V}{u \prec w \prec v}$.
Let~$\prec'$ denote the total order on~$V$ obtained from~$\prec$ by reordering~$\{u,v\} \sqcup \down \sqcup \up$ such that~$\down$ appears first (in~$\prec$ order), then~$v$ and~$u$, and~$\up$ appears last (in~$\prec$ order).
Then~all arcs of~$I_\join(\rope)$ are clearly increasing for~$\prec'$, so that~$I_\join(\rope)$ is indeed acyclic.
Moreover, $(u,v)$ is by definition the only arc of~$D$ reversed in~$I_\join(\rope)$ which belongs to the transitive reduction of~$I_\join(\rope)$.
By \cref{prop:joinIrreducibles}, we conclude that~$I_\join(\rope)$ is join irreducible in~$\AR$.

\medskip
Finally, it is immediate to check that~$I_\join(\rope_\join(J)) = J$ for any join irreductible~$J$ of~$\AR$ and that~$\rope_\join(I_\join(\rope)) = \rope$ for any rope~$\rope$ of~$D$, so that~$\rope_\join$ and~$I_\join$ are inverse bijections between join irreducibles of~$\AR$ and ropes~of~$D$.
\end{proof}

Note that combining \cref{lem:numberRopes}\,(i) and \cref{prop:bijectionsRopes}, we obtain a precise count of the join (resp.~meet) irreducibles of~$\AR$ when~$D$ is skeletal, refining the lower bound of \cref{coro:numberJoinMeetIrreducibles}.

\medskip
We finally observe that the bijections~$\rope_\join$ and~$\rope_\meet$ provide a simple description of the Kreweras maps~$\kappa_\join$ and~$\kappa_\meet$ defined in \cref{subsec:semidistributivity}. Namely, it is easy to check that~$\rope_\join(\kappa_\join(M)) = \rope_\meet(M)$ for any meet irreducible~$M$ of~$\AR$, and $\rope_\meet(\kappa_\meet(J)) = \rope_\join(J)$ for any join irreducible~$J$ of~$\AR$.


\subsection{Rope diagrams and canonical representations}
\label{subsec:ropeDiagrams}

Two ropes~$(u, v, \down, \up)$ and $(u', v', \down', \up')$ are \defn{crossing} if there are distinct vertices~$w \ne w'$ such that~$w \in (\down \cup \{u,v\}) \cap (\up' \cup \{u',v'\})$ and ${w' \in (\up \cup \{u,v\}) \cap (\down' \cup \{u',v'\})}$.
A \defn{non-crossing rope diagram} is a collection of pairwise non-crossing ropes of~$D$.
The \defn{non-crossing rope complex} of~$D$ is the simplicial complex of non-crossing rope diagrams of~$D$.

\medskip
We now connect the non-crossing rope diagrams of~$D$ with the elements of~$\AR$:
\begin{itemize}
\item for an acyclic reorientation~$E$ of~$D$, let~$\diagram_\join(E)$ (resp.~$\diagram_\meet(E)$) be the set of ropes~$\rope_{u,v}^E$ for all arcs~$(u,v)$ reversed (resp.~not reversed) in the transitive reduction of~$E$,
\item for a non-crossing rope diagram~$\diagram$ of~$D$, define
\[
E_\join(\diagram) \eqdef \bigJoin_{\rope \in \diagram} I_\join(\rope)
\qquad\text{(resp. }
E_\meet(\diagram) \eqdef \bigMeet_{\rope \in \diagram} I_\meet(\rope)
\text{).}
\]
\end{itemize}

\begin{proposition}
\label{prop:bijectionsNonCrossingRopeDiagrams}
Assume that~$D$ is skeletal.
The two maps~$\diagram_\join$ and~$E_\join$ (resp.~$\diagram_\meet$ and~$E_\meet$) are inverse bijections between acyclic reorientations of~$D$ and non-crossing rope diagrams of~$D$.
\end{proposition}

\begin{proof}
As~$\AR$ is self-dual under reversing all arcs, we focus on~$\diagram_\join$ and~$E_\join$.

\medskip
We first prove that~$\diagram_\join(E)$ is indeed non-crossing.
Assume by means of contradiction that~$\diagram_\join(E)$ contains two distinct ropes~$\rope \eqdef (u, v, \down, \up)$ and~$\rope' \eqdef (u', v', \down', \up')$ of~$D$ such that there are two distinct vertices~$w \ne w'$ with~$w \in (\down \cup \{u,v\}) \cap (\up' \cup \{u',v'\})$ and~$w' \in (\up \cup \{u,v\}) \cap (\down' \cup \{u',v'\})$.
Since $E$ cannot contain simultaneously the arcs~$(w,w')$ and~$(w',w)$, we can assume for instance that~$(w,w')$ is not in~$E$.
Since~$w \in (\down \cup \{u,v\})$ and~$w' \in (\up \cup \{u,v\})$, \cref{lem:DeltaNabla} implies that~$w = u$ and~$w' = v$.
We distinguish four cases:
\begin{itemize}
\item If~$w = u = v'$ and~$w' = v = u'$, then $E$ cannot contain both arcs~$(v,u)$ and~$(v',u')$,
\item If~$w = u = v'$ and $w' = v \ne u'$, then~$w' \in \down'$ so that $D$ contains the arc~$(w',v') = (v,u)$,
\item If~$w = u \ne v'$ and~$w' = v = u'$, then~$w \in \up'$ so that~$D$ contains the arc~$(u',w) = (v,u)$,
\item If~$w \ne v'$ and $w' \ne u'$, then $w \in \up' \cup \{u'\}$ and $w' \in \down' \cup \{v'\}$, so that $E$ contains a directed path joining~$w' = v$ to~$w = u$ and passing through the arc~$(v',u')$, hence~$(v,u)$ is not in the transitive reduction of~$E$.
\end{itemize}
All four cases contradict our assumption that both~$\rope$ and~$\rope'$ are in~$\diagram_\join(E)$.
We conclude that~$\diagram_\join(E)$ is a non-crossing rope diagram.

\medskip
Conversely, observe that~$E_\join(\delta)$ is well-defined since each~$I_\join(\rope)$ is an acyclic reorientation of~$D$ by \cref{prop:bijectionsRopes}.

\medskip
We now prove that~$E_\join(\diagram_\join(E)) = E$ for any acyclic reorientation~$E$ of~$D$.
From \cref{coro:canonicalJoinMeetRepresentation} and the definition of~$\diagram_\join$ and~$E_\join$, we have:
\[
E = \bigJoin_{(u,v)} E_{(u,v)}
\qquad\text{and}\qquad
E_\join(\diagram_\join(E)) = \bigJoin_{\rope \in \diagram_\join(E)} I_\join(\rope) = \bigJoin_{(u,v)} I_\join(\rope_{u,v}^E)
\]
where~$(u,v)$ runs over all arcs of~$D$ reversed in the transitive reduction of~$E$.
We thus just need to prove that~$I_\join(\rope_{u,v}^E) = E_{(u,v)}$ for any arc~$(u,v)$ of~$D$ reversed in the transitive reduction of~$E$.
We will show that any arc~$(w,w')$ in~$D$ is reversed in~$I_\join(\rope_{u,v}^E)$ if and only if it is reversed in~$E_{(u,v)}$.

If~$(w,w')$ is reversed in~$E_{(u,v)}$, then by definition it is the only arc reversed in~$E$ along a directed path in~$D$ joining~$u$ to~$v$.
Since~$E$ is filled, either~$u = w$ or both $(u,w)$ and~$(w,v)$ are arcs of~$D$.
Moreover, by acyclicity of~$E$, $(u,w)$ is not reversed in~$E$, so that~$w \in \up_{u,v}^E$.
Similarly, either $w' = v$ or $w' \in \down_{u,v}^E$.
It follows by definition that~$(w,w')$ is reversed in~$I_\join(\rope_{u,v}^E)$

Conversely, assume that~$(w,w')$ is reversed in~$I_\join(\rope_{u,v}^E)$.
Then by definition~$w \in \up_{u,v}^E \cup \{u\}$ and~$w' \in \down_{u,v}^E \cup \{v\}$.
Consider the directed path~$\pi$ in~$D$ formed by the arcs~$(u,w)$, $(w,w')$, and~$(w',v)$ (of course, ignore the first arc if~$u = w$ and the last arc if~$w' = v$).
Since~$(w,w')$ is the only arc reversed in~$E$ along~$\pi$, we obtain by definition that~$(w,w')$ is reversed in~$E_{(u,v)}$.

\medskip
Finally, we prove that~$\diagram_\join(E_\join(\diagram)) = \diagram$ for any non-crossing rope diagram~$\diagram$.
By \cref{thm:joinMeet}, an arc is reversed in~$E_\join(\diagram)$ if and only if it belongs to the transitive closure of the arcs reversed in at least one of the reorientations~$I_\join(\rope)$ for~$\rope \in \diagram$.
It immediately follows that if an arc~$(u,v)$ of~$D$ is reversed in the transitive reduction of~$E_\join(\diagram)$, then $\diagram$ contains a rope of the form~$(u,v, \down, \up)$.
Conversely, fix a rope~$\rope \eqdef (u,v, \down, \up) \in \diagram$.
We prove below the following claims:
\begin{enumerate}[(i)]
\item for any~$w \in \down \sqcup \up$, exactly one of the two arcs~$(u,w)$ and~$(w,v)$ is reversed in~$E_\join(\diagram)$,
\item $\rope_{u,v}^{E_\join(\diagram)} = \rope$,
\item the arc~$(u,v)$ is reversed in the transitive reduction of~$E_\join(\diagram)$.
\end{enumerate}
We deduce from~(ii) and~(iii) that~$\diagram_\join(E_\join(\diagram)) = \diagram$.

For (i), suppose by symmetry that~$w \in \down$.
Since~${E_\join(\diagram) \ge I_\join(\rope)}$, the arc~$(u,w)$ is reversed in~$E_\join(\diagram)$.
Assume by means of contradiction that~$(w,v)$ is also reversed in~$E_\join(\diagram)$.
By \cref{thm:joinMeet}, there exists a directed path~$w = w_0, w_1, \dots, w_\ell = v$ such that, for each~$i \in [\ell]$, the arc~$(w_{i-1}, w_i)$ is reversed in~$I_\join(\rope')$ for at least one~$\rope' \in \diagram$.
Since~$D$ is filled, $(u,w_{\ell-1})$ is also an arc of~$D$.
Since~$E_\join(\diagram)$ is acyclic and all arcs~$(u,w)$ and~$(w_{i-1}, w_i)$ are reversed in~$E_\join(\diagram)$, the arc~$(u,w_{\ell-1})$ is also reversed in~$E_\join(\diagram)$.
We thus obtain that $w_{\ell-1}$ belongs to~$\up \cup \down$ and both~$(u,w_{\ell-1})$ and~$(w_{\ell-1},v)$ are reversed in~$E_\join(\diagram)$.
Replacing~$w$ by~$w_{\ell-1}$, we can therefore assume that~$(w,v)$ is reversed in~$I_\join(\rope')$ for at least one~$\rope' \in \diagram$.
By definition of~$I_\join(\rope')$, we have~$w \in \up' \cup \{u'\}$ and~$v \in \down' \cup \{v'\}$.
Since~$w \in \down$ this implies that the ropes~$\rope$ and~$\rope'$ are crossing, contradicting our assumption on~$\diagram$.

For (ii), consider~$w \in \down$. 
Since~$E_\join(\diagram) \ge I_\join(\rope)$, the arcs~$(u,v)$ and~$(u,w)$ is reversed in~$E_\join(\diagram)$, so that the arc~$(w,v)$ is not reversed in~$E_\join(\diagram)$ by~(i).
Hence~$\down \subseteq \down_{u,v}^{E_\join(\diagram)}$.
Similarly, $\up \subseteq \up_{u,v}^{E_\join(\diagram)}$.
We conclude that~$\rope_{u,v}^{E_\join(\diagram)} = \rope$.

For (iii), we already know that the arc~$(u,v)$ is reversed in~$E_\join(\diagram)$.
If it is not in the transitive reduction of~$E_\join(\diagram)$, there is a directed path~$u = w_0, \dots, w_\ell = v$ in~$D$ (with~$\ell > 1$) completely reversed in~$E_\join(\diagram)$.
Since~$D$ is filled, $(w_1,v)$ is also an arc in~$D$, and it is reversed in~$E_\join(\diagram)$ by acyclicity.
This contradicts~(i).

\medskip
To conclude, we have shown that~$E_\join(\diagram_\join(E)) = E$ for any acyclic reorientation~$E$ of~$D$ and that~$\diagram_\join(E_\join(\diagram)) = \diagram$ for any non-crossing rope diagram~$\diagram$, so that~$\diagram_\join$ and~$E_\join$ are inverse bijections between acyclic reorientations of~$D$ and non-crossing rope diagrams of~$D$.
\end{proof}

\enlargethispage{.5cm}
As a consequence of \cref{coro:canonicalJoinMeetRepresentation,prop:bijectionsNonCrossingRopeDiagrams}, we obtain the canonical join and meet representations in~$\AR$ in terms of ropes of~$D$, and the canonical join (resp.~meet) complex in terms of non-crossing rope diagrams of~$D$.
For an illustration of the following two statements, compare \cref{fig:canonicalJoinComplexes}\,(bottom) with \cref{fig:canonicalJoinComplexesRopes}.

\begin{figure}
	\centerline{
			\includegraphics[scale=.9]{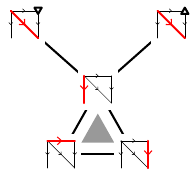} \qquad
			\includegraphics[scale=.9]{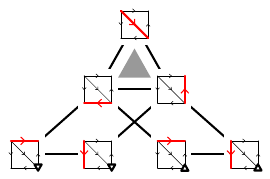} \qquad
			\includegraphics[scale=.9]{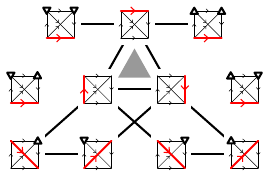}
	}
	\caption{The canonical join complexes of \cref{fig:canonicalJoinComplexes} as non-crossing rope complexes. To represent a rope~$(u, v, \down, \up)$, we highlight the arc~$(u,v)$ in red, and we mark the vertices of~$\down$ and~$\up$ by down and up triangles respectively. The rightmost complex is the non-crossing arc complex of~\cite{Reading-arcDiagrams}.}
	\label{fig:canonicalJoinComplexesRopes}
\end{figure}

\begin{corollary}
\label{prop:canonicalJoinMeetComplex}
Assume that~$D$ is skeletal.
The canonical join (resp.~meet) representation of any acyclic reorientation~$E$ of~$D$ is
\[
E = \bigJoin_{\rope \in \diagram_\join(E)} I_\join(\rope)
\qquad \text{(resp. } 
E = \bigMeet_{\rope \in \diagram_\meet(E)} I_\meet(\rope)
\text{).}
\]
\end{corollary}

\begin{corollary}
\label{coro:canonicalJoinMeetComplex}
Assume that~$D$ is skeletal.
The canonical join (resp.~meet) complex of~$\AR$ is isomorphic to the non-crossing rope complex of~$D$.
\end{corollary}

Note that \cref{prop:bijectionsRopes,prop:bijectionsNonCrossingRopeDiagrams,prop:canonicalJoinMeetComplex,coro:canonicalJoinMeetComplex} fail when~$D$ is not filled.
For instance, in the rightmost acyclic reorientation lattice of \cref{fig:distributiveSemidistributive}, $D$ has $4$ ropes but $6$ join irreducibles, and $16$ non-crossing rope diagrams but $14$ elements.

\medskip
Note that if~$D \eqdef (V,A)$ and~$D' \eqdef (V,A')$ are such that~$A \supseteq A'$ and~$D'$ is pathful in~$D$, then the ropes of~$D'$ are precisely the ropes of~$D$ supported by the arcs of~$D'$, and the non-crossing rope complex of~$D'$ is the subcomplex of the non-crossing rope complex of~$D$ induced by the ropes of~$D'$. This is a special case of lattice quotient of~$\AR$ studied in \cref{sec:congruences}.


\subsection{Rope bidiagrams and intervals}
\label{subsec:bidiagrams}

We finally briefly describe in terms of ropes the canonical complex of~$\AR$, and thus its intervals.
We start with a criterion for a join irreducible acyclic reorientation of~$D$ to be smaller than a meet irreducible acyclic reorientation of~$D$.

\begin{lemma}
\label{lem:bidiagram}
Assume that~$D$ is skeletal.
The following statements are equivalent for any two ropes~$\rope^\join \eqdef (u^\join, v^\join, \down^\join, \up^\join)$ and $\rope^\meet \eqdef (u^\meet, v^\meet, \down^\meet, \up^\meet)$ of~$D$:
\begin{itemize}
\item the join-irreducible~$I_\join(\rope^\join)$ is lower than the meet irreducible~$I_\meet(\rope^\meet)$ in~$\AR$,
\item there is no~$w \in (\up^\join \cup \{u^\join\}) \cap (\down^\meet \cup \{u^\meet\})$ and~$w' \in (\down^\join \cup \{v^\join\}) \cap (\up^\meet \cup \{v^\meet\})$ such that~$(w,w')$ is an arc of~$D$.
\end{itemize}
\end{lemma}

\begin{proof}
By definition, an arc~$(w,w')$ of~$D$ is
\begin{itemize}
\item reversed in~$I_\join(\rope^\join)$ if and only if~$w \in \up^\join \cup \{u^\join\}$ and~$w' \in \down^\join \cup \{v^\join\}$, and
\item unreversed in~$I_\meet(\rope^\meet)$ if and only if~$w \in \down^\meet \cup \{u^\meet\}$ and~$w' \in \up^\meet \cup \{v^\meet\}$.
\end{itemize}
The result immediately follows since $I_\join(\rope^\join)$ is smaller than $I_\meet(\rope^\meet)$ if and only if there is no arc of~$D$ reversed in~$I_\join(\rope^\join)$ and unreversed in~$I_\meet(\rope^\meet)$.
\end{proof}

We write $\rope_\join \vdash \rope_\meet$ if the properties of \cref{lem:bidiagram} are fulfilled.
A \defn{rope bidiagram} of~$D$ is a pair~$(\diagram^\join, \diagram^\meet)$ of non-crossing rope diagrams of~$D$ such that~$\rope^\join \vdash \rope^\meet$ for any~$\rope^\join \in \diagram^\join$ and~$\rope^\meet \in \diagram^\meet$.
The \defn{rope bidiagram complex} of~$D$ the simplicial complex whose ground set contains two copies~$\rope^\join$ and~$\rope^\meet$ of each rope~$\rope$ of~$D$, and whose faces are the sets~$\set{\rope^\join}{\rope \in \diagram^\join} \sqcup \set{\rope^\meet}{\rope \in \diagram^\meet}$ for any rope bidiagram~$(\diagram^\join, \diagram^\meet)$ of~$D$.

\begin{proposition}
\label{prop:bidiagram}
Assume that~$D$ is skeletal.
The two maps~$[E^\join, E^\meet] \mapsto \big( \diagram_\join(E^\join), \diagram_\meet(E^\meet) \big)$ and ${(\diagram^\join, \diagram^\meet) \mapsto \big[ E_\join(\diagram^\join), E_\meet(\diagram^\meet) \big]}$ are inverse bijections between the intervals of~$\AR$ and the rope bidiagrams of~$D$.
\end{proposition}

\begin{proof}
The result immediately follows from \cref{prop:bijectionsNonCrossingRopeDiagrams,prop:canonicalJoinMeetComplex}, \cref{lem:bidiagram} and the fact that~$E \le F$ if and only if any canonical joinand of~$E$ is smaller that any canonical meetand~of~$F$.
Note that this is precisely the argument of~\cite{AlbertinPilaud} to affirm that the canonical complex is flag.
\end{proof}

\begin{corollary}
\label{coro:canonicalComplex}
Assume that~$D$ is skeletal.
The canonical complex of~$\AR$ is isomorphic to the rope bidiagram complex of~$D$.
\end{corollary}


\section{Quotients of acyclic reorientation lattices}
\label{sec:congruences}

Throughout this section, we assume that~$D$ is skeletal, so that the acyclic reorientation poset~$\AR$ is a congruence uniform lattice by \cref{prop:congruenceUniform}.
We use the ropes introduced in \cref{sec:ropeDiagrams} to study the congruences of the acyclic reorientation lattice~$\AR$, generalizing N.~Reading's work on congruences of the weak order~\cite{Reading-arcDiagrams}.


\subsection{Restrictions and extensions of congruences}
\label{subsec:restrictionsExtensionsCongruences}

To describe the congruences of~$\AR$, it will be useful to restrict (resp.~extend) congruences of~$\AR$ to subgraphs (resp.~supergraphs) of~$D$.
Consider thus two directed acyclic graphs~$D \eqdef (V, A)$ and~$D' \eqdef (V, A')$ on the same vertex set~$V$ with~${A \supseteq A'}$, and assume that both~$D$ and~$D'$ are skeletal so that the acyclic reorientation posets~$\AR$ and~$\AR[D']$ are congruence uniform lattices by \cref{prop:congruenceUniform}.
To extend congruences of~$\AR[D']$ to congruences of~$\AR$, we need that~$D'$ be pathful in~$D$, so that~$\restrictionMap$ is a lattice map by~\cref{thm:latticeMap}.

\begin{proposition}
\label{prop:extensionCongruence}
If~$D'$ is pathful in~$D$, then any congruence~$\equiv'$ on~$\AR[D']$ extends to a congruence~$\equiv$ on~$\AR$ defined by~$E \equiv F$ if and only if~$\restrictionMap(E) \equiv' \restrictionMap(F)$.
\end{proposition}

\begin{proof}
Recall that~$\equiv$ is a congruence of a lattice~$L$ if and only if the classes of~$\equiv$ are the fibers of a lattice map~$L \to M$.
The result immediately since the composition~${\lambda_D \eqdef \lambda_{D'} \circ \restrictionMap}$ of any such lattice map~${\lambda_{D'} : \AR[D'] \to M}$ with the lattice map~$\restrictionMap : \AR \to \AR[D']$ is a lattice map~${\lambda_D : \AR \to M}$.
\end{proof}

Conversely, to restrict congruences of~$\AR$ to congruences of~$\AR[D']$, we need that~$D'$ be strongly pathful in~$D$, so that~$\restrictionMap$ restricts to a lattice isomorphism from a lower interval of~$\AR$ to~$\AR[D']$ by~\cref{thm:latticeMap}, whose inverse we denote by~$\inverseRestrictionMap$ (in other words, $\inverseRestrictionMap(E')$ is the acyclic reorientation of~$D$ whose reversed arcs are exactly the reversed arcs of~$E'$).

\begin{proposition}
\label{prop:restrictionCongruence}
If~$D'$ is strongly pathful in~$D$, then any congruence~$\equiv$ on~$\AR$ restricts to a congruence~$\equiv'$ on~$\AR[D']$ defined by~$E' \equiv' F'$ if and only if~$\inverseRestrictionMap(E') \equiv \inverseRestrictionMap(F')$.
\end{proposition}

\begin{proof}
The congruence~$\equiv$ of~$\AR$ restricts to a congruence of the interval~$[D, \inverseRestrictionMap(\bar D')]$ of~$\AR$, which is isomorphic to~$\AR[D']$ since~$D'$ is strongly pathful in~$D$.
\end{proof}


\subsection{Subropes}
\label{subsec:subropes}

Recall from \cref{subsec:congruenceNormality} that the set~$\con(L)$ of congruences of a lattice~$L$, ordered by refinement, is a distributive lattice.
When~$L$ is congruence uniform, the map sending a join irreducible~$j$ of~$L$ to the join irreducible congruence~$\con(j_\star, j)$ of~$\con(L)$ is a bijection (where~$\con(x,y)$ denotes the minimal congruence such that~$x \equiv y$).
In other words, $\con(L)$ is isomorphic to the set of lower ideals of the \defn{forcing order} on join irreducibles of~$L$, defined by~$j \prec j'$ if $\con(j'_\star, j')$ refines $\con(j_\star, j)$.
Moreover, for a congruence~$\equiv$ of~$L$ corresponding to a lower ideal~$\bb{I}$ of the forcing order,
\begin{itemize}
\item an element of~$L$ is minimal in its $\equiv$-class if and only if all the join irreducibles in its canonical join representation belong to~$\bb{I}$,
\item the canonical joinands of a congruence class~$X$ in~$L/{\equiv}$ are the classes of the canonical joinands of the minimal element in~$X$.
\end{itemize}
Dual statements hold using meets instead of joins.
In view of these statements, understanding the congruences and quotients of a congruence uniform lattice amounts to understanding the forcing order on the join irreducibles of~$L$ and its lower ideals.

\medskip
For acyclic reorientation lattices, the forcing order is not difficult to describe in terms of the ropes of \cref{sec:ropeDiagrams}.
A rope~$\rope \eqdef (u,v, \down, \up)$ is a \defn{subrope} of a rope~$\rope' \eqdef (u', v', \down', \up')$ if and only if~${\{u,v\} \subseteq \{u',v'\} \cup \down' \cup \up'}$ and~$\down \subseteq \down'$ while~$\up \subseteq \up'$.
The \defn{subrope order} is the order on ropes of~$D$ defined by~$\rope \prec \rope'$ if $\rope$ is a subrope of~$\rope'$.
Examples of subrope orders are illustrated in \cref{fig:subropeOrder}.

\begin{figure}
	\centerline{
			\includegraphics[scale=.9]{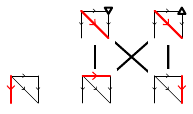} \qquad
			\includegraphics[scale=.9]{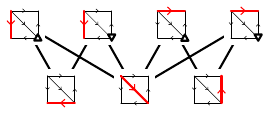} \qquad
			\raisebox{-.5cm}{\includegraphics[scale=.9]{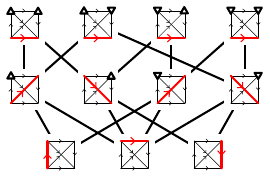}}
	}
	\caption{The subrope orders on the directed acyclic graphs of \cref{fig:canonicalJoinComplexes}. To represent a rope~$(u, v, \down, \up)$, we highlight the arc~$(u,v)$ in red, and we mark the vertices of~$\down$ and~$\up$ by down and up triangles respectively. The rightmost order is the subarc order of~\cite{Reading-arcDiagrams}.}
	\label{fig:subropeOrder}
\end{figure}

\begin{proposition}
\label{prop:subropeOrder}
Assume that~$D$ is skeletal.
For any two join irreducibles~$J$ and~$J'$ of the acyclic reorientation lattice~$\AR$, $J$ forces $J'$ if and only if $\rope_\join(J)$ is a subrope of~$\rope_\join(J')$.
\end{proposition}

\begin{proof}
Here, we could specialize to the acyclic reorientation lattice~$\AR$ the general results on the forcing order among shards for an arbitrary tight hyperplane arrangement~\cite{Reading-latticeCongruences, Reading-PosetRegionsChapter}.
We would in particular recover N.~Reading's description in terms of subarcs of the forcing order among join irreducibles in the weak order on permutations~\cite{Reading-latticeCongruences, Reading-arcDiagrams}.
Let us instead assume this description, and observe that it essentially implies our statement for the acyclic reorientation lattice~$\AR$.
Indeed, consider two join irreducibles~$J$ and~$J'$ of~$\AR$ and let~$\rope \eqdef \rope_\join(J)$ and~$\rope' \eqdef \rope_\join(J')$.
Let~$V'$ denote the transitive support of~$\rope'$ and~$D'$ denote the subgraph of~$D$ induced by~$V'$.

\medskip
Assume first that both endpoints of~$\rope$ belong to~$V'$.
Since~$D$ is filled, $D'$ is a tournament, so that~$\AR[D']$ is isomorphic to the weak order on permutations of~$V'$.
The restriction map~$\restrictionMap$ sends the join irreducibles~$J$ and~$J'$ of~$\AR$ to join irreducibles of~$\AR[D']$.
Since $D'$ is strongly pathful in~$D$, we can transport any lattice congruence~$\equiv$ of~$\AR$ to a lattice congruence~$\equiv'$ of~$\AR[D']$ and \viceversa{} by \cref{prop:extensionCongruence,prop:restrictionCongruence}, preserving the refinement order.
Moreover, observe that~$J$ is contracted in~$\equiv$ if and only if~$\restrictionMap(J)$ is contracted in~$\equiv'$.
We therefore obtain that $J$ forces $J'$ if and only if~$\restrictionMap(J)$ forces $\restrictionMap(J')$.
By N.~Reading's work~\cite{Reading-arcDiagrams}, the latter is equivalent to~$\rope_\join(J)$ being a subrope of~$\rope_\join(J')$ (it is called subarc in~\cite{Reading-arcDiagrams}, we use the term subropes here to avoid confusion with the arcs of directed graphs).

\medskip
Assume now that at least one endpoint of~$\rope$ does not belong to~$V'$.
Observe that the arc~$a$ reversed from~$J_\star$ to~$J$ does not belong to~$D'$, while the arc~$a'$ reversed from~$J'_\star$ to~$J'$ belongs to~$D'$.
By \cref{thm:latticeMap}, the restriction~$\restrictionMap$ is a lattice map, so that its fibers define a lattice congruence~$\equiv$ of~$\AR$.
We have~$J_\star \equiv J$ (since~$a$ does not belong to~$D'$) but~$J_\star' \not\equiv J'$ (since~$a'$ belongs to~$D'$).
It follows that~$J$ does not force~$J'$.
\end{proof}

\begin{corollary}
\label{coro:ropeIdeals}
Assume that~$D$ is skeletal.
The congruence lattice of~$\AR$ is isomorphic to the lattice of lower ideals of the subrope order for~$D$.
\end{corollary}

Observe that if the transitive reduction of~$D$ is a path, then all lower ideals of the subrope order for~$D$ are lower ideals for the subrope order on that path.
In other words, all lattice quotients of~$D$ are lattice quotients of the weak order on permutations.
However, we obtain more lattice quotients as soon as the transitive reduction of~$D$ is not a path.

\medskip
Note that the extension operation of \cref{prop:extensionCongruence} consists in considering a lower ideal of arcs of~$D'$ as a lower ideal of arcs of~$D$, while the restriction operation of \cref{prop:restrictionCongruence} consists in conserving only the ropes supported by the arcs of~$D'$.
These operations are well-defined when~$D'$ is pathful in~$D$ since the ropes of~$D'$ then coincide with the ropes of~$D$ supported by arcs of~$D'$, and the subrope order is the same when regarding these ropes in~$D$ or in~$D'$.

\medskip
\enlargethispage{.4cm}
We denote by~$\bb{I}_\equiv$ the lower ideal of the subrope order corresponding to a congruence~$\equiv$ of~$\AR$, and conversely by~$\equiv_\bb{I}$ the congruence of~$\AR$ corresponding to a lower ideal~$\bb{I}$ of the subrope order.
In other words, $\bb{I}_\equiv$ is the set of ropes~$\rope_\join(J)$ for the join irreducibles~$J$ of~$\AR$ uncontracted by~$\equiv$, and~$\equiv_\bb{I}$ contracts the join irreducibles~$I_\join(\rope)$ for~$\rope$ not in~$\bb{I}$.

\begin{corollary}
\label{coro:minimalElementsIdeal}
Assume that~$D$ is skeletal.
For any congruence~$\equiv$ of~$\AR$,
\begin{itemize}
\item an acyclic reorientation~$E$ of~$D$ is minimal in its $\equiv$-class if and only if~$\diagram_\join(E) \subseteq~\bb{I}_\equiv$,
\item $\AR/{\equiv}$ is isomorphic to the subposet of~$\AR$ induced by $\set{E \in \AR}{\diagram_\join(E) \subseteq~\bb{I}_\equiv}$.
\end{itemize}
A symmetric statement holds for maximal elements and~$\diagram_\meet$.
\end{corollary}

We can finally use bidiagrams to describe the intervals of the quotient lattice.

\begin{corollary}
\label{coro:intervalsQuotient}
Assume that~$D$ is skeletal.
For any congruence~$\equiv$ of~$\AR$, the intervals of~$\AR/{\equiv}$ are in bijection with the rope bidiagrams of~$D$ whose ropes are all~in~$\bb{I}_\equiv$.
\end{corollary}

In connection to the simpliciality of the quotient fans (or equivalently to the simplicity of the quotientopes) defined below, it would be interesting to understand which quotients of~$\AR$ have a regular cover graph (meaning all vertices have the same degree).
For instance, when~$D$ is a tournament, H.~Hoang and T.~M\"utze proved in~\cite{HoangMutze} that the cover graph of~$\AR/{\equiv}$ is regular if and only if the generators (as an upper ideal of the subrope order) of the complement of~$\bb{I}_{\equiv}$ are all of the form~$(u, v, \down, \varnothing)$ or~$(u, v, \varnothing, \up)$.
We hope that the rope interpretation of the congruences of~$\AR$ will help to extend this result for arbitrary skeletal directed acyclic graphs

\begin{problem}
\label{pb:regularCoverGraph}
Characterize the skeletal directed acyclic graphs~$D$ and the congruences~$\equiv$ of~$\AR$ for which the cover graph of~$\AR/{\equiv}$ is regular.
\end{problem}


\subsection{Partial reorientations}
\label{subsec:partialReorientationsCongruence}

We have seen above that the non-crossing rope diagrams (resp.~the rope bidiagrams) are particularly suited to encode the elements (resp.~the intervals) of the lattice quotients of~$\AR$.
Here, we define alternative combinatorial models, based on the observation of \cref{prop:intervalsAreFibers} that any interval of~$\AR$ can be seen as the fiber of a partial acyclic reorientation of~$D$ under the corresponding restriction map.
Namely, for an interval~$I$ of~$\AR/{\equiv}$, define
\begin{itemize}
\item $P_I$ to be the set of arcs which belong to all acyclic reorientations in all classes of~$I$, and
\item $R_I$ to be the transitive reduction of~$P_I$.
\end{itemize}
For a single class~$X \in \AR/{\equiv}$, we write~$P_X$ and~$R_X$ instead of~$P_{\{X\}}$ and~$R_{\{X\}}$.
Define 
\begin{alignat*}{3}
\c{IP}_\equiv & \eqdef \set{P_I}{I \text{ interval of } \AR/{\equiv}}
& \qquad\text{and}\qquad &&
\c{IR}_\equiv & \eqdef \set{R_I}{I \text{ interval of } \AR/{\equiv}},
\\
\c{P}_\equiv & \eqdef \set{P_X}{X \text{ class of } {\equiv}}
& \qquad\text{and}\qquad &&
\c{R}_\equiv & \eqdef \set{R_X}{X \text{ class of } {\equiv}}.
\end{alignat*}
For instance, for the trivial congruence~$=$ (where the congruence classes are all singletons), 
\begin{itemize}
\item a partial acyclic reorientation~$P$ of~$D$ is in~$\c{IP}_{=}$ if and only if $(u,w) \in P$ or~$(w,v) \in P$ for any arc~$(u,v) \in P$ and any~$w$ in between~$u$ and~$v$ in the transitive reduction of~$D$,
\item the elements of~$\c{P}_{=}$ and~$\c{R}_{=}$ are the acyclic reorientations of~$D$ and their \mbox{transitive reductions.}
\end{itemize}
The criterion for~$\c{IP}_{=}$ is a specialization of \cref{prop:minimalMaximalInFiber} in the situation where~$D$ is skeletal.
It generalizes the classical criterion of~\cite{BjornerWachs} for the integer posets corresponding to intervals of the weak order, see also \cite{ChatelPilaudPons}.
We will see further relevant examples of these families~$\c{P}_\equiv$, $\c{R}_\equiv$, $\c{IP}_\equiv$, and~$\c{IR}_\equiv$ of partial acyclic reorientations of~$D$ in \cref{subsec:coherentCongruencesPrincipalCongruences}.

\medskip
For a partial acyclic reorientation~$P$ of~$D$, define~$P^\join \eqdef P \ssm D$ and~$P^\meet \eqdef P \cap D$.
Order the set of partial acyclic reorientations of~$D$ by~$P \le Q$ if and only if~$P^\join \supseteq Q^\join$ and~$P^\meet \subseteq Q^\meet$.
The following generalizes the motivating observation of~\cite{ChatelPilaudPons}.

\begin{proposition}
\label{prop:partialReorientationsCongruence}
Assume that~$D$ is skeletal.
For any congruence~$\equiv$ of~$\AR$,
\begin{itemize}
\item the quotient~$\AR/{\equiv}$ is isomorphic to~$(\c{P}_\equiv, \le)$,
\item the lattice of intervals of~$\AR/{\equiv}$ is isomorphic to~$(\c{IP}_\equiv, \le)$.
\end{itemize}
\end{proposition}

\begin{proof}
Observe that
\begin{itemize}
\item for two acyclic reorientations~$E$ and~$F$ of~$D$, we have~$E \le F \! \iff \! E^\join \subseteq F^\join \! \iff \! E^\meet \supseteq F^\meet$,
\item for an interval~$I \eqdef [E, F]$ of~$\AR$, we have~$P_I = E^\join \sqcup F^\meet$.
\end{itemize}
Therefore, for two intervals~$I \eqdef [E,F]$ and~$I' \eqdef [E',F']$, we have
\[
I \le I' \iff E \le E' \text{ and } F \le F' \iff E^\join \subseteq E'^\join \text{ and } F^\meet \supseteq F'^\meet \iff P_I \le P_{I'}.
\]
This shows the second point of the statement.
The first point follows by specializing to singleton intervals.
\end{proof}

These partial acyclic reorientations provide a different perspective on the elements and the intervals of~$\AR$.
For instance, the degree of an $\equiv$-class~$X$ in the Hasse diagram of~$\AR/{\equiv}$ is the number of arcs of~$R_X$.
\cref{pb:regularCoverGraph} can thus be reformulated as follows.

\begin{problem}
\label{pb:partialReorientationsAcyclic}
Characterize the skeletal directed acyclic graphs~$D$ and the congruences~$\equiv$ of~$\AR$ for which all partial acyclic reorientations of~$\c{R}_\equiv$ are forests.
\end{problem}


\subsection{Coherent congruences and principal congruences}
\label{subsec:coherentCongruencesPrincipalCongruences}

The prototypical lattice congruence of the weak order on~$\f{S}_n$ is the sylvester congruence~\cite{HivertNovelliThibon-algebraBinarySearchTrees}, whose quotient is the Tamari lattice~\cite{Tamari,HuangTamari}.
The sylvester congruence can be defined equivalently as
\begin{enumerate}[(i)]
\item the congruence where each class is the set of linear extensions of a binary tree (labeled in inorder and oriented toward its root),
\item the transitive closure of the rewriting rule~$UacVbW \equiv UcaVbW$ for letters~$1 \le a < b < c \le n$.
\end{enumerate}
It follows from~(i) that the sylvester class posets are the standard binary search trees, and from~(ii) that a permutation is minimal (resp.~maximal) in its sylvester class if and only if it avoids the pattern~$312$ (resp.~$132$).
The sylvester congruence was extended in~\cite{Reading-CambrianLattices} to Cambrian congruences and in~\cite{PilaudPons-permutrees} to permutree congruences.
We next define analogues of these congruences for the acyclic reorientation lattices.

\para{Coherent congruences}
Fix a pair~$(\decorationDown, \decorationUp)$ of arbitrary subsets of~$V$.
We denote by~$\bb{I}_{(\decorationDown, \decorationUp)}$ the set of ropes~$(u,v, \down, \up)$ of~$D$ such that~$\down \subseteq \decorationDown$ and~$\up \subseteq \decorationUp$.
Note that the intersection of~$\decorationDown$ or~$\decorationUp$ with the set~$L$ of leaves of the transitive reduction of~$D$ is irrelevant for the definition of~$\bb{I}_{(\decorationDown, \decorationUp)}$.
Observe also that the set~$\bb{I}_{(\decorationDown, \decorationUp)}$ is clearly a lower ideal of the subrope order whose
\begin{itemize}
\item generators are the ropes~$(u, v, \down, \up)$ with~${u, v \in L \cup \big( V \ssm (\decorationDown \cup \decorationUp) \big)}$, and~$\down \subseteq \decorationDown$ while~${\up \subseteq \decorationUp}$,
\item cogenerators are the ropes~$(u, v, \{w\}, \varnothing)$ for~$w \notin \decorationDown$ and~$(u,v , \varnothing, \{w\})$ for~$w \notin \decorationUp$.
\end{itemize}
We denote by~$\equiv_{(\decorationDown, \decorationUp)}$ the corresponding congruence of~$\AR$.
We say that~$\equiv_{(\decorationDown, \decorationUp)}$ is a \defn{coherent congruence}.
For instance,
\begin{itemize}
\item $\bb{I}_{(V,V)}$ contains all ropes on~$D$, hence~$\equiv_{(V,V)}$ has one class for each acyclic reorientation~of~$D$,
\item $\bb{I}_{(\varnothing, \varnothing)}$ contains only the ropes~$(u,v,\varnothing,\varnothing)$ for~$(u,v)$ in the transitive reduction of~$D$, hence~$\equiv_{(\varnothing,\varnothing)}$ has one class for each acyclic reorientation of the transitive reduction~of~$D$.
\end{itemize}
More interestingly, we define
\begin{itemize}
\item the \defn{sylvester congruence} of~$\AR$ as the coherent congruence~$\equiv_{(V,\varnothing)}$, and the \defn{Tamari lattice} of~$D$ as the quotient~$\AR/{\equiv_{(V,\varnothing)}}$, generalizing~\cite{HivertNovelliThibon-algebraBinarySearchTrees, Tamari},
\item the \defn{Cambrian congruences} of~$\AR$ as the coherent congruences~$\equiv_{(\decorationDown, \decorationUp)}$ such that~${\decorationDown \sqcup \decorationUp = V}$, and the \defn{Cambrian lattices} of~$D$ as the corresponding quotients of~$\AR$, generalizing~\cite{Reading-CambrianLattices}.
\end{itemize}
For instance, \cref{fig:sylvesterCongruences} illustrates the partitions of~$\AR$ into sylvester classes and the Tamari lattices for the acyclic reorientation lattices of \cref{fig:canonicalJoinComplexes}.

\begin{figure}[p]
	\centerline{
		\begin{tabular}{c@{\qquad}c@{\qquad}c}
			\includegraphics[scale=.9]{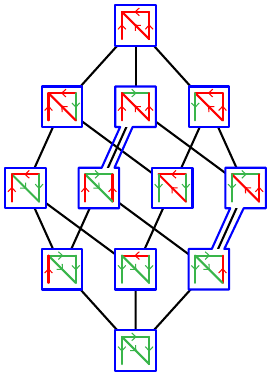} &
			\includegraphics[scale=.9]{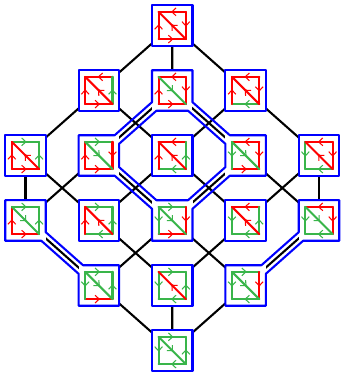} &
			\includegraphics[scale=.9]{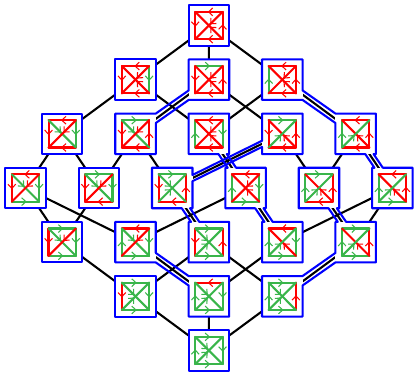}
			\\[.2cm]
			\includegraphics[scale=.9]{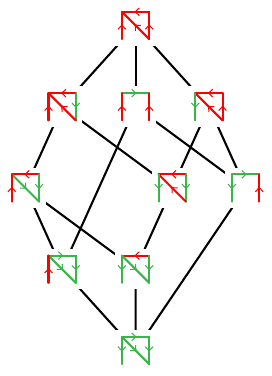} &
			\includegraphics[scale=.9]{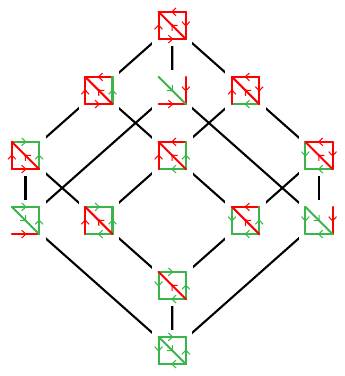} &
			\includegraphics[scale=.9]{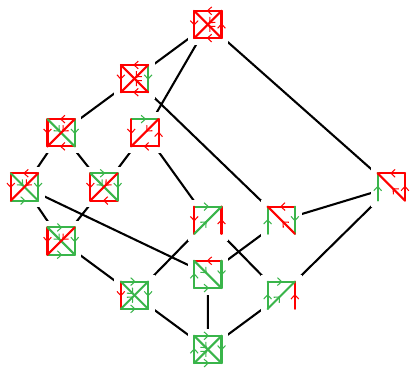}
			\\[.2cm]
			\includegraphics[scale=.9]{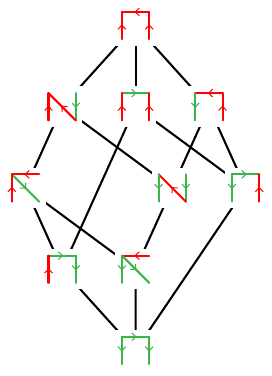} &
			\includegraphics[scale=.9]{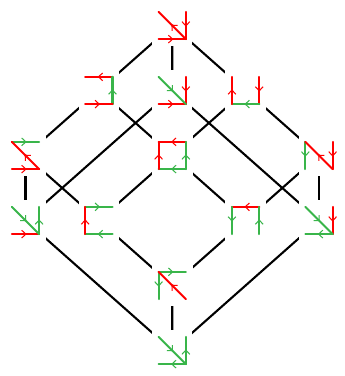} &
			\includegraphics[scale=.9]{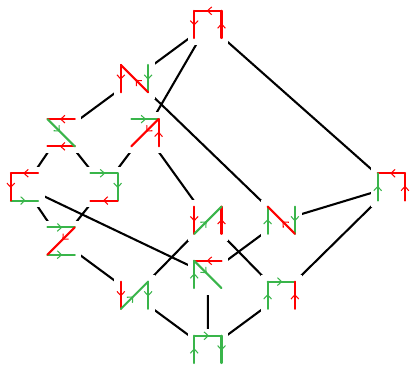}
		\end{tabular}
	}
	\caption{The sylvester congruences~$\equiv_{(V,\varnothing)}$ and the Tamari lattices~$\AR/{\equiv_{(V,\varnothing)}}$ for the acyclic reorientation lattices of \cref{fig:canonicalJoinComplexes}. The first line shows the sylvester classes as blue bubbles. The second and third lines show the Tamari lattices, where each sylvester class~$X$ is represented by~$P_X$ on the second line and~$R_X$ on the third line. The rightmost congruence is the classical sylvester congruence on the weak order, its quotient is the classical Tamari lattice, and the partial reorientations~$R_X$ are standard binary search trees (to see it, just redraw~$R_X$ with green arcs pointing northeast and red arcs pointing northwest).}
	\label{fig:sylvesterCongruences}
\end{figure}

\para{Three problems on Cambrian congruences}
Before studying coherent congruences in general, let us already observe that the Tamari and Cambrian lattices do not always behave as in the classical situation of the weak order.
This is illustrated in particular by the following three problems, verified by computer experiments on all skeletal directed acyclic graphs up to $6$ vertices.
The first two problems are specific cases of \cref{pb:regularCoverGraph,pb:partialReorientationsAcyclic}.

\begin{problem}
\label{pb:TamariRegular}
Prove the equivalence of the following assertions for a skeletal directed acyclic~graph~$D$:
\begin{enumerate}[(i)]
\item $D$ has no induced subgraph isomorphic to\,\smash{\raisebox{-.25cm}{\includegraphics[scale=.9]{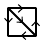}}}\,or\,\smash{\raisebox{-.25cm}{\includegraphics[scale=.9]{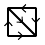}}},
\item the cover graph of the Tamari lattice of~$D$ is regular,
\item the partial acyclic reorientations~$R_X$ for the sylvester classes~$X$ are all forests.
\end{enumerate}
\end{problem}

\begin{problem}
\label{pb:allCambrianRegular}
Prove the equivalence of the following assertions for a skeletal directed acyclic~graph~$D$:
\begin{enumerate}[(i)]
\item $D$ has no induced subgraph isomorphic to\,\smash{\raisebox{-.25cm}{\includegraphics[scale=.9]{digraph2}}}\,or\,\smash{\raisebox{-.25cm}{\includegraphics[scale=.9]{digraph3}}}\,or\,\smash{\raisebox{-.25cm}{\includegraphics[scale=.9]{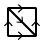}}},
\item the cover graphs of all Cambrian lattices of~$D$ are regular,
\item the partial acyclic reorientations~$R_X$ for the Cambrian classes~$X$ are all forests.
\end{enumerate}
\end{problem}

\begin{problem}
\label{pb:allCambrianSameGraphs}
Prove the equivalence of the following assertions for a skeletal directed acyclic~graph~$D$:
\begin{enumerate}[(i)]
\item $D$ has no induced subgraph isomorphic to\,\smash{\raisebox{-.25cm}{\includegraphics[scale=.9]{digraph1}}},
\item all Cambrian lattices of~$D$ have the same number of elements,
\item the cover graphs of the Cambrian lattices of~$D$ are all isomorphic (as undirected graphs).
\end{enumerate}
\end{problem}

\para{Combinatorial properties of coherent congruences}
We now provide analogues for the coherent congruences of the classical properties of the sylvester~\cite{HivertNovelliThibon-algebraBinarySearchTrees}, Cambrian~\cite{Reading-CambrianLattices} and permutree~\cite{PilaudPons-permutrees} congruences recalled above.
We start by the following analogue of the rewriting rule of the sylvester congruence.

\begin{proposition}
Assume that~$D$ is skeletal.
For any~$\decorationDown, \decorationUp \subseteq V$ and any acyclic reorientation~$E$~of~$D$, the $\equiv_{(\decorationDown, \decorationUp)}$-class of~$E$ is preserved by flipping any arc~$(u,v)$ of the transitive reduction~$E$ such that~$\down_{u,v}^E \not\subseteq \decorationDown$ or~$\up_{u,v}^E \not\subseteq \decorationUp$.
Moreover, the congruence~$\equiv_{(\decorationDown, \decorationUp)}$ is the transitive closure of these~flips.
\end{proposition}

\begin{proof}
It follows from the definition of the join irreducibles preserved by~$\equiv_{(\decorationDown, \decorationUp)}$.
\end{proof}

Next, we give an analogue for the coherent congruences of the pattern avoidance property of the minimal and maximal permutations in sylvester congruence classes.

\begin{proposition}
\label{prop:decorationCongruence}
Assume that~$D$ is skeletal.
For any~$\decorationDown, \decorationUp \subseteq V$ and any acyclic reorientation~$E$ of~$D$, the following statements are equivalent:
\begin{enumerate}[(i)]
\item $E$ minimal (resp.~maximal) in its~$\equiv_{(\decorationDown, \decorationUp)}$-congruence class,
\item $\down_{u,v}^E \subseteq \decorationDown$ and $\up_{u,v}^E \subseteq \decorationUp$ for any arc~$(u,v)$ of~$D$ reversed (resp.~unreversed) in the transitive reduction of~$E$,
\item $\down_{u,v}^E \subseteq \decorationDown$ and $\up_{u,v}^E \subseteq \decorationUp$ for any arc~$(u,v)$ of~$D$ reversed (resp.~unreversed) in~$E$.
\end{enumerate}
\end{proposition}

\begin{proof}
We focus on minimal elements, the proof for maximal elements is symmetric.
By \cref{coro:minimalElementsIdeal}, $E$ is minimal in its $\equiv_{(\decorationDown, \decorationUp)}$-class if and only if~$\diagram_\join(E) \subseteq \bb{I}_{(\decorationDown, \decorationUp)}$.
Since~$\diagram_\join(E)$ is formed by the ropes~$(u, v, \down_{u,v}^E, \up_{u,v}^E)$ for~$(u,v)$ reversed in the transitive reduction of~$E$, we obtain (i)$\iff$(ii).

\medskip
Assume now that~(ii) holds, consider an arc~$(u,v)$ reversed in~$E$, and let~$w \in \up_{u,v}^E$.
If~$(u,v)$ is in the transitive reduction of~$E$, then~$w \in \decorationUp$ by~(ii).
Otherwise, let~$u' \in V$ be the last vertex before~$u$ in the directed path from~$v$ to~$u$ in the transitive reduction of~$E$.
Since~$(u,w)$ is an arc of~$E$, we obtain that~$u' \ne w$ and that~$(w,u')$ is reversed in~$E$ (as otherwise the arcs~$(u,w)$, $(w,u')$ and~$(u',u)$ would form a directed cycle in~$E$).
Therefore, $w \in \up_{u,u'}^E \subseteq \decorationUp$.
We conclude that~$\up_{u,v}^E \subseteq \decorationUp$ and by symmetry that~$\down_{u,v}^E \subseteq \decorationDown$, so that (ii)$\iff$(iii).
\end{proof}

We now focus on the partial acyclic reorientations arising from coherent congruences.
Let us abbreviate~$\c{P}_{\equiv_{(\decorationDown, \decorationUp)}}$ into $\c{P}_{(\decorationDown, \decorationUp)}$ and similarly for~$\c{R}_{(\decorationDown, \decorationUp)}$, $\c{IP}_{(\decorationDown, \decorationUp)}$, and~$\c{IR}_{(\decorationDown, \decorationUp)}$.
The next statement characterizes the partial acyclic reorientations in~$\c{IP}_{(\decorationDown, \decorationUp)}$, generalizing~\cite[Sect.~2.3.2]{ChatelPilaudPons} for the permutree interval posets.
Recall from \cref{subsec:partialReorientationsCongruence} that a partial reorientation~$P$ of~$D$ belongs to~$\c{IP}_=$ (\ie corresponds to an interval of~$\AR$) if and only if $(u,w) \in P$ or~$(w,v) \in P$ for any arc~$(u,v) \in P$ and any~$w$ in between~$u$ and~$v$ in the transitive reduction of~$D$.

\begin{proposition}
\label{prop:intervalPosetsCoherentCongruences}
Assume that~$D$ is skeletal.
For any~$\decorationDown, \decorationUp \subseteq V$, the following assertions are equivalent for a partial acyclic reorientation~$P$ of~$D$:
\begin{enumerate}[(i)]
\item $P$ belongs to~$\c{IP}_{(\decorationDown, \decorationUp)}$,
\item for any arc~$(u,v) \in P$ and any~$w$ in between~$u$ and~$v$ in the transitive reduction of~$D$, we have $(u,w) \in P$ or~$(w,v) \in P$, and moreover $(u,w) \in P$ if~$w \notin \decorationDown$, and $(w,v) \in P$~if~$w \notin \decorationUp$.
\end{enumerate}
\end{proposition}

\begin{proof}
Consider first the partial reorientation~$P_I$ of~$D$ corresponding to an interval~$I$ of~$\AR/{\equiv}$, and consider~$(u,v) \in P_I$ and~$w$ in between~$u$ and~$v$ in the transitive reduction of~$D$.
Since any interval of~$\AR/{\equiv}$ comes from an interval of~$\AR$, we have~$P \in \c{IP}_=$, so that~$(u,w) \in P_I$ or~$(w,v) \in P_I$.
Assume for instance that~$(u,v) \in D$ and that~$(u,w) \notin P_I$, and consider the maximal acyclic reorientation~$E$ of~$D$ that agrees with~$P_I$.
We have~$(w,u) \in E$ (since~$(u,w) \notin P_I$) and~$(w,v) \in E$ (since~$(w,v) \in P_I$).
Hence~$w \in \down_{u,v}^E \subseteq \decorationDown$ by \cref{prop:decorationCongruence}, since $E$ is maximal in its~$\equiv_{(\decorationDown, \decorationUp)}$-class and~$(u,v)$ is unreversed in~$E$.
The proof is symmetric when~$(v,u) \in D$ or when~$(w,v) \notin P$.
We conclude that (i) implies (ii)

\medskip
Conversely, consider a partial reorientation~$P$ of~$D$ satisfying (ii).
It follows in particular that~$P \in \c{IP}_=$.
Let~$E$ be the maximal acyclic reorientation of~$D$ that agrees with~$P$.
Consider an arc $(u,v)$ of~$D$ unreversed in~$E$ and $w$ in between~$u$ and~$v$ in the transitive reduction of~$D$.
If~$w \notin \decorationDown$, then~$(u,w) \in P$ so that~$(u,w) \in E$ and~$w \notin \down_{u,v}^E$.
Hence, $ \down_{u,v}^E \subseteq \decorationDown$ and by symmetry $\up_{u,v}^E \subseteq \decorationUp$.
We conclude that~$E$ is maximal in its $\equiv_{(\decorationDown, \decorationUp)}$-class by \cref{prop:decorationCongruence}.
Similarly, the minimal acyclic reorientation of~$D$ that agrees with~$P$ is minimal in its $\equiv_{(\decorationDown, \decorationUp)}$-class.
We conclude that~$P$ defines indeed an interval of~$\AR/{\equiv}$.
\end{proof}

In contrast, we are still missing a criterion similar to~\cite[Sect.~2.3.3]{ChatelPilaudPons} to distinguish the partial reorientations of~$\c{P}_{(\decorationDown, \decorationUp)}$ among that of~$\c{IP}_{(\decorationDown, \decorationUp)}$.

\begin{problem}
\label{pb:descriptionElementPosetsCoherentCongruences}
Describe the partial acyclic reorientations of~$\c{P}_{(\decorationDown, \decorationUp)}$ for any~$\decorationDown, \decorationUp \subseteq V$.
\end{problem}

Similarly, we are still missing an analogue of binary trees (or of permutrees) to characterize the partial acyclic reorientations of~$\c{R}_{(\decorationDown, \decorationUp)}$.
Since the cover graph of the Tamari lattice of~$D$ is not always regular as illustrated in \cref{fig:sylvesterCongruences}\,(middle) and discussed in \cref{pb:TamariRegular}, the partial acyclic reorientations of~$\c{R}_{(\decorationDown, \decorationUp)}$ are not always forests in contrast to the binary trees (or the permutrees) in the classical case.
In view of \cref{prop:restrictionCongruence}, we can however observe that for any subset~$U$ of~$V$ forming a path in the transitive reduction of~$D$ and any congruence class~$X$ of~$\equiv_{(\decorationDown, \decorationUp)}$, the subgraph of~$R_X$ induced by~$U$ is contained in a permutree for the restriction of the decoration~$(\decorationDown, \decorationUp)$ to~$U$.
We leave the precise characterization as an open problem for further research.

\begin{problem}
\label{pb:descriptionPartialReorientations}
Describe the partial acyclic reorientations of~$\c{R}_{(\decorationDown, \decorationUp)}$ for any~$\decorationDown, \decorationUp \subseteq V$.
\end{problem}

\para{Principal congruences}
We finally introduce another family of lattice congruences of~$\AR$, generalizing the sylvester and Cambrian congruences of the classical weak order, that will play an important role in the sequel of this paper.
For a rope~$\rope \eqdef (u, v, \down, \up)$ of~$D$, we denote by~$\bb{I}_\rope$ the principal lower ideal of the subrope order generated by~$\rope$, and by~$\equiv_\rope$ the corresponding lattice congruence of~$\AR$.
We say that~$\equiv_\rope$ is a \defn{principal congruence}.

Denote by~$D'$ the subgraph of~$D$ induced by the transitive support of~$(u,v)$ in~$D$.
Applying the restriction and extension operations of~\cref{prop:restrictionCongruence,prop:extensionCongruence}, the principal congruence~$\equiv_\rope$ can be seen as a Cambrian congruence~$\equiv_\rope'$ on~$D'$.
We therefore completely control the combinatorics of~$\equiv_\rope$.
For instance, the analogue of \cref{pb:descriptionPartialReorientations} for principal congruences has a simple answer: the partial acyclic reorientations of~$\c{R}_\rope \eqdef \c{R}_{\equiv_\rope}$ are precisely the Cambrian trees considered in~\cite{LangePilaud, ChatelPilaud, PilaudPons-permutrees}, for the signature given by the partition~$\down \sqcup \up$ along the directed path joining~$u$ to~$v$ in the transitive reduction of~$D$.


\subsection{Hamiltonian quotients}
\label{subsec:HamiltonianQuotients}

We conclude this section by a brief discussion of an open problem concerning Hamiltonian cycles in quotients of the acyclic reorientation lattice~$\AR$.

\medskip
A classical result, independently discovered in~\cite{Trotter, Johnson, Steinhaus}, states that the graph of the permutahedron admits a Hamiltonian cycle.
In contrast, not all acyclic reorientation graphs admit a Hamiltonian cycle.
Indeed, recall that the parity of the number of reversed arcs defines a proper bipartition of the acyclic reorientation graph.
Hence, a necessary condition for the acyclic reorientation graph to admit a Hamiltonian cycle (and even a Hamiltonian path) is that the number of even acyclic reorientations equals the number of odd acyclic reorientations.
For instance, the acyclic reorientation graph of a $4$-cycle illustrated in \cref{fig:acyclicReorientationPosets}\,(right) has no Hamiltonian path since it has $8$ even acyclic reorientations and $6$ odd acyclic reorientations.
It is conjectured that this condition is also sufficient, but the question still remains open in general to the best of our knowledge.
Importantly for our discussion, it was proved in~\cite{SavageSquireWest} that the acyclic reorientation graph of a chordal graph admits a Hamiltonian cycle, which can be explicitly constructed in a similar way as the classical Gray code for permutations of~\cite{Trotter, Johnson, Steinhaus}.

\medskip
\enlargethispage{.2cm}
Another classical result, proved in~\cite{Lucas, HurtadoNoy}, states that the graph of the associahedron admits a Hamiltonian cycle.
It was proved recently in~\cite{HoangMutze} that the graph of any lattice quotient of the weak order actually admits a Hamiltonian path (the question of the existence of a Hamiltonian cycle remains open in general).
The approach of~\cite{HoangMutze} being largely based on non-crossing arc diagrams, it motivates the following question, which has been positively answered by computer experiments on all lattice congruences of the acyclic reorientation lattices of all skeletal directed acyclic graphs up to $5$ vertices.

\begin{problem}
\label{pb:HamiltonianQuotients}
Assuming that~$D$ is skeletal (thus chordal), do all graphs of lattice quotients of~$\AR$ admit a Hamiltonian cycle?
\end{problem}


\pagebreak

\section{Quotient fans and quotientopes}
\label{sec:quotientFansQuotientopes}

We now switch to the geometric side of this paper.
As originally observed by C.~Greene~\cite{Greene} (see also~\cite[Lem.~7.1]{GreeneZaslavsky}), the acyclic reorientation poset~$\AR$ can be interpreted geometrically on the graphical fan of~$D$ or on the graphical zonotope of~$D$.
When~$D$ is skeletal, we consider the quotient fans of the congruences of~$\AR$ (obtained by glueing regions of the graphical arrangement according to congruence classes) and show that they are normal fans of quotientopes (obtained either as Minkowski sums of associahedra of~\cite{HohlwegLange}, or as Minkowski sums of shard polytopes~\cite{PadrolPilaudRitter} of ropes).


\subsection{Graphical fan, shards, and quotient fans}
\label{subsec:graphicalFanShardsQuotientFans}

Recall that a (polyhedral) \defn{cone} is a subset of~$\R^n$ defined equivalently as the positive span of finitely many vectors, or as the intersection of finitely many linear halfspaces.
Its \defn{faces} are its intersections with its supporting linear hyperplanes, and its \defn{rays} (resp.~\defn{facets}) are its dimension~$1$ (resp.~codimension~$1$) faces.
A (polyhedral) \defn{fan}~$\Fan[]$ is a collection of cones which are closed under faces (if~$\polytope{C} \in \Fan[]$ and~$\polytope{F}$ is a face of~$\polytope{C}$, then~$\polytope{F} \in \Fan[]$) and intersect properly (if~$\polytope{C}, \polytope{C}' \in \Fan[]$, then~$\polytope{C} \cap \polytope{C}'$ is a face of both~$\polytope{C}$ and~$\polytope{C}'$).
The \defn{chambers} (resp.~\defn{walls}, resp.~\defn{rays}) of~$\Fan[]$ are its codimension~$0$ (resp.~codimension~$1$, resp.~dimension~$1$) cones.
The fan~$\Fan[]$ is \defn{complete} if the union of its cones covers~$\R^V$, \defn{essential} if the origin is a cone of~$\Fan[]$, and~\defn{simplicial} if the rays of each cone of~$\Fan[]$ are linearly independent.

\para{Graphical fan}
Here, we work in the vector space~$\R^V$ indexed by the vertex set~$V$ of~$D$.
We denote the standard basis by~$(\b{e}_v)_{v \in V}$, and the characteristic vector of a subset~$U \subseteq V$ by~$\one_U \eqdef \sum_{u \in U} \b{e}_u$.
The \defn{graphical arrangement}~$\c{H}_D$ of~$D \eqdef (V,A)$ is the arrangement containing the hyperplanes $\bb{H}_{uv} \eqdef \set{\b{x} \in \R^V}{x_u = x_v}$ for all arcs~$(u,v) \in A$.
It defines the \defn{graphical fan}~$\Fan$ of~$D$, whose chambers are the closures of the connected components of~$\smash{\R^V \ssm \bigcup_{(u,v) \in A} \bb{H}_{uv}}$.
Note that~$\Fan$ is complete but not essential since all its cones contain the linear subspace~$\K$ generated by the characteristic vectors of the connected components of~$D$.
The intersection~$\Fan \cap \K^\perp$ of~$\Fan$ with the orthogonal complement~$\K^\perp$ of~$\K$ is an essential fan with the same combinatorics as~$\Fan$.
The cones of~$\Fan$ are in bijection with \defn{ordered partitions} of~$D$, \ie pairs~$(\mu, \omega)$~where
\begin{itemize}
\item $\mu$ is a partition of~$V$ where each part induces a connected subgraph of~$D$,
\item $\omega$ is an acyclic reorientation on the quotient graph~$D/\mu$.
\end{itemize}
More precisely, the cone of~$\Fan$ corresponding to the ordered partition~$(\mu, \omega)$ of~$D$ is defined by the inequalities~${x_u \le x_v}$ if there is a directed path in~$\omega$ from the part of~$\mu$ containing~$u$ to the part of~$\mu$ containing~$v$ (in particular, we have the equalities~$x_u = x_v$ if~$u$ and~$v$ belong to the same part of~$\mu$).
In particular, 
\begin{itemize}
\item each acyclic reorientation~$E$ of~$D$ corresponds to a chamber~$\polytope{C}_E$ of~$\Fan$ defined by the inequalities~$x_u \le x_v$ for all arcs~$(u,v)$ of~$E$ (or just that of the transitive reduction of~$E$),
\item each \defn{biconnected subset}~$U$ of~$D$ (\ie non-empty connected subset~$U \subset V$ whose complement~$\bar U$ in its connected component of~$D$ is also non-empty and connected) corresponds to a ray of~$\Fan \cap \K^\perp$ directed by the vector~$\b{r}_U \eqdef |U| \one_{\bar U} - |\bar U| \one_U$.
\end{itemize}
Note that the ray~$\b{r}_U$ belongs to the chamber~$\polytope{C}_E$ if and only if there is no arc oriented from~$\bar U$ to~$U$ in~$E$.
Moreover, the Hasse diagram of the acyclic reorientation poset~$\AR$ is isomorphic to the dual graph of the graphical fan~$\Fan$, oriented in the direction~${\b{\omega}_D \eqdef \sum_{(u,v) \in A} \b{e}_v - \b{e}_u}$.
Note that the graphical arrangement~$\c{H}_D$ and the graphical fan~$\Fan$ only depend on the underlying undirected graph of the directed graph~$D$, but that~$D$ determines the direction~$\b{\omega}_D$.

\medskip
For instance, when~$D$ is the increasing tournament on~$[n]$, the graphical fan~$\Fan$ is the braid fan, defined by the braid arrangement, with all hyperplanes~$\bb{H}_{ij}$ for~$1 \le i < j \le n$.
Its cones correspond to ordered partitions of~$[n]$, its regions to permutations of~$[n]$, its rays to proper subsets of~$[n]$, and its dual graph is isomorphic to the Hasse diagram of the weak order on~$\f{S}_n$.
Some examples of graphical fans are represented in \cref{fig:graphicalArrangementsIntro,fig:graphicalArrangements}.

\begin{figure}
	\centerline{
		\begin{tabular}{ccc}
			\includegraphics[scale=.5]{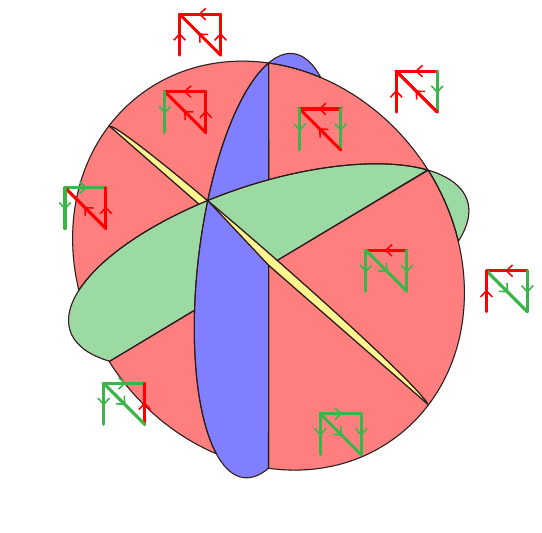} &
			\includegraphics[scale=.5]{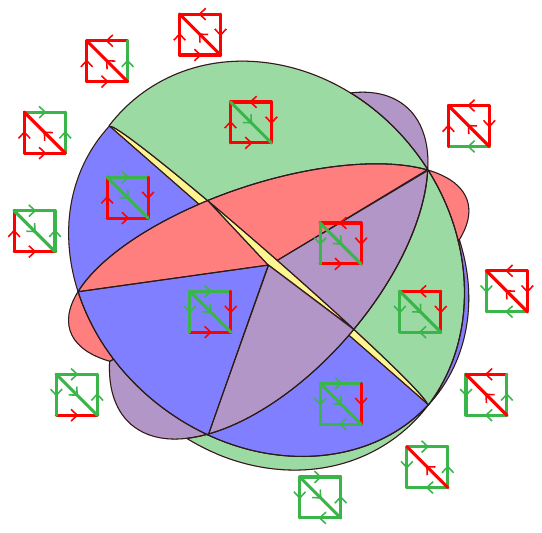} &
			\includegraphics[scale=.5]{graphicalArrangementAcyclicReorientations7}
			\\
			\includegraphics[scale=.75]{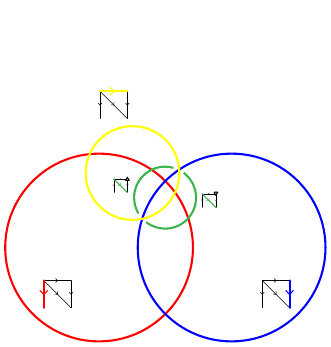} \qquad &
			\includegraphics[scale=.75]{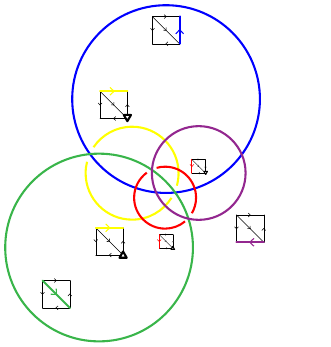} \qquad &
			\includegraphics[scale=.75]{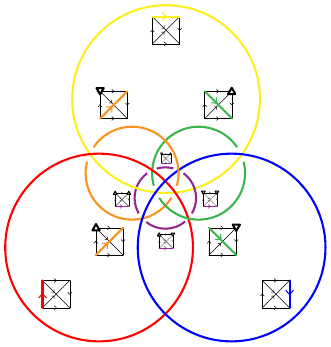}
		\end{tabular}
	}
	\caption{The graphical arrangements for the directed acyclic graphs of \cref{fig:canonicalJoinComplexes}. Their dual graphs oriented appropriately are isomorphic to the Hasse diagrams of the corresponding acyclic reorientation lattices. On top, the regions are labeled by the corresponding acyclic reorientations and the hyperplanes are colored according to the corresponding arc. On bottom, the arrangements are intersected with the unit circle and projected stereographically from the chamber corresponding to the reversed reorientation~$\bar D$, and the hyperplanes are decomposed into shards labeled by the corresponding ropes. The rightmost arrangement is the braid arrangement. The middle fan is not simplicial while the other two are.}
	\vspace{-.2cm}
	\label{fig:graphicalArrangements}
\end{figure}

\medskip
\enlargethispage{.3cm}
Observe that the graphical fan~$\Fan$ is not always simplicial.
Recall that we say that~$D$ is chordful if its underlying undirected graph~$G$ is, meaning that any cycle induces a clique.
The following statement is illustrated in \cref{fig:graphicalArrangementsIntro,fig:graphicalArrangements,fig:simplicialTightSupersolvable}.
It is explicitely stated in~\cite[Rem.~6.2]{Kim} and~\cite[Prop.~5.2]{PostnikovReinerWilliams}, but the proof is omitted.

\begin{proposition}
\label{prop:simplicialGraphicalFan}
The graphical fan~$\Fan$ is simplicial if and only if~$D$ is chordful.
\end{proposition}

\begin{proof}
Observe first that a region of~$\Fan$ is simplicial if and only if the transitive reduction of the corresponding acyclic reorientation of~$D$ is a forest.

Assume that~$\Fan$ is not simplicial, so that there exists an acyclic reorientation~$E$ of~$D$ whose transitive reduction is not a forest.
Therefore, the transitive reduction of~$E$ contains an (undirected) cycle~$C$.
This cycle cannot induce a tournament of~$D$, otherwise it would induce a tournament of~$E$ and one of the arcs of~$C$ would not be in the transitive reduction of~$E$.

Conversely, assume that~$D$ is not chordful, and let~$C$ be an (undirected) cycle of~$D$ such that two vertices~$u$ and~$v$ of~$C$ are not adjacent in~$D$.
Let~$X$ and~$Y$ denote the two two connected components of~$C \ssm \{u,v\}$.
Consider any linear ordering~$\prec$ of~$V$ such that the vertices of~$X$ arrive first, then~$u$ and~$v$, then the vertices of~$Y$, and then all vertices of~$V$ not in~$C$.
Let~$x$ be the minimum element of~$X$ for~$\prec$ and~$y$ be the maximum element of~$Y$ for~$\prec$.
Let~$E$ be the acyclic reorientation of~$D$ where all arcs are increasing for~$\prec$.
Then the transitive closure of~$E$ contains a path from~$x$ to~$y$ passing through~$u$, and a path from~$x$ to~$y$ passing through~$v$, and these two paths cannot coincide since there is no arc in~$D$ connecting~$u$ and~$v$.
Thus, the transitive closure of~$E$ is not a forest, so that~$\Fan$ is not simplicial.
\end{proof}

\vspace{-.2cm}
\para{Shards and quotient fan}
\enlargethispage{.4cm}
Assume now that~$D$ is skeletal as in \cref{sec:ropeDiagrams,sec:congruences}, so that the acyclic reorientation poset~$\AR$ is a congruence uniform lattice by \cref{prop:congruenceUniform}.
The ropes of~$D$ provide a natural way to decompose the hyperplanes of~$\Arrang$ into pieces.
Namely, the \defn{shard}~$\shard_\rope$ associated to a rope~$\rope \eqdef (u,v, \down, \up)$ of~$D$ is
\[
\shard_\rope \eqdef \set{\b{x} \in \R^V}{x_{w} \le x_u = x_v \le x_{w'} \text{ for any } w \in \down \text{ and } w' \in \up}.
\]
Some examples of shards are represented on the bottom of \cref{fig:graphicalArrangements}.

\begin{figure}
	\centerline{
		\begin{tabular}{ccc}
			\includegraphics[scale=.5]{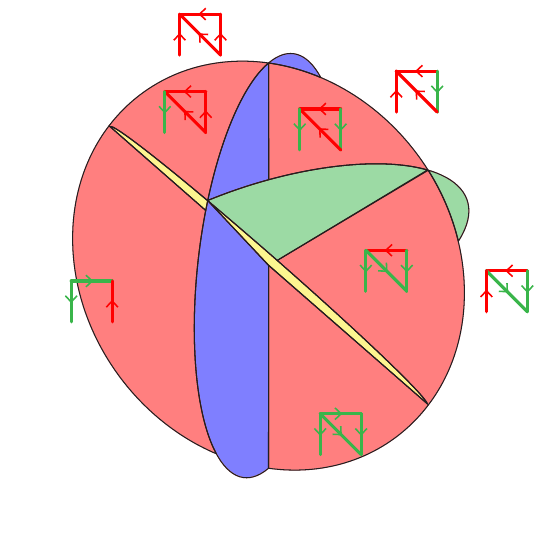} &
			\includegraphics[scale=.5]{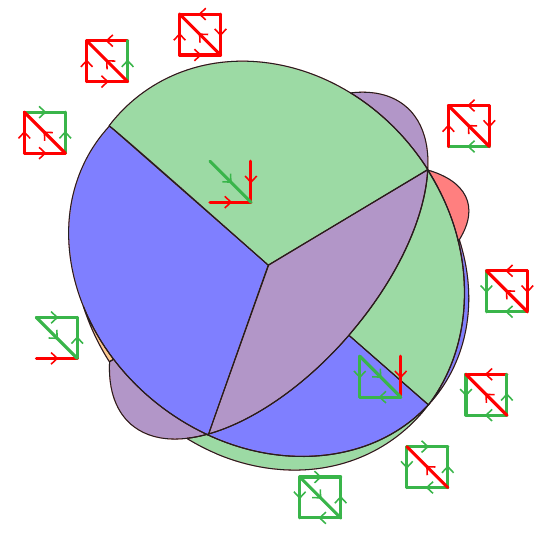} &
			\includegraphics[scale=.5]{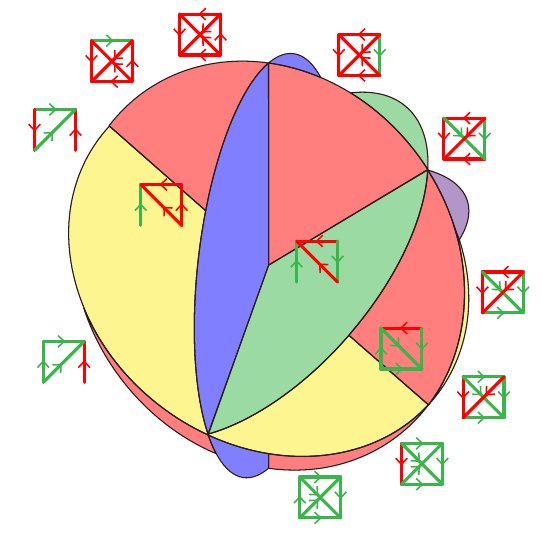}
			\\
			\includegraphics[scale=.75]{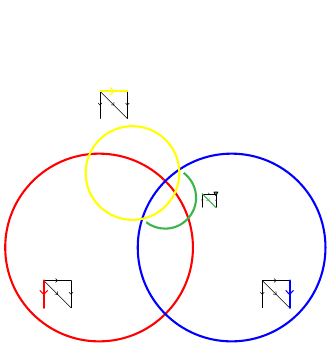} &
			\includegraphics[scale=.75]{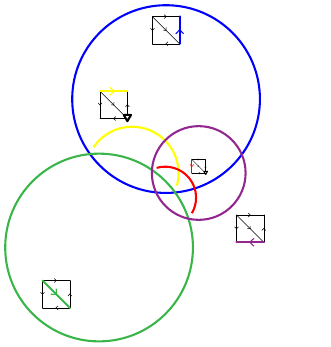} &
			\includegraphics[scale=.75]{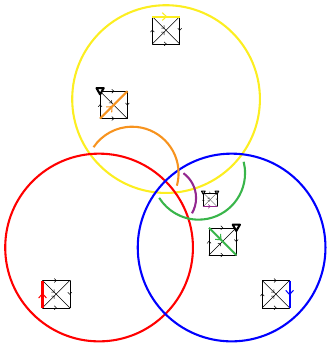}
		\end{tabular}
	}
	\caption{The sylvester fans for the directed acyclic graphs of \cref{fig:canonicalJoinComplexes}. Their dual graphs oriented appropriately are isomorphic to the Hasse diagrams of the corresponding Tamari lattices represented in \cref{fig:sylvesterCongruences}. On top, the chambers are labeled by the corresponding partial acyclic reorientations and the shards are colored according to the corresponding arc. On bottom, the fans are intersected with the unit circle and projected stereographically from the chamber corresponding to the reversed reorientation~$\bar D$, and the shards labeled by the corresponding ropes. The rightmost fan is the classical sylvester fan. The middle fan is not simplicial while the other two are.}
	\vspace{-.2cm}
	\label{fig:quotientFans}
\end{figure}

\medskip
For a congruence~$\equiv$ of~$\AR$, the \defn{quotient fan}~$\Fan[\equiv]$ is the fan defined equivalently as follows:
\begin{itemize}
\item the chambers of~$\Fan[\equiv]$ are obtained by glueing the chambers of the graphical arrangement of~$D$ corresponding to acyclic reorientations in the same congruence class of~$\equiv$,
\item the union of the walls of~$\Fan[\equiv]$ is the union of the shards~$\shard_\rope$ for~$\rope$ in the rope ideal~$\bb{I}_\equiv$.
\end{itemize}
The fact that these two descriptions coincide and indeed define a fan was proved by N.~Reading~\cite{Reading-latticeCongruences,Reading-HopfAlgebras} in the context of congruences of the lattice of regions of a hyperplane arrangement tight with respect to its base region (see \cref{sec:posetRegions} for definitions and details).
Note~that
\begin{itemize}
\item each $\equiv$-class~$X$ corresponds to a chamber of the quotient fan~$\Fan[\equiv]$ defined by the inequalities~$x_u \le x_v$ for all arcs~$(u,v)$ of the partial acyclic reorientation~$P_X$ (or just that of~$R_X$),
\item a biconnected subset~$U$ of~$D$ corresponds to a ray of the quotient fan~$\Fan[\equiv]$ directed by the vector~$\b{r}_U \eqdef |U| \one_{\bar U} - |\bar U| \one_U$ if and only if the subrope ideal~$\bb{I}_\equiv$ contains all ropes of the form~$(u, v, \varnothing, \up)$ with~$u, v \in U$ and~$\up \cap U = \varnothing$, and of the form~$(u, v, \down, \varnothing)$ with~$u, v \notin U$ and~$\down \subseteq U$. In particular, for a coherent congruence~$\equiv_{(\decorationDown, \decorationUp)}$, the ray~$\b{r}_U$ of~$\Fan$ is a ray of~$\Fan[(\decorationDown, \decorationUp)]$ if and only if $u,v \in U$ implies $w \notin \decorationDown \ssm U$ and $u,v \notin U$ implies $w \notin \decorationUp \cap U$, for any~$u, v, w \in V$ such that~$w$ appears along a directed path in~$D$ joining~$u$ to~$v$. (These two observations can be shown mimicking the approach of~\cite[Sect.~3.1]{AlbertinPilaudRitter}.) 
\end{itemize}
Moreover, the Hasse diagram of the quotient~$\AR/{\equiv}$ is isomorphic to the dual graph of the quotient fan~$\Fan[\equiv]$, oriented in the direction~$\b{\omega}_D \eqdef \sum_{(u,v) \in A} \b{e}_v - \b{e}_u$.
Similarly to \cref{prop:simplicialGraphicalFan}, it would be interesting to characterize which of these quotient fans are simplicial, which is a reformulation of \cref{pb:regularCoverGraph} for arbitrary congruences and \cref{pb:TamariRegular,pb:allCambrianRegular} for Cambrian~congruences.

\medskip
For instance, when $D$ is the increasing tournament on~$[n]$ and $\equiv$ is the sylvester congruence~\cite{HivertNovelliThibon-algebraBinarySearchTrees}, the quotient fan~$\Fan[\equiv]$ is the sylvester fan.
Its cones correspond to Schr\"oder trees on~$[n]$, its chambers to binary trees on~$[n]$, its rays to intervals of~$[n]$, and its dual graph is isomorphic to the Hasse diagram of the Tamari lattice on binary trees on~$[n]$.
Similar combinatorial descriptions in terms of Cambrian trees and permutrees hold for the quotient fans of the Cambrian congruences~\cite{Reading-latticeCongruences, Reading-CambrianLattices} and of the permutree congruences~\cite{PilaudPons-permutrees} of a tournament.

\medskip
The \defn{sylvester fan} of~$D$ is the quotient fan~$\Fan[(V,\varnothing)]$ of the sylvester congruence~$\equiv_{(V,\varnothing)}$.
Note that the rays of~$\Fan[(V,\varnothing)]$ correspond to biconnected subsets of~$D$ which are connected in the transitive reduction of~$D$.
Some examples of sylvester fans are represented in \cref{fig:quotientFans}.
As suggested in \cref{pb:TamariRegular} and illustrated in \cref{fig:quotientFans}\,(middle), the sylvester fan of~$D$ is not always a simplicial.
The \defn{Cambrian fans} of~$D$ are the quotient fans of the Cambrian congruences~$\equiv_{(\decorationDown, \decorationUp)}$ with~$\decorationDown \sqcup \decorationUp = V$.

\medskip
Note that the quotient fans behave properly with respect to the restriction and contraction operations of \cref{subsec:restrictionsExtensionsCongruences}.
Namely,
\begin{itemize}
\item if~$\equiv$ extends~$\equiv'$ as in \cref{prop:extensionCongruence}, then~$\Fan[\equiv]$ is a product of~$\Fan[\equiv]$ with a linear subspace,
\item if~$\equiv$ restricts to~$\equiv'$ as in \cref{prop:restrictionCongruence}, then~$\Fan[\equiv']$ is a section of~$\Fan[\equiv]$ by a linear subspace.
\end{itemize}


\subsection{Graphical zonotope, associahedra, shard polytopes, and quotientopes}
\label{subsec:graphicalZonotopeAssociahedraShardPolytopesQuotientopes}

A \defn{polytope} is a subset of~$\R^n$ defined equivalently as the convex hull of finitely many points or as a bounded intersection of finitely many closed affine halfspaces.
Its \defn{faces} are its intersections with its supporting affine hyperplanes, and its \defn{vertices} (resp.~\defn{edges}, resp.~\defn{facets}) are its dimension~$0$ (resp.~dimension~$1$, codimension~$1$) faces.
The \defn{normal cone} of a face~$\polytope{F}$ of a polytope~$\polytope{P}$ is the cone of vectors~$\b{v} \in \R^n$ such that~$\polytope{F}$ is the face of~$\polytope{P}$ maximizing the functional~$\b{x} \mapsto \dotprod{\b{v}}{\b{x}}$.
When~$\polytope{P}$ is full-dimensional, the normal cone of~$\polytope{F}$ is generated by the outer normal vectors of the facets of~$\polytope{P}$ containing~$\polytope{F}$.
The \defn{normal fan} of~$\polytope{P}$ is the fan formed by the normal cones of all faces of~$\polytope{P}$.

\medskip
The \defn{Minkowski sum} of two polytopes~$\polytope{P}, \polytope{Q} \subset \R^n$ is the polytope~$\polytope{P} + \polytope{Q} \eqdef \set{\b{p}+\b{q}}{\b{p} \in \polytope{P}, \, \b{q} \in \polytope{Q}}$.
For any~$\b{r} \in \R^n$, the face maximizing the direction~$\b{r}$ on~$\polytope{P} + \polytope{Q}$ is the Minkowski sum of the faces maximizing the direction~$\b{r}$ on~$\polytope{P}$ and~$\polytope{Q}$.
Therefore, 
\begin{itemize}
\item the normal fan of~$\polytope{P} + \polytope{Q}$ is the common refinement of the normal fans of~$\polytope{P}$ and~$\polytope{Q}$,
\item the vertex of~$\polytope{P} + \polytope{Q}$ maximizing a generic~$\b{r}$ is the sum of vertices of~$\polytope{P}$ and~$\polytope{Q}$ maximizing~$\b{r}$,
\item the facet of~$\polytope{P} + \polytope{Q}$ maximizing a ray~$\b{r}$ is defined by~$\smash{\dotprod{\b{r}}{\b{x}} = \max\limits_{\b{p} \in \polytope{P}} \dotprod{\b{r}}{\b{p}} + \max\limits_{\b{q} \in \polytope{Q}} \dotprod{\b{r}}{\b{q}}}$.
\end{itemize}

\para{Graphical zonotope}
Consider the \defn{graphical zonotope}~$\Zono$, defined as the Minkowski sum of the segments~$[\b{e}_u,\b{e}_v]$ for all~${(u,v) \in A}$.
Note that~$\Zono$ is not full-dimensional as it is orthogonal to the linear subspace~$\K$ generated by the characteristic vectors of the connected components of~$D$.
Since the normal fan of a Minkowski sum is the common refinement of the normal fans of the summands, the graphical fan~$\Fan$ is clearly the normal fan of the graphical zonotope~$\Zono$.
Hence, the faces of~$\Zono$ are in bijection with ordered partitions of~$D$.
In particular,
\begin{itemize}
\item each acyclic reorientation~$E$ of~$D$ corresponds to a vertex~$\sum_{(u,v) \in E} \b{e}_v$ of~$\Zono$,
\item each biconnected subset~$U$ of~$D$ corresponds to a facet with inequality~$\dotprod{\one_U}{\b{x}} \ge \iota_U$, where~$\iota_U \eqdef |\set{a \in A}{|a \cap U| = 2}|$ counts the arcs of~$D$ with both endpoints in~$U$.
\end{itemize}
Moreover, the Hasse diagram of the acyclic reorientation poset~$\AR$ is isomorphic to the graph of~$\Zono$, oriented in the direction~${\b{\omega}_D \eqdef \sum_{(u,v) \in A} \b{e}_v - \b{e}_u}$.
Note that the graphical zonotope~$\Zono$ only depends on the underlying undirected graph of the directed graph~$D$, but that~$D$ determines the direction~$\b{\omega}_D$.
Finally, note that by \cref{prop:simplicialGraphicalFan}, the graphical zonotope~$\Zono$ is simple if and only if~$D$ is chordful.

\medskip
\enlargethispage{.3cm}
For instance, when~$D$ is a tournament on~$[n]$, 
the graphical zonotope~$\Zono$ coincides up to a translation of the vector~$\one$ with the classical permutahedron, defined equivalently as
\begin{itemize}
\item the convex hull of the points~$\sum_{i \in [n]} \sigma_i \, \b{e}_i$ for all permutations~$\sigma$ of~$[n]$,
\item the intersection of the hyperplane~$\bigset{\b{x} \in \R^n}{\dotprod{\one}{\b{x}} = \binom{n+1}{2}}$ with the halfspaces ${\bigset{\b{x} \in \R^n}{\dotprod{\one_U}{\b{x}} \ge \binom{|U|+1}{2}}}$ for all proper subsets~$\varnothing \ne U \subsetneq [n]$,
\item the Minkowski sum of the vector~$\one$ and the segments~$[\b{e}_i, \b{e}_j]$ for all~$1 \le i < j \le n$.
\end{itemize}
Some examples of graphical zonotopes are illustrated in \cref{fig:graphicalZonotopesIntro,fig:graphicalZonotopes}.

\pagebreak

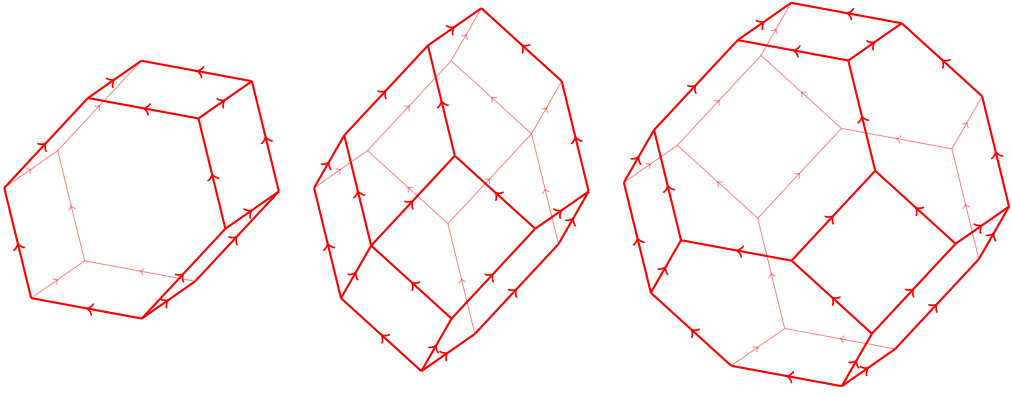
\begin{figure}
	\centerline{

\begin{tikzpicture}%
	[x={(-0.366215cm, -0.789554cm)},
	y={(0.235950cm, -0.590693cm)},
	z={(0.900119cm, -0.166391cm)},
	scale=1.000000,
	back/.style={very thin, opacity=0.5},
	edge/.style={color=red, thick, decoration={markings, mark=at position 0.5 with {\arrow{>}}}},
	facet/.style={fill=red,fill opacity=0},
	vertex/.style={}]
%
%

\coordinate (0.00000, 0.00000, 0.00000) at (0.00000, 0.00000, 0.00000);
\coordinate (0.00000, 0.00000, 1.63299) at (0.00000, 0.00000, 1.63299);
\coordinate (1.33333, -0.94281, 0.00000) at (1.33333, -0.94281, 0.00000);
\coordinate (1.33333, -0.94281, 1.63299) at (1.33333, -0.94281, 1.63299);
\coordinate (1.33333, 0.47140, -0.81650) at (1.33333, 0.47140, -0.81650);
\coordinate (4.00000, 0.00000, 1.63299) at (4.00000, 0.00000, 1.63299);
\coordinate (1.33333, 0.47140, 2.44949) at (1.33333, 0.47140, 2.44949);
\coordinate (2.66667, -0.47140, -0.81650) at (2.66667, -0.47140, -0.81650);
\coordinate (4.00000, 0.00000, 0.00000) at (4.00000, 0.00000, 0.00000);
\coordinate (2.66667, -0.47140, 2.44949) at (2.66667, -0.47140, 2.44949);
\coordinate (2.66667, 0.94281, 0.00000) at (2.66667, 0.94281, 0.00000);
\coordinate (2.66667, 0.94281, 1.63299) at (2.66667, 0.94281, 1.63299);
\draw[edge,postaction={decorate},back] (1.33333, 0.47140, -0.81650) -- (0.00000, 0.00000, 0.00000);
\draw[edge,postaction={decorate},back] (2.66667, -0.47140, -0.81650) -- (1.33333, 0.47140, -0.81650);
\draw[edge,postaction={decorate},back] (2.66667, 0.94281, 0.00000) -- (1.33333, 0.47140, -0.81650);
\draw[edge,postaction={decorate},back] (4.00000, 0.00000, 0.00000) -- (2.66667, 0.94281, 0.00000);
\draw[edge,postaction={decorate},back] (2.66667, 0.94281, 1.63299) -- (2.66667, 0.94281, 0.00000);
\node[vertex] at (1.33333, 0.47140, -0.81650)     {};
\node[vertex] at (2.66667, 0.94281, 0.00000)     {};
\fill[facet] (1.33333, -0.94281, 1.63299) -- (0.00000, 0.00000, 1.63299) -- (0.00000, 0.00000, 0.00000) -- (1.33333, -0.94281, 0.00000) -- cycle {};
\fill[facet] (2.66667, -0.47140, 2.44949) -- (1.33333, -0.94281, 1.63299) -- (0.00000, 0.00000, 1.63299) -- (1.33333, 0.47140, 2.44949) -- cycle {};
\fill[facet] (1.33333, 0.47140, 2.44949) -- (2.66667, 0.94281, 1.63299) -- (4.00000, 0.00000, 1.63299) -- (2.66667, -0.47140, 2.44949) -- cycle {};
\fill[facet] (4.00000, 0.00000, 1.63299) -- (2.66667, -0.47140, 2.44949) -- (1.33333, -0.94281, 1.63299) -- (1.33333, -0.94281, 0.00000) -- (2.66667, -0.47140, -0.81650) -- (4.00000, 0.00000, 0.00000) -- cycle {};
\draw[edge,postaction={decorate}] (0.00000, 0.00000, 1.63299) -- (0.00000, 0.00000, 0.00000);
\draw[edge,postaction={decorate}] (1.33333, -0.94281, 0.00000) -- (0.00000, 0.00000, 0.00000);
\draw[edge,postaction={decorate}] (1.33333, -0.94281, 1.63299) -- (0.00000, 0.00000, 1.63299);
\draw[edge,postaction={decorate}] (1.33333, 0.47140, 2.44949) -- (0.00000, 0.00000, 1.63299);
\draw[edge,postaction={decorate}] (1.33333, -0.94281, 1.63299) -- (1.33333, -0.94281, 0.00000);
\draw[edge,postaction={decorate}] (2.66667, -0.47140, -0.81650) -- (1.33333, -0.94281, 0.00000);
\draw[edge,postaction={decorate}] (2.66667, -0.47140, 2.44949) -- (1.33333, -0.94281, 1.63299);
\draw[edge,postaction={decorate}] (4.00000, 0.00000, 1.63299) -- (4.00000, 0.00000, 0.00000);
\draw[edge,postaction={decorate}] (4.00000, 0.00000, 1.63299) -- (2.66667, -0.47140, 2.44949);
\draw[edge,postaction={decorate}] (4.00000, 0.00000, 1.63299) -- (2.66667, 0.94281, 1.63299);
\draw[edge,postaction={decorate}] (2.66667, -0.47140, 2.44949) -- (1.33333, 0.47140, 2.44949);
\draw[edge,postaction={decorate}] (2.66667, 0.94281, 1.63299) -- (1.33333, 0.47140, 2.44949);
\draw[edge,postaction={decorate}] (4.00000, 0.00000, 0.00000) -- (2.66667, -0.47140, -0.81650);
\node[vertex] at (0.00000, 0.00000, 0.00000)     {};
\node[vertex] at (0.00000, 0.00000, 1.63299)     {};
\node[vertex] at (1.33333, -0.94281, 0.00000)     {};
\node[vertex] at (1.33333, -0.94281, 1.63299)     {};
\node[vertex] at (4.00000, 0.00000, 1.63299)     {};
\node[vertex] at (1.33333, 0.47140, 2.44949)     {};
\node[vertex] at (2.66667, -0.47140, -0.81650)     {};
\node[vertex] at (4.00000, 0.00000, 0.00000)     {};
\node[vertex] at (2.66667, -0.47140, 2.44949)     {};
\node[vertex] at (2.66667, 0.94281, 1.63299)     {};
\end{tikzpicture}
			\raisebox{-.7cm}{

\begin{tikzpicture}%
	[x={(-0.366215cm, -0.789554cm)},
	y={(0.235950cm, -0.590693cm)},
	z={(0.900119cm, -0.166391cm)},
	scale=1.000000,
	back/.style={very thin, opacity=0.5},
	edge/.style={color=red, thick, decoration={markings, mark=at position 0.5 with {\arrow{>}}}},
	facet/.style={fill=red, fill opacity=0},
	vertex/.style={}]
%
%

\coordinate (0.00000, 0.00000, 0.00000) at (0.00000, 0.00000, 0.00000);
\coordinate (0.00000, 1.41421, -0.81650) at (0.00000, 1.41421, -0.81650);
\coordinate (0.00000, 1.41421, 0.81650) at (0.00000, 1.41421, 0.81650);
\coordinate (0.00000, 2.82843, 0.00000) at (0.00000, 2.82843, 0.00000);
\coordinate (1.33333, -0.94281, 0.00000) at (1.33333, -0.94281, 0.00000);
\coordinate (4.00000, 2.82843, 0.00000) at (4.00000, 2.82843, 0.00000);
\coordinate (4.00000, 1.41421, 0.81650) at (4.00000, 1.41421, 0.81650);
\coordinate (1.33333, 1.88562, -1.63299) at (1.33333, 1.88562, -1.63299);
\coordinate (4.00000, 1.41421, -0.81650) at (4.00000, 1.41421, -0.81650);
\coordinate (1.33333, 1.88562, 1.63299) at (1.33333, 1.88562, 1.63299);
\coordinate (1.33333, 3.29983, -0.81650) at (1.33333, 3.29983, -0.81650);
\coordinate (1.33333, 3.29983, 0.81650) at (1.33333, 3.29983, 0.81650);
\coordinate (2.66667, -0.47140, -0.81650) at (2.66667, -0.47140, -0.81650);
\coordinate (2.66667, -0.47140, 0.81650) at (2.66667, -0.47140, 0.81650);
\coordinate (2.66667, 0.94281, -1.63299) at (2.66667, 0.94281, -1.63299);
\coordinate (4.00000, 0.00000, 0.00000) at (4.00000, 0.00000, 0.00000);
\coordinate (2.66667, 0.94281, 1.63299) at (2.66667, 0.94281, 1.63299);
\coordinate (2.66667, 3.77124, 0.00000) at (2.66667, 3.77124, 0.00000);
\draw[edge,postaction={decorate},back] (0.00000, 1.41421, -0.81650) -- (0.00000, 0.00000, 0.00000);
\draw[edge,postaction={decorate},back] (0.00000, 2.82843, 0.00000) -- (0.00000, 1.41421, -0.81650);
\draw[edge,postaction={decorate},back] (1.33333, 1.88562, -1.63299) -- (0.00000, 1.41421, -0.81650);
\draw[edge,postaction={decorate},back] (0.00000, 2.82843, 0.00000) -- (0.00000, 1.41421, 0.81650);
\draw[edge,postaction={decorate},back] (1.33333, 3.29983, -0.81650) -- (0.00000, 2.82843, 0.00000);
\draw[edge,postaction={decorate},back] (1.33333, 3.29983, 0.81650) -- (0.00000, 2.82843, 0.00000);
\draw[edge,postaction={decorate},back] (1.33333, 3.29983, -0.81650) -- (1.33333, 1.88562, -1.63299);
\draw[edge,postaction={decorate},back] (2.66667, 0.94281, -1.63299) -- (1.33333, 1.88562, -1.63299);
\draw[edge,postaction={decorate},back] (2.66667, 3.77124, 0.00000) -- (1.33333, 3.29983, -0.81650);
\node[vertex] at (0.00000, 2.82843, 0.00000)     {};
\node[vertex] at (0.00000, 1.41421, -0.81650)     {};
\node[vertex] at (1.33333, 1.88562, -1.63299)     {};
\node[vertex] at (1.33333, 3.29983, -0.81650)     {};
\fill[facet] (2.66667, -0.47140, 0.81650) -- (2.66667, 0.94281, 1.63299) -- (4.00000, 1.41421, 0.81650) -- (4.00000, 0.00000, 0.00000) -- cycle {};
\fill[facet] (4.00000, 0.00000, 0.00000) -- (2.66667, -0.47140, -0.81650) -- (1.33333, -0.94281, 0.00000) -- (2.66667, -0.47140, 0.81650) -- cycle {};
\fill[facet] (2.66667, 0.94281, 1.63299) -- (1.33333, 1.88562, 1.63299) -- (0.00000, 1.41421, 0.81650) -- (0.00000, 0.00000, 0.00000) -- (1.33333, -0.94281, 0.00000) -- (2.66667, -0.47140, 0.81650) -- cycle {};
\fill[facet] (2.66667, -0.47140, -0.81650) -- (4.00000, 0.00000, 0.00000) -- (4.00000, 1.41421, -0.81650) -- (2.66667, 0.94281, -1.63299) -- cycle {};
\fill[facet] (4.00000, 0.00000, 0.00000) -- (4.00000, 1.41421, 0.81650) -- (4.00000, 2.82843, 0.00000) -- (4.00000, 1.41421, -0.81650) -- cycle {};
\fill[facet] (1.33333, 1.88562, 1.63299) -- (1.33333, 3.29983, 0.81650) -- (2.66667, 3.77124, 0.00000) -- (4.00000, 2.82843, 0.00000) -- (4.00000, 1.41421, 0.81650) -- (2.66667, 0.94281, 1.63299) -- cycle {};
\draw[edge,postaction={decorate}] (0.00000, 1.41421, 0.81650) -- (0.00000, 0.00000, 0.00000);
\draw[edge,postaction={decorate}] (1.33333, -0.94281, 0.00000) -- (0.00000, 0.00000, 0.00000);
\draw[edge,postaction={decorate}] (1.33333, 1.88562, 1.63299) -- (0.00000, 1.41421, 0.81650);
\draw[edge,postaction={decorate}] (2.66667, -0.47140, -0.81650) -- (1.33333, -0.94281, 0.00000);
\draw[edge,postaction={decorate}] (2.66667, -0.47140, 0.81650) -- (1.33333, -0.94281, 0.00000);
\draw[edge,postaction={decorate}] (4.00000, 2.82843, 0.00000) -- (4.00000, 1.41421, 0.81650);
\draw[edge,postaction={decorate}] (4.00000, 2.82843, 0.00000) -- (4.00000, 1.41421, -0.81650);
\draw[edge,postaction={decorate}] (4.00000, 2.82843, 0.00000) -- (2.66667, 3.77124, 0.00000);
\draw[edge,postaction={decorate}] (4.00000, 1.41421, 0.81650) -- (4.00000, 0.00000, 0.00000);
\draw[edge,postaction={decorate}] (4.00000, 1.41421, 0.81650) -- (2.66667, 0.94281, 1.63299);
\draw[edge,postaction={decorate}] (4.00000, 1.41421, -0.81650) -- (2.66667, 0.94281, -1.63299);
\draw[edge,postaction={decorate}] (4.00000, 1.41421, -0.81650) -- (4.00000, 0.00000, 0.00000);
\draw[edge,postaction={decorate}] (1.33333, 3.29983, 0.81650) -- (1.33333, 1.88562, 1.63299);
\draw[edge,postaction={decorate}] (2.66667, 0.94281, 1.63299) -- (1.33333, 1.88562, 1.63299);
\draw[edge,postaction={decorate}] (2.66667, 3.77124, 0.00000) -- (1.33333, 3.29983, 0.81650);
\draw[edge,postaction={decorate}] (2.66667, 0.94281, -1.63299) -- (2.66667, -0.47140, -0.81650);
\draw[edge,postaction={decorate}] (4.00000, 0.00000, 0.00000) -- (2.66667, -0.47140, -0.81650);
\draw[edge,postaction={decorate}] (4.00000, 0.00000, 0.00000) -- (2.66667, -0.47140, 0.81650);
\draw[edge,postaction={decorate}] (2.66667, 0.94281, 1.63299) -- (2.66667, -0.47140, 0.81650);
\node[vertex] at (0.00000, 0.00000, 0.00000)     {};
\node[vertex] at (0.00000, 1.41421, 0.81650)     {};
\node[vertex] at (1.33333, -0.94281, 0.00000)     {};
\node[vertex] at (4.00000, 2.82843, 0.00000)     {};
\node[vertex] at (4.00000, 1.41421, 0.81650)     {};
\node[vertex] at (4.00000, 1.41421, -0.81650)     {};
\node[vertex] at (1.33333, 1.88562, 1.63299)     {};
\node[vertex] at (1.33333, 3.29983, 0.81650)     {};
\node[vertex] at (2.66667, -0.47140, -0.81650)     {};
\node[vertex] at (2.66667, -0.47140, 0.81650)     {};
\node[vertex] at (2.66667, 0.94281, -1.63299)     {};
\node[vertex] at (4.00000, 0.00000, 0.00000)     {};
\node[vertex] at (2.66667, 0.94281, 1.63299)     {};
\node[vertex] at (2.66667, 3.77124, 0.00000)     {};
\end{tikzpicture}}
			\raisebox{-.9cm}{

\begin{tikzpicture}%
	[x={(-0.366215cm, -0.789554cm)},
	y={(0.235950cm, -0.590693cm)},
	z={(0.900119cm, -0.166391cm)},
	scale=1.000000,
	back/.style={very thin, opacity=0.5},
	edge/.style={color=red, thick, decoration={markings, mark=at position 0.5 with {\arrow{>}}}},
	facet/.style={fill=red,fill opacity=0},
	vertex/.style={}]
%
%

\coordinate (0.00000, 0.00000, 0.00000) at (0.00000, 0.00000, 0.00000);
\coordinate (0.00000, 0.00000, 1.63299) at (0.00000, 0.00000, 1.63299);
\coordinate (0.00000, 1.41421, -0.81650) at (0.00000, 1.41421, -0.81650);
\coordinate (4.00000, 2.82843, 1.63299) at (4.00000, 2.82843, 1.63299);
\coordinate (0.00000, 1.41421, 2.44949) at (0.00000, 1.41421, 2.44949);
\coordinate (0.00000, 2.82843, 0.00000) at (0.00000, 2.82843, 0.00000);
\coordinate (0.00000, 2.82843, 1.63299) at (0.00000, 2.82843, 1.63299);
\coordinate (1.33333, -0.94281, 0.00000) at (1.33333, -0.94281, 0.00000);
\coordinate (1.33333, -0.94281, 1.63299) at (1.33333, -0.94281, 1.63299);
\coordinate (4.00000, 2.82843, 0.00000) at (4.00000, 2.82843, 0.00000);
\coordinate (4.00000, 1.41421, 2.44949) at (4.00000, 1.41421, 2.44949);
\coordinate (4.00000, 1.41421, -0.81650) at (4.00000, 1.41421, -0.81650);
\coordinate (1.33333, 1.88562, -1.63299) at (1.33333, 1.88562, -1.63299);
\coordinate (4.00000, 0.00000, 1.63299) at (4.00000, 0.00000, 1.63299);
\coordinate (4.00000, 0.00000, 0.00000) at (4.00000, 0.00000, 0.00000);
\coordinate (1.33333, 1.88562, 3.26599) at (1.33333, 1.88562, 3.26599);
\coordinate (1.33333, 3.29983, -0.81650) at (1.33333, 3.29983, -0.81650);
\coordinate (2.66667, 3.77124, 1.63299) at (2.66667, 3.77124, 1.63299);
\coordinate (1.33333, 3.29983, 2.44949) at (1.33333, 3.29983, 2.44949);
\coordinate (2.66667, -0.47140, -0.81650) at (2.66667, -0.47140, -0.81650);
\coordinate (2.66667, 3.77124, 0.00000) at (2.66667, 3.77124, 0.00000);
\coordinate (2.66667, -0.47140, 2.44949) at (2.66667, -0.47140, 2.44949);
\coordinate (2.66667, 0.94281, -1.63299) at (2.66667, 0.94281, -1.63299);
\coordinate (2.66667, 0.94281, 3.26599) at (2.66667, 0.94281, 3.26599);
\draw[edge,postaction={decorate},back] (0.00000, 1.41421, -0.81650) -- (0.00000, 0.00000, 0.00000);
\draw[edge,postaction={decorate},back] (0.00000, 2.82843, 0.00000) -- (0.00000, 1.41421, -0.81650);
\draw[edge,postaction={decorate},back] (1.33333, 1.88562, -1.63299) -- (0.00000, 1.41421, -0.81650);
\draw[edge,postaction={decorate},back] (0.00000, 2.82843, 1.63299) -- (0.00000, 1.41421, 2.44949);
\draw[edge,postaction={decorate},back] (0.00000, 2.82843, 1.63299) -- (0.00000, 2.82843, 0.00000);
\draw[edge,postaction={decorate},back] (1.33333, 3.29983, -0.81650) -- (0.00000, 2.82843, 0.00000);
\draw[edge,postaction={decorate},back] (1.33333, 3.29983, 2.44949) -- (0.00000, 2.82843, 1.63299);
\draw[edge,postaction={decorate},back] (4.00000, 2.82843, 0.00000) -- (2.66667, 3.77124, 0.00000);
\draw[edge,postaction={decorate},back] (1.33333, 3.29983, -0.81650) -- (1.33333, 1.88562, -1.63299);
\draw[edge,postaction={decorate},back] (2.66667, 0.94281, -1.63299) -- (1.33333, 1.88562, -1.63299);
\draw[edge,postaction={decorate},back] (2.66667, 3.77124, 0.00000) -- (1.33333, 3.29983, -0.81650);
\draw[edge,postaction={decorate},back] (2.66667, 3.77124, 1.63299) -- (2.66667, 3.77124, 0.00000);
\node[vertex] at (1.33333, 1.88562, -1.63299)     {};
\node[vertex] at (1.33333, 3.29983, -0.81650)     {};
\node[vertex] at (2.66667, 3.77124, 0.00000)     {};
\node[vertex] at (0.00000, 1.41421, -0.81650)     {};
\node[vertex] at (0.00000, 2.82843, 0.00000)     {};
\node[vertex] at (0.00000, 2.82843, 1.63299)     {};
\fill[facet] (2.66667, 0.94281, -1.63299) -- (4.00000, 1.41421, -0.81650) -- (4.00000, 0.00000, 0.00000) -- (2.66667, -0.47140, -0.81650) -- cycle {};
\fill[facet] (1.33333, -0.94281, 1.63299) -- (0.00000, 0.00000, 1.63299) -- (0.00000, 0.00000, 0.00000) -- (1.33333, -0.94281, 0.00000) -- cycle {};
\fill[facet] (2.66667, 0.94281, 3.26599) -- (1.33333, 1.88562, 3.26599) -- (0.00000, 1.41421, 2.44949) -- (0.00000, 0.00000, 1.63299) -- (1.33333, -0.94281, 1.63299) -- (2.66667, -0.47140, 2.44949) -- cycle {};
\fill[facet] (4.00000, 0.00000, 0.00000) -- (4.00000, 0.00000, 1.63299) -- (2.66667, -0.47140, 2.44949) -- (1.33333, -0.94281, 1.63299) -- (1.33333, -0.94281, 0.00000) -- (2.66667, -0.47140, -0.81650) -- cycle {};
\fill[facet] (2.66667, 0.94281, 3.26599) -- (4.00000, 1.41421, 2.44949) -- (4.00000, 0.00000, 1.63299) -- (2.66667, -0.47140, 2.44949) -- cycle {};
\fill[facet] (1.33333, 1.88562, 3.26599) -- (2.66667, 0.94281, 3.26599) -- (4.00000, 1.41421, 2.44949) -- (4.00000, 2.82843, 1.63299) -- (2.66667, 3.77124, 1.63299) -- (1.33333, 3.29983, 2.44949) -- cycle {};
\fill[facet] (4.00000, 0.00000, 0.00000) -- (4.00000, 1.41421, -0.81650) -- (4.00000, 2.82843, 0.00000) -- (4.00000, 2.82843, 1.63299) -- (4.00000, 1.41421, 2.44949) -- (4.00000, 0.00000, 1.63299) -- cycle {};
\draw[edge,postaction={decorate}] (0.00000, 0.00000, 1.63299) -- (0.00000, 0.00000, 0.00000);
\draw[edge,postaction={decorate}] (1.33333, -0.94281, 0.00000) -- (0.00000, 0.00000, 0.00000);
\draw[edge,postaction={decorate}] (0.00000, 1.41421, 2.44949) -- (0.00000, 0.00000, 1.63299);
\draw[edge,postaction={decorate}] (1.33333, -0.94281, 1.63299) -- (0.00000, 0.00000, 1.63299);
\draw[edge,postaction={decorate}] (4.00000, 2.82843, 1.63299) -- (4.00000, 2.82843, 0.00000);
\draw[edge,postaction={decorate}] (4.00000, 2.82843, 1.63299) -- (4.00000, 1.41421, 2.44949);
\draw[edge,postaction={decorate}] (4.00000, 2.82843, 1.63299) -- (2.66667, 3.77124, 1.63299);
\draw[edge,postaction={decorate}] (1.33333, 1.88562, 3.26599) -- (0.00000, 1.41421, 2.44949);
\draw[edge,postaction={decorate}] (1.33333, -0.94281, 1.63299) -- (1.33333, -0.94281, 0.00000);
\draw[edge,postaction={decorate}] (2.66667, -0.47140, -0.81650) -- (1.33333, -0.94281, 0.00000);
\draw[edge,postaction={decorate}] (2.66667, -0.47140, 2.44949) -- (1.33333, -0.94281, 1.63299);
\draw[edge,postaction={decorate}] (4.00000, 2.82843, 0.00000) -- (4.00000, 1.41421, -0.81650);
\draw[edge,postaction={decorate}] (4.00000, 1.41421, 2.44949) -- (4.00000, 0.00000, 1.63299);
\draw[edge,postaction={decorate}] (4.00000, 1.41421, 2.44949) -- (2.66667, 0.94281, 3.26599);
\draw[edge,postaction={decorate}] (4.00000, 1.41421, -0.81650) -- (4.00000, 0.00000, 0.00000);
\draw[edge,postaction={decorate}] (4.00000, 1.41421, -0.81650) -- (2.66667, 0.94281, -1.63299);
\draw[edge,postaction={decorate}] (4.00000, 0.00000, 1.63299) -- (4.00000, 0.00000, 0.00000);
\draw[edge,postaction={decorate}] (4.00000, 0.00000, 1.63299) -- (2.66667, -0.47140, 2.44949);
\draw[edge,postaction={decorate}] (4.00000, 0.00000, 0.00000) -- (2.66667, -0.47140, -0.81650);
\draw[edge,postaction={decorate}] (1.33333, 3.29983, 2.44949) -- (1.33333, 1.88562, 3.26599);
\draw[edge,postaction={decorate}] (2.66667, 0.94281, 3.26599) -- (1.33333, 1.88562, 3.26599);
\draw[edge,postaction={decorate}] (2.66667, 3.77124, 1.63299) -- (1.33333, 3.29983, 2.44949);
\draw[edge,postaction={decorate}] (2.66667, 0.94281, -1.63299) -- (2.66667, -0.47140, -0.81650);
\draw[edge,postaction={decorate}] (2.66667, 0.94281, 3.26599) -- (2.66667, -0.47140, 2.44949);
\node[vertex] at (0.00000, 0.00000, 0.00000)     {};
\node[vertex] at (0.00000, 0.00000, 1.63299)     {};
\node[vertex] at (4.00000, 2.82843, 1.63299)     {};
\node[vertex] at (0.00000, 1.41421, 2.44949)     {};
\node[vertex] at (1.33333, -0.94281, 0.00000)     {};
\node[vertex] at (1.33333, -0.94281, 1.63299)     {};
\node[vertex] at (4.00000, 2.82843, 0.00000)     {};
\node[vertex] at (4.00000, 1.41421, 2.44949)     {};
\node[vertex] at (4.00000, 1.41421, -0.81650)     {};
\node[vertex] at (4.00000, 0.00000, 1.63299)     {};
\node[vertex] at (4.00000, 0.00000, 0.00000)     {};
\node[vertex] at (1.33333, 1.88562, 3.26599)     {};
\node[vertex] at (2.66667, 3.77124, 1.63299)     {};
\node[vertex] at (1.33333, 3.29983, 2.44949)     {};
\node[vertex] at (2.66667, -0.47140, -0.81650)     {};
\node[vertex] at (2.66667, -0.47140, 2.44949)     {};
\node[vertex] at (2.66667, 0.94281, -1.63299)     {};
\node[vertex] at (2.66667, 0.94281, 3.26599)     {};
\end{tikzpicture}}
	}
	\caption{The graphical zonotopes for the directed acyclic graphs of \cref{fig:canonicalJoinComplexes}. Their normal fans are the graphical fans of \cref{fig:graphicalArrangements} and their graphs oriented appropriately are isomorphic to the Hasse diagrams of the acyclic reorientation lattices of \cref{fig:canonicalJoinComplexes}. The rightmost zonotope is the classical permutahedron. The middle zonotope is not simple while the other two are.}
	\label{fig:graphicalZonotopes}
\end{figure}

\para{Quotientopes and associahedra}
Assume now that~$D$ is skeletal as in \cref{sec:ropeDiagrams,sec:congruences}, so that the acyclic reorientation poset~$\AR$ is a congruence uniform lattice by \cref{prop:congruenceUniform}.
The main result of this section is the following statement.

\begin{theorem}
\label{thm:existenceQuotientopes}
Assume that $D$ is skeletal.
For any congruence~$\equiv$ of~$\AR$, the quotient fan~$\Fan[\equiv]$ is the normal fan of a polytope.
\end{theorem}

A \defn{quotientope} is any polytopal realization of the quotient fan~$\Fan[\equiv]$.
We provide two general approaches to construct quotientopes in \cref{thm:MinkowskiSumAssociahedra,thm:MinkowskiSumShardPolytopes}, and we discuss a third approach specific to the coherent congruences in \cref{prop:associahedraRemovahedra,pb:CambrianAssociahedraRemovahedra,pb:coherentQuotientopesRemovahedra}.

\medskip
An \defn{associahedron} for~$D$ is any quotientope for the sylvester congruence~$\equiv_{(V,\varnothing)}$.
To avoid any confusion, let us insist that the associahedron of the directed acyclic graph~$D$ is not the associahedron of the underlying undirected graph~$G$ as defined by M.~Carr and S.~Devadoss in~\cite{CarrDevadoss}, except if~$D$ is a disjoint union of tournaments.
In fact, as suggested in \cref{pb:TamariRegular} and illustrated in \cref{fig:associahedra}\,(middle), the associahedron of~$D$ is not even always a simple polytope.

\medskip
The \defn{Cambrian associahedra} of~$D$ are the quotientopes for the Cambrian congruences~$\equiv_{(\decorationDown, \decorationUp)}$ with~$\decorationDown \sqcup \decorationUp = V$.
As already mentioned in \cref{pb:allCambrianSameGraphs}, not all Cambrian lattices have the same number of elements.
In fact, computer experiments on all skeletal directed acyclic graphs up to $6$ vertices indicate that the following stronger version of \cref{pb:allCambrianSameGraphs} should hold.

\begin{problem}
\label{pb:allCambrianSameFaceLattices}
Prove the equivalence of the following assertions for a skeletal directed acyclic~graph~$D$:
\begin{enumerate}[(i)]
\item $D$ has no induced subgraph isomorphic to\,\smash{\raisebox{-.25cm}{\includegraphics[scale=.9]{digraph1}}},
\item all Cambrian associahedra of~$D$ have the same number of vertices,
\item all Cambrian associahedra of~$D$ have isomorphic $1$-skeleta,
\item all Cambrian associahedra of~$D$ have isomorphic face lattices.
\end{enumerate}
\end{problem}

Note that Points~(ii) and~(iii) in \cref{pb:allCambrianSameFaceLattices} are just geometric translations of Points~(ii) and~(iii) in \cref{pb:allCambrianSameGraphs}.
Point~(iv) is a consequence of Point~(iii) when the associahedron is a simple polytope, since the face lattice of a simple polytope is determined by its graph~\cite{BlindMani, Kalai-simplePolytopes}.
However, Point~(iv) is stronger than Point~(iii) when the associahedron is not simple, which happens when $D$ has no induced subgraph isomorphic to\,\smash{\raisebox{-.25cm}{\includegraphics[scale=.9]{digraph1}}}\,but some induced subgraph isomorphic to\,\smash{\raisebox{-.25cm}{\includegraphics[scale=.9]{digraph2}}}\,or\,\smash{\raisebox{-.25cm}{\includegraphics[scale=.9]{digraph3}}}\,by \cref{pb:TamariRegular}.

\pagebreak

\para{Quotientopes from classical associahedra}
Our first approach to realize the quotient fan~$\Fan[\equiv]$ as the normal fan of a polytope is based on the associahedra~\cite{ShniderSternberg, Loday, HohlwegLange}.
Recall first that when~$D$ is the increasing tournament on~$[n]$, the sylvester fan is the normal fan of the classical associahedron, defined equivalently~as
\begin{itemize}
\item the convex hull of the points~$\sum_{j \in [n]} \ell(T,j) \, r(T,j) \, \b{e}_j$ for all binary trees~$T$ on~$n$ nodes, where $\ell(T,j)$ and~$r(T,j)$ respectively denote the numbers of leaves in the left and right subtrees of the node~$j$ of~$T$ (labeled in inorder), see~\cite{Loday},
\item the intersection of the hyperplane~$\bigset{\b{x} \in \R^n}{\dotprod{\one}{\b{x}} = \binom{n+1}{2}}$ with the halfspaces ${\bigset{\b{x} \in \R^n}{\dotprod{\one_{[a,b]}}{\b{x}} \ge \binom{b-a+2}{2}}}$ for all intervals~$1 \le a \le b \le n$, see~\cite{ShniderSternberg},
\item (a translate of) the Minkowski sum of~$\simplex_{[a,b]}$ for all intervals ${1 \le a \le b \le n}$, where for~${I \subseteq [n]}$, $\simplex_I \eqdef \conv\set{\b{e}_i}{i \in I}$ is the face of the standard simplex~$\simplex_{[n]}$ labeled by~$I$, see~\cite{Postnikov}.
\end{itemize}
Similar polytopal realizations were constructed for the quotient fans of the Cambrian congruences in~\cite{HohlwegLange, LangePilaud} and of the permutree congruences in~\cite{PilaudPons-permutrees} of the weak order, using analogous vertex and facet descriptions.
The resulting associahedra and permutreehedra can also be written as Minkowski sums and differences of faces of the standard simplex, though the description is not as simple, see \eg \cite{Lange}.
Here, we skip the precise vertex, facet, and Minkowski descriptions of all these polytopes and refer to \cite{HohlwegLange, LangePilaud, Lange, PilaudPons-permutrees} for details.
We just need to observe that the existence of these polytopes together with \cref{prop:extensionCongruence} ensure that the quotient fan of the principal congruence~$\equiv_\rope$ of any rope~$\rope$ of~$D$ is the normal fan of an associahedron~$\Asso[\rope]$ obtained as an embedding of a Cambrian associahedron of~\cite{HohlwegLange} in~$\R^V$.
Mimicking~\cite[Thm.~1]{PadrolPilaudRitter}, we now observe that any quotient fan can be realized as the normal fan of a Minkowski sum of (low dimensional) Cambrian associahedra of~\cite{HohlwegLange}.

\begin{theorem}
\label{thm:MinkowskiSumAssociahedra}
Assume that~$D$ is skeletal.
Consider any congruence~$\equiv$ of~$\AR$, and let~$\rope_1, \dots, \rope_p$ denote the ropes generating the lower ideal~$\bb{I}_\equiv$ of the subrope order.
Then the quotient fan~$\Fan[\equiv]$~is
\begin{itemize}
\item the common refinement of the Cambrian fans~$\Fan[\rope_1], \dots, \Fan[\rope_p]$,
\item the normal fan of the Minkowski sum of the Cambrian associahedra~$\Asso[\rope_1], \dots, \Asso[\rope_p]$.
\end{itemize}
\end{theorem}

\begin{proof}
The first point is immediate since~$\bb{I}_\equiv = \bb{I}_{\rope_1} \cup \dots \cup \bb{I}_{\rope_p}$ and the union of the walls of~$\Fan[\equiv]$ is the union of the shards~$\shard_\rope$ for~$\rope \in \bb{I}_\equiv$.
The second point follows from the fact that the Cambrian fan~$\Fan[\rope]$ is the normal fan of the Cambrian associahedron~$\Asso[\rope]$, and that the normal fan of a Minkowski sum is the common refinement of the normal fans of the summands.
\end{proof}

\para{Quotientopes from shard polytopes}
\enlargethispage{.4cm}
Our second approach to realize the quotient fan~$\Fan[\equiv]$ as the normal fan of a polytope is based on the shard polytopes of~\cite{PadrolPilaudRitter}.
Consider a rope~$\rope \eqdef (u, v, \down, \up)$ and let~$\pi$ denote the directed path from~$u$ to~$v$ in the transitive reduction of~$D$.
Define
\begin{itemize}
\item a \defn{$\rope$-alternating matching} as a pair~$(M_\down, M_\up)$ with~$M_\down \subseteq \{u\} \cup \down$ and~$M_\up \subseteq \up \cup \{v\}$ such that~$M_\down$ and~$M_\up$ are alternating along~$\pi$, 
\item a \defn{$\rope$-fall} (resp.~\defn{$\rope$-rise}) as a subset of the vertices of~$\pi$ situated between~$u$ and an arc~$(w,w')$ of~$\pi$ such that~$w \in \{u\} \cup \down$ while~$w' \in \up \cup \{v\}$ (resp.~$w \in \{u\} \cup \up$ while~$w' \in \down \cup \{v\}$).
\end{itemize}
The \defn{shard polytope}~$\shardPolytope$ of a rope~$\rope \eqdef (u, v, \down, \up)$ is the polytope of~$\R^V$ defined equivalently as
\begin{itemize}
\item the convex hull of the vectors~$\one_{M_\down} - \one_{M_\up}$ for all $\rope$-alternating matchings~$M \eqdef (M_\down, M_\up)$,
\item the subset of the plane~$\HH$ (orthogonal to the characteristic vectors of the connected components of~$D$) defined by
	\begin{enumerate}[$\circ$]
	\item $x_w = 0$ for all~$v \notin \{u,v\} \cup \down \cup \up$,
	\item $x_w \ge 0$ for~$w \in \down$ and~$x_w \le 0$ for~$w \in \up$,
	\item $\sum_{w \in F} x_w \le 1$ for each $\rope$-fall~$F$ and $\sum_{w \in R} x_w \ge 0$ for each $\rope$-rise~$R$.
	\end{enumerate}
\end{itemize}
For instance, the shard polytope~$\shardPolytope$ of a rope of the form~$\rope \eqdef (u, v, \down, \varnothing)$ is the face~$\simplex_{\{u,v\} \cup \down}$ of the standard simplex, translated by the vector~$-\b{e}_v$.
We refer to~\cite{PadrolPilaudRitter} for an alternative definition of the shard polytope~$\shardPolytope$ as the matroid polytope of a series-parallel graph associated to~$\rope$.
The following statement is the fundamental property of shard polytopes.

\begin{proposition}
\label{prop:shardPolytopes}
Assume that~$D$ is skeletal.
For any rope~$\rope$ of~$D$, the union of the walls of the normal fan of the shard polytope~$\shardPolytope$ contains the shard~$\shard_\rope$ and is contained in the union of the shards~$\shard_{\rope'}$ for all subropes~$\rope'$ of~$\rope$.
\end{proposition}

\begin{proof}
It was proved in~\cite{PadrolPilaudRitter} when~$D$ is a tournament, and thus follows in general since the shard polytope~$\shardPolytope$ is just an embedding of a classical shard polytope in~$\R^V$.
\end{proof}

Based on \cref{prop:shardPolytopes}, we obtain polytopal realizations of all lattice quotients of~$\AR$ as Minkowski sums of shard polytopes.

\begin{theorem}
\label{thm:MinkowskiSumShardPolytopes}
Assume that~$D$ is skeletal.
For any congruence~$\equiv$ of~$\AR$ and any positive coefficients~${\b{s} \in (\R_{>0})^{\bb{I}_\equiv}}$, the quotient fan~$\Fan[\equiv]$ is the normal fan of the Minkowski sum~$\sum_{\rope \in \bb{I}_\equiv} \b{s}_\rope \, \shardPolytope$.
\end{theorem}

\begin{proof}
The normal fan of a Minkowski sum is the common refinement of the normal fans of the summands.
Hence, by \cref{prop:shardPolytopes}, the union of the walls of the normal fan of~$\sum_{\rope \in \bb{I}_\equiv} \b{s}_\rope \, \shardPolytope$ is precisely the union of the shards~$\shard_\rope$ for all ropes~$\rope \in \bb{I}_\equiv$.
In other words, the normal fan of~$\sum_{\rope \in \bb{I}_\equiv} \b{s}_\rope \, \shardPolytope$ has the same walls as the quotient fan~$\Fan[\equiv]$, so that these two fans coincide.
\end{proof}

It follows that the Hasse diagram of the quotient~$\AR/{\equiv}$ is isomorphic to the graph of the polytope~$\sum_{\rope \in \bb{I}_\equiv} \b{s}_\rope \, \shardPolytope$, oriented in the direction~$\b{\omega}_D \eqdef \sum_{(u,v) \in A} \b{e}_v - \b{e}_u$.
These Minkowski sums are illustrated in \cref{fig:associahedra}.
It would be interesting to characterize which of these quotientopes are simple, which is a reformulation of \cref{pb:regularCoverGraph} for arbitrary congruences and \cref{pb:TamariRegular,pb:allCambrianRegular} for Cambrian congruences.

\begin{figure}
	\centerline{

\begin{tikzpicture}%
	[x={(-0.366215cm, -0.789554cm)},
	y={(0.235950cm, -0.590693cm)},
	z={(0.900119cm, -0.166391cm)},
	scale=1.000000,
	back/.style={very thin, opacity=0.5},
	edge/.style={color=blue, thick, decoration={markings, mark=at position 0.5 with {\arrow{>}}}},
	facet/.style={fill=blue,fill opacity=0},
	vertex/.style={}]
%
%

\coordinate (0.00000, 0.00000, 0.00000) at (0.00000, 0.00000, 0.00000);
\coordinate (0.00000, 0.00000, 1.63299) at (0.00000, 0.00000, 1.63299);
\coordinate (1.33333, -0.94281, 0.00000) at (1.33333, -0.94281, 0.00000);
\coordinate (1.33333, -0.94281, 1.63299) at (1.33333, -0.94281, 1.63299);
\coordinate (4.00000, 0.00000, 1.63299) at (4.00000, 0.00000, 1.63299);
\coordinate (1.33333, 0.47140, 2.44949) at (1.33333, 0.47140, 2.44949);
\coordinate (4.00000, 0.00000, -1.63299) at (4.00000, 0.00000, -1.63299);
\coordinate (2.66667, -0.47140, 2.44949) at (2.66667, -0.47140, 2.44949);
\coordinate (2.66667, 0.94281, -1.63299) at (2.66667, 0.94281, -1.63299);
\coordinate (2.66667, 0.94281, 1.63299) at (2.66667, 0.94281, 1.63299);
\draw[edge,postaction={decorate},back] (2.66667, 0.94281, -1.63299) -- (0.00000, 0.00000, 0.00000);
\draw[edge,postaction={decorate},back] (4.00000, 0.00000, -1.63299) -- (2.66667, 0.94281, -1.63299);
\draw[edge,postaction={decorate},back] (2.66667, 0.94281, 1.63299) -- (2.66667, 0.94281, -1.63299);
\node[vertex] at (2.66667, 0.94281, -1.63299)     {};
\fill[facet] (1.33333, -0.94281, 1.63299) -- (0.00000, 0.00000, 1.63299) -- (0.00000, 0.00000, 0.00000) -- (1.33333, -0.94281, 0.00000) -- cycle {};
\fill[facet] (2.66667, -0.47140, 2.44949) -- (1.33333, -0.94281, 1.63299) -- (0.00000, 0.00000, 1.63299) -- (1.33333, 0.47140, 2.44949) -- cycle {};
\fill[facet] (1.33333, 0.47140, 2.44949) -- (2.66667, 0.94281, 1.63299) -- (4.00000, 0.00000, 1.63299) -- (2.66667, -0.47140, 2.44949) -- cycle {};
\fill[facet] (4.00000, 0.00000, 1.63299) -- (2.66667, -0.47140, 2.44949) -- (1.33333, -0.94281, 1.63299) -- (1.33333, -0.94281, 0.00000) -- (4.00000, 0.00000, -1.63299) -- cycle {};
\draw[edge,postaction={decorate}] (0.00000, 0.00000, 1.63299) -- (0.00000, 0.00000, 0.00000);
\draw[edge,postaction={decorate}] (1.33333, -0.94281, 0.00000) -- (0.00000, 0.00000, 0.00000);
\draw[edge,postaction={decorate}] (1.33333, -0.94281, 1.63299) -- (0.00000, 0.00000, 1.63299);
\draw[edge,postaction={decorate}] (1.33333, 0.47140, 2.44949) -- (0.00000, 0.00000, 1.63299);
\draw[edge,postaction={decorate}] (1.33333, -0.94281, 1.63299) -- (1.33333, -0.94281, 0.00000);
\draw[edge,postaction={decorate}] (4.00000, 0.00000, -1.63299) -- (1.33333, -0.94281, 0.00000);
\draw[edge,postaction={decorate}] (2.66667, -0.47140, 2.44949) -- (1.33333, -0.94281, 1.63299);
\draw[edge,postaction={decorate}] (4.00000, 0.00000, 1.63299) -- (4.00000, 0.00000, -1.63299);
\draw[edge,postaction={decorate}] (4.00000, 0.00000, 1.63299) -- (2.66667, -0.47140, 2.44949);
\draw[edge,postaction={decorate}] (4.00000, 0.00000, 1.63299) -- (2.66667, 0.94281, 1.63299);
\draw[edge,postaction={decorate}] (2.66667, -0.47140, 2.44949) -- (1.33333, 0.47140, 2.44949);
\draw[edge,postaction={decorate}] (2.66667, 0.94281, 1.63299) -- (1.33333, 0.47140, 2.44949);
\node[vertex] at (0.00000, 0.00000, 0.00000)     {};
\node[vertex] at (0.00000, 0.00000, 1.63299)     {};
\node[vertex] at (1.33333, -0.94281, 0.00000)     {};
\node[vertex] at (1.33333, -0.94281, 1.63299)     {};
\node[vertex] at (4.00000, 0.00000, 1.63299)     {};
\node[vertex] at (1.33333, 0.47140, 2.44949)     {};
\node[vertex] at (4.00000, 0.00000, -1.63299)     {};
\node[vertex] at (2.66667, -0.47140, 2.44949)     {};
\node[vertex] at (2.66667, 0.94281, 1.63299)     {};
\end{tikzpicture}
			\raisebox{-.7cm}{

\begin{tikzpicture}%
	[x={(-0.366215cm, -0.789554cm)},
	y={(0.235950cm, -0.590693cm)},
	z={(0.900119cm, -0.166391cm)},
	scale=1.000000,
	back/.style={very thin, opacity=0.5},
	edge/.style={color=blue, thick, decoration={markings, mark=at position 0.5 with {\arrow{>}}}},
	facet/.style={fill=blue,fill opacity=0},
	vertex/.style={}]
%
%

\coordinate (0.00000, 0.00000, 0.00000) at (0.00000, 0.00000, 0.00000);
\coordinate (0.00000, 1.41421, -0.81650) at (0.00000, 1.41421, -0.81650);
\coordinate (0.00000, 1.41421, 0.81650) at (0.00000, 1.41421, 0.81650);
\coordinate (0.00000, 2.82843, 0.00000) at (0.00000, 2.82843, 0.00000);
\coordinate (4.00000, 2.82843, 0.00000) at (4.00000, 2.82843, 0.00000);
\coordinate (1.33333, 1.88562, -1.63299) at (1.33333, 1.88562, -1.63299);
\coordinate (4.00000, 0.00000, 1.63299) at (4.00000, 0.00000, 1.63299);
\coordinate (1.33333, 1.88562, 1.63299) at (1.33333, 1.88562, 1.63299);
\coordinate (1.33333, 3.29983, -0.81650) at (1.33333, 3.29983, -0.81650);
\coordinate (1.33333, 3.29983, 0.81650) at (1.33333, 3.29983, 0.81650);
\coordinate (4.00000, 0.00000, -1.63299) at (4.00000, 0.00000, -1.63299);
\coordinate (2.66667, 3.77124, 0.00000) at (2.66667, 3.77124, 0.00000);
\coordinate (4.00000, -2.82843, 0.00000) at (4.00000, -2.82843, 0.00000);
\draw[edge,postaction={decorate},back] (0.00000, 1.41421, -0.81650) -- (0.00000, 0.00000, 0.00000);
\draw[edge,postaction={decorate},back] (0.00000, 2.82843, 0.00000) -- (0.00000, 1.41421, -0.81650);
\draw[edge,postaction={decorate},back] (1.33333, 1.88562, -1.63299) -- (0.00000, 1.41421, -0.81650);
\draw[edge,postaction={decorate},back] (0.00000, 2.82843, 0.00000) -- (0.00000, 1.41421, 0.81650);
\draw[edge,postaction={decorate},back] (1.33333, 3.29983, -0.81650) -- (0.00000, 2.82843, 0.00000);
\draw[edge,postaction={decorate},back] (1.33333, 3.29983, 0.81650) -- (0.00000, 2.82843, 0.00000);
\draw[edge,postaction={decorate},back] (1.33333, 3.29983, -0.81650) -- (1.33333, 1.88562, -1.63299);
\draw[edge,postaction={decorate},back] (4.00000, 0.00000, -1.63299) -- (1.33333, 1.88562, -1.63299);
\draw[edge,postaction={decorate},back] (2.66667, 3.77124, 0.00000) -- (1.33333, 3.29983, -0.81650);
\node[vertex] at (0.00000, 2.82843, 0.00000)     {};
\node[vertex] at (0.00000, 1.41421, -0.81650)     {};
\node[vertex] at (1.33333, 1.88562, -1.63299)     {};
\node[vertex] at (1.33333, 3.29983, -0.81650)     {};
\fill[facet] (4.00000, 0.00000, 1.63299) -- (4.00000, -2.82843, 0.00000) -- (0.00000, 0.00000, 0.00000) -- (0.00000, 1.41421, 0.81650) -- (1.33333, 1.88562, 1.63299) -- cycle {};
\fill[facet] (4.00000, -2.82843, 0.00000) -- (4.00000, 0.00000, 1.63299) -- (4.00000, 2.82843, 0.00000) -- (4.00000, 0.00000, -1.63299) -- cycle {};
\fill[facet] (2.66667, 3.77124, 0.00000) -- (4.00000, 2.82843, 0.00000) -- (4.00000, 0.00000, 1.63299) -- (1.33333, 1.88562, 1.63299) -- (1.33333, 3.29983, 0.81650) -- cycle {};
\draw[edge,postaction={decorate}] (0.00000, 1.41421, 0.81650) -- (0.00000, 0.00000, 0.00000);
\draw[edge,postaction={decorate}] (4.00000, -2.82843, 0.00000) -- (0.00000, 0.00000, 0.00000);
\draw[edge,postaction={decorate}] (1.33333, 1.88562, 1.63299) -- (0.00000, 1.41421, 0.81650);
\draw[edge,postaction={decorate}] (4.00000, 2.82843, 0.00000) -- (4.00000, 0.00000, 1.63299);
\draw[edge,postaction={decorate}] (4.00000, 2.82843, 0.00000) -- (4.00000, 0.00000, -1.63299);
\draw[edge,postaction={decorate}] (4.00000, 2.82843, 0.00000) -- (2.66667, 3.77124, 0.00000);
\draw[edge,postaction={decorate}] (4.00000, 0.00000, 1.63299) -- (1.33333, 1.88562, 1.63299);
\draw[edge,postaction={decorate}] (4.00000, 0.00000, 1.63299) -- (4.00000, -2.82843, 0.00000);
\draw[edge,postaction={decorate}] (1.33333, 3.29983, 0.81650) -- (1.33333, 1.88562, 1.63299);
\draw[edge,postaction={decorate}] (2.66667, 3.77124, 0.00000) -- (1.33333, 3.29983, 0.81650);
\draw[edge,postaction={decorate}] (4.00000, 0.00000, -1.63299) -- (4.00000, -2.82843, 0.00000);
\node[vertex] at (0.00000, 0.00000, 0.00000)     {};
\node[vertex] at (0.00000, 1.41421, 0.81650)     {};
\node[vertex] at (4.00000, 2.82843, 0.00000)     {};
\node[vertex] at (4.00000, 0.00000, 1.63299)     {};
\node[vertex] at (1.33333, 1.88562, 1.63299)     {};
\node[vertex] at (1.33333, 3.29983, 0.81650)     {};
\node[vertex] at (4.00000, 0.00000, -1.63299)     {};
\node[vertex] at (2.66667, 3.77124, 0.00000)     {};
\node[vertex] at (4.00000, -2.82843, 0.00000)     {};
\end{tikzpicture}}
			\raisebox{-.9cm}{

\begin{tikzpicture}%
	[x={(-0.366215cm, -0.789554cm)},
	y={(0.235950cm, -0.590693cm)},
	z={(0.900119cm, -0.166391cm)},
	scale=1.000000,
	back/.style={very thin, opacity=0.5},
	edge/.style={color=blue, thick, decoration={markings, mark=at position 0.5 with {\arrow{>}}}},
	facet/.style={fill=blue,fill opacity=0},
	vertex/.style={}]
%
%

\coordinate (0.00000, 0.00000, 0.00000) at (0.00000, 0.00000, 0.00000);
\coordinate (0.00000, 0.00000, 1.63299) at (0.00000, 0.00000, 1.63299);
\coordinate (4.00000, 2.82843, 1.63299) at (4.00000, 2.82843, 1.63299);
\coordinate (0.00000, 1.41421, 2.44949) at (0.00000, 1.41421, 2.44949);
\coordinate (0.00000, 2.82843, -1.63299) at (0.00000, 2.82843, -1.63299);
\coordinate (4.00000, 2.82843, -3.26599) at (4.00000, 2.82843, -3.26599);
\coordinate (0.00000, 2.82843, 1.63299) at (0.00000, 2.82843, 1.63299);
\coordinate (4.00000, 0.00000, 3.26599) at (4.00000, 0.00000, 3.26599);
\coordinate (1.33333, 1.88562, 3.26599) at (1.33333, 1.88562, 3.26599);
\coordinate (4.00000, -2.82843, 1.63299) at (4.00000, -2.82843, 1.63299);
\coordinate (1.33333, 3.29983, 2.44949) at (1.33333, 3.29983, 2.44949);
\coordinate (2.66667, 3.77124, -3.26599) at (2.66667, 3.77124, -3.26599);
\coordinate (4.00000, -2.82843, 0.00000) at (4.00000, -2.82843, 0.00000);
\coordinate (2.66667, 3.77124, 1.63299) at (2.66667, 3.77124, 1.63299);
\draw[edge,postaction={decorate},back] (0.00000, 2.82843, -1.63299) -- (0.00000, 0.00000, 0.00000);
\draw[edge,postaction={decorate},back] (0.00000, 2.82843, 1.63299) -- (0.00000, 1.41421, 2.44949);
\draw[edge,postaction={decorate},back] (0.00000, 2.82843, 1.63299) -- (0.00000, 2.82843, -1.63299);
\draw[edge,postaction={decorate},back] (2.66667, 3.77124, -3.26599) -- (0.00000, 2.82843, -1.63299);
\draw[edge,postaction={decorate},back] (4.00000, 2.82843, -3.26599) -- (2.66667, 3.77124, -3.26599);
\draw[edge,postaction={decorate},back] (1.33333, 3.29983, 2.44949) -- (0.00000, 2.82843, 1.63299);
\draw[edge,postaction={decorate},back] (2.66667, 3.77124, 1.63299) -- (2.66667, 3.77124, -3.26599);
\node[vertex] at (0.00000, 2.82843, -1.63299)     {};
\node[vertex] at (2.66667, 3.77124, -3.26599)     {};
\node[vertex] at (0.00000, 2.82843, 1.63299)     {};
\fill[facet] (4.00000, -2.82843, 0.00000) -- (0.00000, 0.00000, 0.00000) -- (0.00000, 0.00000, 1.63299) -- (4.00000, -2.82843, 1.63299) -- cycle {};
\fill[facet] (4.00000, 0.00000, 3.26599) -- (4.00000, -2.82843, 1.63299) -- (0.00000, 0.00000, 1.63299) -- (0.00000, 1.41421, 2.44949) -- (1.33333, 1.88562, 3.26599) -- cycle {};
\fill[facet] (2.66667, 3.77124, 1.63299) -- (4.00000, 2.82843, 1.63299) -- (4.00000, 0.00000, 3.26599) -- (1.33333, 1.88562, 3.26599) -- (1.33333, 3.29983, 2.44949) -- cycle {};
\fill[facet] (4.00000, -2.82843, 0.00000) -- (4.00000, 2.82843, -3.26599) -- (4.00000, 2.82843, 1.63299) -- (4.00000, 0.00000, 3.26599) -- (4.00000, -2.82843, 1.63299) -- cycle {};
\draw[edge,postaction={decorate}] (0.00000, 0.00000, 1.63299) -- (0.00000, 0.00000, 0.00000);
\draw[edge,postaction={decorate}] (4.00000, -2.82843, 0.00000) -- (0.00000, 0.00000, 0.00000);
\draw[edge,postaction={decorate}] (0.00000, 1.41421, 2.44949) -- (0.00000, 0.00000, 1.63299);
\draw[edge,postaction={decorate}] (4.00000, -2.82843, 1.63299) -- (0.00000, 0.00000, 1.63299);
\draw[edge,postaction={decorate}] (4.00000, 2.82843, 1.63299) -- (4.00000, 2.82843, -3.26599);
\draw[edge,postaction={decorate}] (4.00000, 2.82843, 1.63299) -- (4.00000, 0.00000, 3.26599);
\draw[edge,postaction={decorate}] (4.00000, 2.82843, 1.63299) -- (2.66667, 3.77124, 1.63299);
\draw[edge,postaction={decorate}] (1.33333, 1.88562, 3.26599) -- (0.00000, 1.41421, 2.44949);
\draw[edge,postaction={decorate}] (4.00000, 2.82843, -3.26599) -- (4.00000, -2.82843, 0.00000);
\draw[edge,postaction={decorate}] (4.00000, 0.00000, 3.26599) -- (1.33333, 1.88562, 3.26599);
\draw[edge,postaction={decorate}] (4.00000, 0.00000, 3.26599) -- (4.00000, -2.82843, 1.63299);
\draw[edge,postaction={decorate}] (1.33333, 3.29983, 2.44949) -- (1.33333, 1.88562, 3.26599);
\draw[edge,postaction={decorate}] (4.00000, -2.82843, 1.63299) -- (4.00000, -2.82843, 0.00000);
\draw[edge,postaction={decorate}] (2.66667, 3.77124, 1.63299) -- (1.33333, 3.29983, 2.44949);
\node[vertex] at (0.00000, 0.00000, 0.00000)     {};
\node[vertex] at (0.00000, 0.00000, 1.63299)     {};
\node[vertex] at (4.00000, 2.82843, 1.63299)     {};
\node[vertex] at (0.00000, 1.41421, 2.44949)     {};
\node[vertex] at (4.00000, 2.82843, -3.26599)     {};
\node[vertex] at (4.00000, 0.00000, 3.26599)     {};
\node[vertex] at (1.33333, 1.88562, 3.26599)     {};
\node[vertex] at (4.00000, -2.82843, 1.63299)     {};
\node[vertex] at (1.33333, 3.29983, 2.44949)     {};
\node[vertex] at (4.00000, -2.82843, 0.00000)     {};
\node[vertex] at (2.66667, 3.77124, 1.63299)     {};
\end{tikzpicture}}
	}
	\caption{The associahedra for the directed acyclic graphs of \cref{fig:canonicalJoinComplexes}. Their normal fans are the sylvester fans of \cref{fig:quotientFans} and their graphs oriented appropriately are isomorphic to the Hasse diagrams of the Tamari lattices of \cref{fig:sylvesterCongruences}. The rightmost associahedron is the classical associahedron of~\cite{ShniderSternberg,Loday}. The middle associahedron is not simple while the other two are.}
	\label{fig:associahedra}
\end{figure}
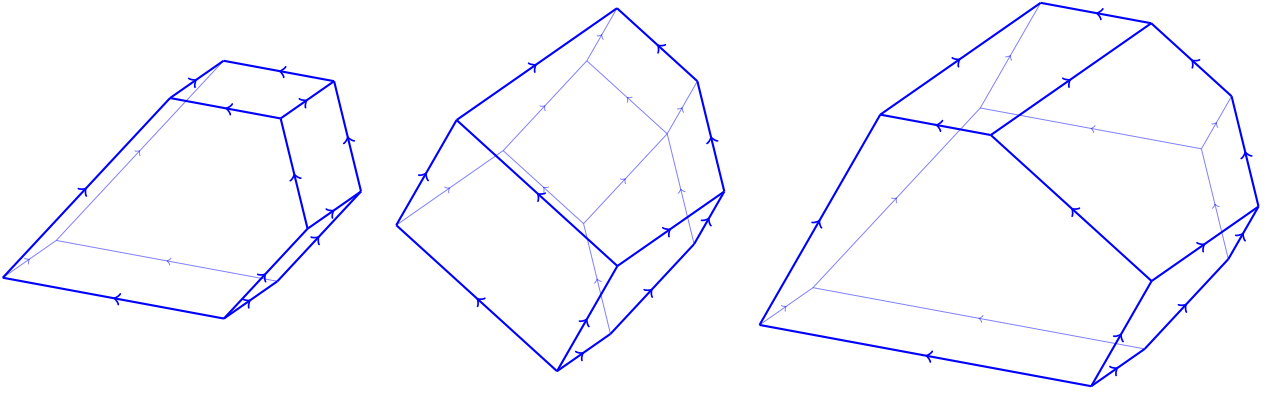

\para{Shard polytopes and deformed graphical zonotopes}
A \defn{deformation} of the graphical zonotope~$\Zono$ is any polytope whose normal fan coarsens the graphical fan~$\Fan$.
Under dilation and Minkowski addition, these deformations form a polyhedral cone, called the \defn{deformation cone} of the graphical zonotope~$\Zono$, whose interior is also called the \defn{type cone} of the graphical fan~$\Fan$.
We refer to~\cite{Postnikov, PostnikovReinerWilliams, McMullen-typeCone, PadrolPaluPilaudPlamondon, AlbertinPilaudRitter, PadrolPilaudRitter, PadrolPilaudPoullot-typeConeNestedFan, PadrolPilaudPoullot-typeConeGraphicalFan} for more details on the deformation cone of a polytope and type cone of a fan, in particular in the context of permutahedra and associahedra.

\medskip
\enlargethispage{.2cm}
For instance, deformations of the classical permutahedron were called generalized permutahedra in~\cite{Postnikov}.
One important result on deformed permutahedra is that they can all be written as Minkowski sum and difference of dilates of the faces of the standard simplex~\cite{Postnikov, ArdilaBenedettiDoker}.
This extends for graphical zonotopes as follows.

\begin{proposition}[\cite{PadrolPilaudPoullot-typeConeGraphicalFan}]
\label{prop:MinkowskiBasisFacesStandardSimplex}
For any directed acyclic graph~$D$ (not necessarily skeletal), any deformation of~$\Zono$ can be written as a Minkowski sum and difference of dilates of the faces~$\simplex_K$ of the standard simplex~$\simplex_V$ for all cliques~$K$ of~$D$ with~$|K| \ge 2$.
In fact, the faces~$\simplex_K$ for all cliques~$K$ of~$D$ with~$|K| \ge 2$ form a linear basis of rays of the deformation cone of the graphical zonotope~$\Zono$.
\end{proposition}

Here, we just want to observe a similar property for shard polytopes when~$D$ is skeletal.
We first observe that it directly follows from~\cite{PadrolPilaudRitter} that shard polytopes are Minkowski indecomposable (thus correspond to certain rays of the deformation cone of the graphical fan~$\Fan$).

\begin{proposition}
When~$D$ is skeletal, any deformation of~$\Zono$ can be written as a Minkowski sum and difference of dilates of the shard polytopes~$\shardPolytope$ for the ropes~$\rope$ of~$D$.
In other words, the shard polytopes~$\shardPolytope$ for the ropes~$\rope$ of~$D$ form a linear basis of rays of the deformation cone of the graphical zonotope~$\Zono$.
\end{proposition}

\begin{proof}
The same proof as \cite[Prop.~75]{PadrolPilaudRitter} shows that the shard polytopes correspond to linearly independent rays of the deformation cone of the graphical fan~$\Fan$.
The fact that they indeed form a basis is thus a consequence of \cref{prop:MinkowskiBasisFacesStandardSimplex} and \cref{lem:numberRopes}\,(i).
\end{proof}

When~$D$ is skeletal, we thus have two linear bases of the deformation cone of the graphical fan~$\Fan$: the faces~$\simplex_K$ provide a basis adapted to graphical fans of subgraphs of~$D$, while the shard polytopes~$\shardPolytope$ provide a basis adapted to quotient fans of congruences of~$\AR$.

\medskip
The deformation cone of the graphical zonotope~$\Zono$ is studied in details in~\cite{PadrolPilaudPoullot-typeConeGraphicalFan}, with a precise description of its facet description.
In view of the recent results of~\cite{PadrolPaluPilaudPlamondon, AlbertinPilaudRitter, PadrolPilaudPoullot-typeConeNestedFan}, it seems relevant to investigate the deformation cones of quotientopes of congruences of the acyclic reorientation lattices.

\begin{problem}
\label{pb:deformationConesQuotientopes}
Provide a (irredundant) facet description of the deformation cones of the quotientopes of~$D$, in particular for the sylvester, Cambrian and coherent congruences.
\end{problem}

\vspace{-.1cm}
\para{Associahedra as removahedra}
We now focus on quotientopes for coherent congruences and more specifically on associahedra and Cambrian associahedra.
Our next statement, illustrated in \cref{fig:graphicalZonotopesAssociahedra}, relates two constructions to obtain an associahedron for~$D$:
\begin{itemize}
\item either by deleting inequalities in the facet description of~$\Zono$, generalizing~\cite{ShniderSternberg}, 
\item or as Minkowski sums of faces of the standard simplex, generalizing~\cite{Postnikov}.
\end{itemize}
Let us just recall from our discussion above that the graphical zonotope~$\Zono$ 
\begin{itemize}
\item lives in the affine subspace $\K^\perp + \sum_{(u,v) \in A} \b{e}_v$ defined by the equations $\dotprod{\one_K}{\b{x}} = |A \cap \binom{K}{2}|$ for all connected components~$K$ of~$D$, and
\item is defined by the facet inequalities~$\dotprod{\one_U}{\b{x}} \ge \iota_U$ for all biconnected subsets~$U$ of~$D$, where~$\iota_U \eqdef |\set{a \in A}{|a \cap U| = 2}|$ counts the arcs of~$D$ with both endpoints in~$U$.
\end{itemize}

\begin{proposition}
\label{prop:associahedraRemovahedra}
Assume that~$D$ is skeletal.
The sylvester fan~$\Fan[(V,\varnothing)]$ is the normal fan of the associahedron~$\Asso[D]$ defined equivalently as
\begin{itemize}
\item the intersection of $\K^\perp + \sum_{(u,v) \in A} \b{e}_v$ with the halfspaces ${\bigset{\b{x} \in \R^V}{\dotprod{\one_U}{\b{x}} \ge \iota_U}}$ for all biconnected subsets~$U$ of~$D$ which are connected in the transitive reduction of~$D$,
\item the Minkowski sum of the faces~$\simplex_\pi$ of the standard simplex~$\simplex_V$, for all directed paths~$\pi$ in the transitive reduction of~$D$ whose endpoints are connected by an arc of~$D$.
\end{itemize}
\end{proposition}

\begin{proof}
By definition, the subrope ideal~$\bb{I}_{(V,\varnothing)}$ of the sylvester congruence contains precisely the ropes of the form~$\rope \eqdef (u, v, \down, \varnothing)$ for all arcs~$(u,v)$ of~$D$.
We have already mentioned that the shard polytope~$\shardPolytope$ of a rope~$\rope \eqdef (u, v, \down, \varnothing)$ is the face~$\simplex_{\{u,v\} \cup \down}$ of the standard simplex, translated by the vector~$-\b{e}_v$.
It follows that the Minkowski sum of the faces~$\simplex_\pi$ of the standard simplex~$\simplex_V$, for all directed paths~$\pi$ in the transitive reduction of~$D$ whose endpoints are connected by an arc of~$D$, is indeed an associahedron~$\Asso[D]$ by \cref{thm:MinkowskiSumShardPolytopes}.

\enlargethispage{.3cm}
We now prove the facet description of~$\Asso[D]$.
Observe first that $\Asso[D]$ is indeed contained in $\K^\perp + \sum_{(u,v) \in A} \b{e}_v$.
Moreover, since the normal fan of~$\Asso[D]$ is the sylvester fan~$\Fan[(V,\varnothing)]$, its rays indeed correspond to the biconnected subsets of~$D$ which are connected in the transitive reduction of~$D$.
For such a subset~$U$, we have
\[
\min_{\b{x} \in \sum_\pi \simplex_\pi} \dotprod{\one_U}{\b{x}} = \sum_\pi \min_{\b{x} \in \simplex_\pi} \dotprod{\one_U}{\b{x}} = \sum_\pi \delta_{\pi \subseteq U} = |\set{a \in A}{|a \cap U| = 2}| = \iota_U,
\]
where all sums range over the directed paths~$\pi$ in the transitive reduction of~$D$ whose endpoints are connected by an arc of~$D$.
We conclude that the facet inequality of~$\Asso[D]$ corresponding to~$U$ is indeed given by~$\dotprod{\one_U}{\b{x}} \ge \iota_U$, which is the facet inequality of~$\Zono$ corresponding to~$U$.
\end{proof}

In contrast, note that we are still missing a simple vertex description of the associahedron~$\Asso[D]$ similar to that of~\cite{Loday} for the classical associahedron.
We leave this question open for further research.

\begin{problem}
\label{pb:vertexDescriptionAssociahedron}
Provide a simple formula to describe the vertex of the associahedron~$\Asso[D]$ corresponding to a partial acyclic reorientation of~$\c{R}_{(V, \varnothing)}$.
\end{problem}

We now switch to arbitrary Cambrian congruences~$\equiv_{(\decorationDown, \decorationUp)}$ with~$\decorationDown \sqcup \decorationUp = V$.
We believe that \cref{prop:associahedraRemovahedra} extends to any Cambrian congruence, generalizing~\cite{HohlwegLange}.
The proof however is not as immediate and requires further investigation.

\begin{problem}
\label{pb:CambrianAssociahedraRemovahedra}
Prove that the Minkowski sum of the shard polytopes of the ropes of~$\bb{I}_{(\decorationDown, \decorationUp)}$ is obtained by deleting from the facet description of the graphical zonotope~$\Zono$ the inequalities normal to the rays of the graphical fan~$\Fan$ that are not rays of the quotient fan~$\Fan[(\decorationDown, \decorationUp)]$ (\ie the inequalities given by the biconnected subsets~$U$ of~$D$ such that exist~$u, v, w \in V$ with~$w$ along a directed path in~$D$ joining~$u$ to~$v$ such that $u,v \in U$ and $w \in \decorationDown \ssm U$, or $u,v \notin U$ and $w \in \decorationUp \cap U$).
\end{problem}

Let us now switch to arbitrary coherent congruences~$\equiv_{(\decorationDown, \decorationUp)}$.
As already observed in~\cite{PadrolPilaudRitter}, \cref{prop:associahedraRemovahedra} fails when~$\decorationDown \cap \decorationUp \ne \varnothing$ or~$\decorationDown \cup \decorationUp \ne V$.
Indeed, the classical permutahedron is actually not a positive Minkowski combination of the shard polytopes, see~\cite[Coro.~59]{PadrolPilaudRitter}.
However, we still conjecture that removing the appropriate inequalities in the facet description of the graphical zonotope~$\Zono$ defines a realization of the quotient fan~$\Fan[(\decorationDown, \decorationUp)]$ of any coherent congruence, which is proved in~\cite{PilaudPons-permutrees, AlbertinPilaudRitter} when~$D$ is a tournament.
This still requires some work.

\begin{problem}
\label{pb:coherentQuotientopesRemovahedra}
Prove that for any~$\decorationDown, \decorationUp \subseteq V$, the quotient fan~$\Fan[(\decorationDown, \decorationUp)]$ is the normal fan of the polytope obtained by deleting from the facet description of the graphical zonotope~$\Zono$ the inequalities normal to the rays of the graphical fan~$\Fan$ that are not rays of the quotient fan~$\Fan[(\decorationDown, \decorationUp)]$.
\end{problem}

\begin{figure}
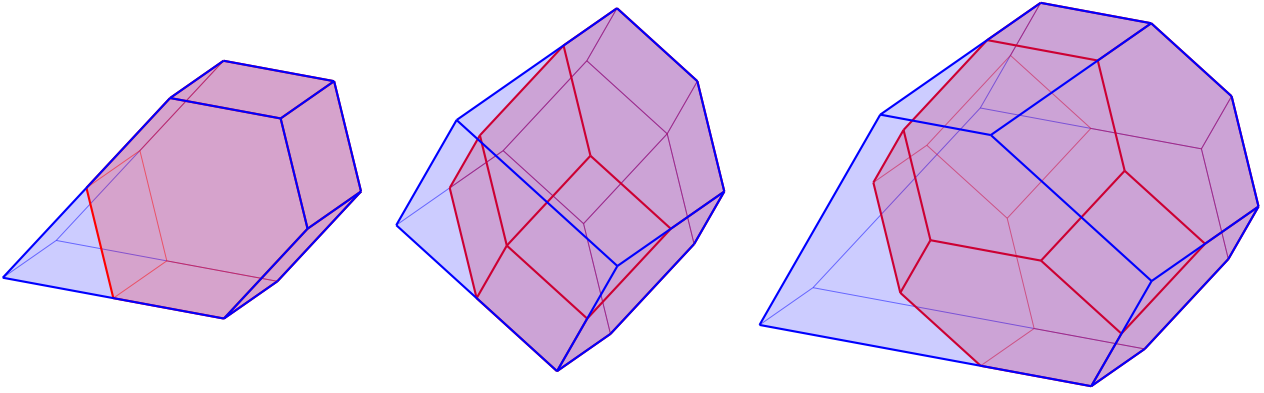

	\centerline{

\begin{tikzpicture}%
	[x={(-0.366215cm, -0.789554cm)},
	y={(0.235950cm, -0.590693cm)},
	z={(0.900119cm, -0.166391cm)},
	scale=1.000000,
	back/.style={very thin, opacity=0.5},
	edgeZono/.style={color=red, thick, decoration={markings, mark=at position 0.5 with {}}},
	edgeAsso/.style={color=blue, thick, decoration={markings, mark=at position 0.5 with {}}},
	facetZono/.style={fill=red,fill opacity=.2},
	facetAsso/.style={fill=blue,fill opacity=.2},
	vertex/.style={}]
%
%

\coordinate (0.00000, 0.00000, 0.00000) at (0.00000, 0.00000, 0.00000);
\coordinate (0.00000, 0.00000, 1.63299) at (0.00000, 0.00000, 1.63299);
\coordinate (1.33333, -0.94281, 0.00000) at (1.33333, -0.94281, 0.00000);
\coordinate (1.33333, -0.94281, 1.63299) at (1.33333, -0.94281, 1.63299);
\coordinate (4.00000, 0.00000, 1.63299) at (4.00000, 0.00000, 1.63299);
\coordinate (1.33333, 0.47140, 2.44949) at (1.33333, 0.47140, 2.44949);
\coordinate (4.00000, 0.00000, -1.63299) at (4.00000, 0.00000, -1.63299);
\coordinate (2.66667, -0.47140, 2.44949) at (2.66667, -0.47140, 2.44949);
\coordinate (2.66667, 0.94281, -1.63299) at (2.66667, 0.94281, -1.63299);
\coordinate (2.66667, 0.94281, 1.63299) at (2.66667, 0.94281, 1.63299);
\draw[edgeAsso,postaction={decorate},back] (2.66667, 0.94281, -1.63299) -- (0.00000, 0.00000, 0.00000);
\draw[edgeAsso,postaction={decorate},back] (4.00000, 0.00000, -1.63299) -- (2.66667, 0.94281, -1.63299);
\draw[edgeAsso,postaction={decorate},back] (2.66667, 0.94281, 1.63299) -- (2.66667, 0.94281, -1.63299);
\node[vertex] at (2.66667, 0.94281, -1.63299)     {};
\fill[facetAsso] (1.33333, -0.94281, 1.63299) -- (0.00000, 0.00000, 1.63299) -- (0.00000, 0.00000, 0.00000) -- (1.33333, -0.94281, 0.00000) -- cycle {};
\fill[facetAsso] (2.66667, -0.47140, 2.44949) -- (1.33333, -0.94281, 1.63299) -- (0.00000, 0.00000, 1.63299) -- (1.33333, 0.47140, 2.44949) -- cycle {};
\fill[facetAsso] (1.33333, 0.47140, 2.44949) -- (2.66667, 0.94281, 1.63299) -- (4.00000, 0.00000, 1.63299) -- (2.66667, -0.47140, 2.44949) -- cycle {};
\fill[facetAsso] (4.00000, 0.00000, 1.63299) -- (2.66667, -0.47140, 2.44949) -- (1.33333, -0.94281, 1.63299) -- (1.33333, -0.94281, 0.00000) -- (4.00000, 0.00000, -1.63299) -- cycle {};
\coordinate (0.00000, 0.00000, 0.00000) at (0.00000, 0.00000, 0.00000);
\coordinate (0.00000, 0.00000, 1.63299) at (0.00000, 0.00000, 1.63299);
\coordinate (1.33333, -0.94281, 0.00000) at (1.33333, -0.94281, 0.00000);
\coordinate (1.33333, -0.94281, 1.63299) at (1.33333, -0.94281, 1.63299);
\coordinate (1.33333, 0.47140, -0.81650) at (1.33333, 0.47140, -0.81650);
\coordinate (4.00000, 0.00000, 1.63299) at (4.00000, 0.00000, 1.63299);
\coordinate (1.33333, 0.47140, 2.44949) at (1.33333, 0.47140, 2.44949);
\coordinate (2.66667, -0.47140, -0.81650) at (2.66667, -0.47140, -0.81650);
\coordinate (4.00000, 0.00000, 0.00000) at (4.00000, 0.00000, 0.00000);
\coordinate (2.66667, -0.47140, 2.44949) at (2.66667, -0.47140, 2.44949);
\coordinate (2.66667, 0.94281, 0.00000) at (2.66667, 0.94281, 0.00000);
\coordinate (2.66667, 0.94281, 1.63299) at (2.66667, 0.94281, 1.63299);
\draw[edgeZono,postaction={decorate},back] (1.33333, 0.47140, -0.81650) -- (0.00000, 0.00000, 0.00000);
\draw[edgeZono,postaction={decorate},back] (2.66667, -0.47140, -0.81650) -- (1.33333, 0.47140, -0.81650);
\draw[edgeZono,postaction={decorate},back] (2.66667, 0.94281, 0.00000) -- (1.33333, 0.47140, -0.81650);
\draw[edgeZono,postaction={decorate},back] (4.00000, 0.00000, 0.00000) -- (2.66667, 0.94281, 0.00000);
\draw[edgeZono,postaction={decorate},back] (2.66667, 0.94281, 1.63299) -- (2.66667, 0.94281, 0.00000);
\node[vertex] at (1.33333, 0.47140, -0.81650)     {};
\node[vertex] at (2.66667, 0.94281, 0.00000)     {};
\fill[facetZono] (1.33333, -0.94281, 1.63299) -- (0.00000, 0.00000, 1.63299) -- (0.00000, 0.00000, 0.00000) -- (1.33333, -0.94281, 0.00000) -- cycle {};
\fill[facetZono] (2.66667, -0.47140, 2.44949) -- (1.33333, -0.94281, 1.63299) -- (0.00000, 0.00000, 1.63299) -- (1.33333, 0.47140, 2.44949) -- cycle {};
\fill[facetZono] (1.33333, 0.47140, 2.44949) -- (2.66667, 0.94281, 1.63299) -- (4.00000, 0.00000, 1.63299) -- (2.66667, -0.47140, 2.44949) -- cycle {};
\fill[facetZono] (4.00000, 0.00000, 1.63299) -- (2.66667, -0.47140, 2.44949) -- (1.33333, -0.94281, 1.63299) -- (1.33333, -0.94281, 0.00000) -- (2.66667, -0.47140, -0.81650) -- (4.00000, 0.00000, 0.00000) -- cycle {};
\draw[edgeZono,postaction={decorate}] (0.00000, 0.00000, 1.63299) -- (0.00000, 0.00000, 0.00000);
\draw[edgeZono,postaction={decorate}] (1.33333, -0.94281, 0.00000) -- (0.00000, 0.00000, 0.00000);
\draw[edgeZono,postaction={decorate}] (1.33333, -0.94281, 1.63299) -- (0.00000, 0.00000, 1.63299);
\draw[edgeZono,postaction={decorate}] (1.33333, 0.47140, 2.44949) -- (0.00000, 0.00000, 1.63299);
\draw[edgeZono,postaction={decorate}] (1.33333, -0.94281, 1.63299) -- (1.33333, -0.94281, 0.00000);
\draw[edgeZono,postaction={decorate}] (2.66667, -0.47140, -0.81650) -- (1.33333, -0.94281, 0.00000);
\draw[edgeZono,postaction={decorate}] (2.66667, -0.47140, 2.44949) -- (1.33333, -0.94281, 1.63299);
\draw[edgeZono,postaction={decorate}] (4.00000, 0.00000, 1.63299) -- (4.00000, 0.00000, 0.00000);
\draw[edgeZono,postaction={decorate}] (4.00000, 0.00000, 1.63299) -- (2.66667, -0.47140, 2.44949);
\draw[edgeZono,postaction={decorate}] (4.00000, 0.00000, 1.63299) -- (2.66667, 0.94281, 1.63299);
\draw[edgeZono,postaction={decorate}] (2.66667, -0.47140, 2.44949) -- (1.33333, 0.47140, 2.44949);
\draw[edgeZono,postaction={decorate}] (2.66667, 0.94281, 1.63299) -- (1.33333, 0.47140, 2.44949);
\draw[edgeZono,postaction={decorate}] (4.00000, 0.00000, 0.00000) -- (2.66667, -0.47140, -0.81650);
\node[vertex] at (0.00000, 0.00000, 0.00000)     {};
\node[vertex] at (0.00000, 0.00000, 1.63299)     {};
\node[vertex] at (1.33333, -0.94281, 0.00000)     {};
\node[vertex] at (1.33333, -0.94281, 1.63299)     {};
\node[vertex] at (4.00000, 0.00000, 1.63299)     {};
\node[vertex] at (1.33333, 0.47140, 2.44949)     {};
\node[vertex] at (2.66667, -0.47140, -0.81650)     {};
\node[vertex] at (4.00000, 0.00000, 0.00000)     {};
\node[vertex] at (2.66667, -0.47140, 2.44949)     {};
\node[vertex] at (2.66667, 0.94281, 1.63299)     {};
\draw[edgeAsso,postaction={decorate}] (0.00000, 0.00000, 1.63299) -- (0.00000, 0.00000, 0.00000);
\draw[edgeAsso,postaction={decorate}] (1.33333, -0.94281, 0.00000) -- (0.00000, 0.00000, 0.00000);
\draw[edgeAsso,postaction={decorate}] (1.33333, -0.94281, 1.63299) -- (0.00000, 0.00000, 1.63299);
\draw[edgeAsso,postaction={decorate}] (1.33333, 0.47140, 2.44949) -- (0.00000, 0.00000, 1.63299);
\draw[edgeAsso,postaction={decorate}] (1.33333, -0.94281, 1.63299) -- (1.33333, -0.94281, 0.00000);
\draw[edgeAsso,postaction={decorate}] (4.00000, 0.00000, -1.63299) -- (1.33333, -0.94281, 0.00000);
\draw[edgeAsso,postaction={decorate}] (2.66667, -0.47140, 2.44949) -- (1.33333, -0.94281, 1.63299);
\draw[edgeAsso,postaction={decorate}] (4.00000, 0.00000, 1.63299) -- (4.00000, 0.00000, -1.63299);
\draw[edgeAsso,postaction={decorate}] (4.00000, 0.00000, 1.63299) -- (2.66667, -0.47140, 2.44949);
\draw[edgeAsso,postaction={decorate}] (4.00000, 0.00000, 1.63299) -- (2.66667, 0.94281, 1.63299);
\draw[edgeAsso,postaction={decorate}] (2.66667, -0.47140, 2.44949) -- (1.33333, 0.47140, 2.44949);
\draw[edgeAsso,postaction={decorate}] (2.66667, 0.94281, 1.63299) -- (1.33333, 0.47140, 2.44949);
\node[vertex] at (0.00000, 0.00000, 0.00000)     {};
\node[vertex] at (0.00000, 0.00000, 1.63299)     {};
\node[vertex] at (1.33333, -0.94281, 0.00000)     {};
\node[vertex] at (1.33333, -0.94281, 1.63299)     {};
\node[vertex] at (4.00000, 0.00000, 1.63299)     {};
\node[vertex] at (1.33333, 0.47140, 2.44949)     {};
\node[vertex] at (4.00000, 0.00000, -1.63299)     {};
\node[vertex] at (2.66667, -0.47140, 2.44949)     {};
\node[vertex] at (2.66667, 0.94281, 1.63299)     {};
\end{tikzpicture}
			\raisebox{-.7cm}{\input{zonotopeAssociahedron4}}
			\raisebox{-.9cm}{\input{zonotopeAssociahedron7}}
	}
	\caption{The associahedra of \cref{fig:associahedra} are obtained by deleting inequalities in the facet description of the graphical zonotopes of \cref{fig:graphicalZonotopes}.}
	\label{fig:graphicalZonotopesAssociahedra}
\end{figure}

Finally, we switch to arbitrary congruences of~$\AR$.
when~$D$ is a tournament, it was shown in~\cite{AlbertinPilaudRitter} that the permutree congruences are the only congruences of the weak order whose quotient fan can be realized by deleting inequalities in the facet description of the classical permutahedron.
The analogue statement still needs to be investigated for an arbitrary skeletal directed acyclic graph~$D$.

\begin{problem}
\label{pb:coherentQuotientopesRemovahedra}
Prove that the coherent congruences are the only congruences of~$\AR$ whose quotient fan~$\Fan[\equiv]$ can be realized by deleting from the facet description of the graphical zonotope~$\Zono$ the inequalities normal to the rays of the graphical fan~$\Fan$ that are not rays of the quotient fan~$\Fan[\equiv]$.
\end{problem}


\section{Posets of regions of hyperplane arrangements}
\label{sec:posetRegions}

To conclude, we discuss the possible extensions of our results to the posets of regions of arbitrary hyperplane arrangements studied by A.~Bj\"orner, P.~Edelman and G.~Ziegler \mbox{in~\cite{Edelman, BjornerEdelmanZiegler}}.
Let~$\b{A}$ be a finite collection of non-zero vectors in~$\R^n$ which all belong to a halfspace.
Consider
\begin{itemize}
\item the arrangement~$\Arrang[\b{A}]$ formed by the hyperplanes~$\set{\b{x} \in \R^n}{\dotprod{\b{x}}{\b{a}} = 0}$ for~$\b{a} \in \b{A}$,
\item the zonotope~$\Zono[\b{A}]$ defined as the Minkowski sum of all segments~$[-\b{a}, \b{a}]$ for~$\b{a} \in \b{A}$.
\end{itemize}
These two objects are normal to each other: the regions of~$\Arrang[\b{A}]$ correspond to the vertices of~$\Zono[\b{A}]$ and the rays of~$\Arrang[\b{A}]$ correspond to the facets of~$\Zono[\b{A}]$.
We say that a region~$R$ of~$\Arrang[\b{A}]$ lies on the positive (resp.~negative) side of~$\b{a} \in \b{A}$ if it lies in the halfspace where the scalar product with~$\b{a}$ is positive (resp.~negative).
The \defn{positive set} of a region~$R$ is the subset of vectors of~$\b{A}$ for which $R$ lies on the positive side.
The region~$\polytope{B}$ on the negative side of all vectors in~$\b{A}$ is called the \defn{base region}.
The \defn{poset of regions}~$\PR[\b{A}]$ is the partial order on all regions of~$\Arrang[\b{A}]$ ordered by inclusion of their positive sets.
In other words, the Hasse diagram of this poset is the graph of the zonotope~$\Zono[\b{A}]$ oriented in the direction~$\sum_{\b{a} \in \b{A}} \b{a}$.
For instance, as already discussed in \cref{sec:quotientFansQuotientopes}, the acyclic reorientation poset~$\AR$ of a directed acyclic graph~$D$ is isomorphic to the poset of regions~$\PR[\b{A}_D]$ of the incidence configuration~$\b{A}_D \eqdef \set{\b{e}_u-\b{e}_v}{(u,v) \in D}$ of~$D$.
In general, it was proved in~\cite{Edelman, BjornerEdelmanZiegler} that
\begin{itemize}
\item if the poset of regions~$\PR[\b{A}]$ is a lattice, then the base region~$\polytope{B}$ is simplicial (or dually if the cone generated by~$\b{A}$ is simplicial),
\item if the arrangement~$\Arrang[\b{A}]$ is simplicial, then the poset of regions~$\PR[\b{A}]$ is a lattice.
\end{itemize}
We also note that N.~Reading showed in~\cite{Reading-PosetRegionsChapter} that the poset of regions~$\PR[\b{A}]$ is a congruence uniform lattice if and only if~$\Arrang[\b{A}]$ is \defn{tight} with respect to~$\polytope{B}$, meaning that for each region~$\polytope{R}$ of~$\Arrang[\b{A}]$, every pair of upper (resp.~lower) facets of~$\polytope{R}$ with respect to~$\polytope{B}$ intersects in a codimension~$2$ face.
However, there is still no characterization of the collections of vectors~$\b{A}$ whose poset of regions~$\PR[\b{A}]$ is a lattice.
In view of \cref{thm:characterizationLattice}, it is natural to consider the following conditions.

\begin{proposition}
\label{prop:criteriaPosetRegions}
The following conditions are equivalent for a set~$\b{A}$ of non-zero vectors in~$\R^n$:
\begin{itemize}
\item for any linear hyperplane~$H$ of~$\R^n$, the cone generated by the vectors of~$\b{A} \cap H$ is simplicial,
\item for any $d$-dimensional face~$\polytope{F}$ of~$\Zono[\b{A}]$, the source of~$\polytope{F}$ in~$\PR[\b{A}]$ has degree $d$ in~$\polytope{F}$.
\end{itemize}
\end{proposition}

\begin{proof}
Given a linear hyperplane~$H$ of~$\R^n$, let~$\polytope{F}$ be the face of~$\Zono[\b{A}]$ maximizing the dot product with a normal vector of~$H$.
Conversely, given a face~$\polytope{F}$ of~$\Zono[\b{A}]$, let~$H$ be an arbitrary supporting hyperplane of~$\polytope{F}$.
Then the cone generated by~$\b{A} \cap H$ is simplicial if and only if the source of~$\polytope{F}$ in~$\PR[\b{A}]$ is a simple vertex of~$\polytope{F}$.
\end{proof}

The conditions of \cref{prop:criteriaPosetRegions} just translate to arbitrary arrangements the conditions of \cref{thm:characterizationLattice} for graphical arrangements.
Indeed, for a directed acyclic graph~$D$, choosing a face~$\polytope{F}$ of~$\Zono[D]$ is choosing an ordered partition~$(\mu, \omega)$ of~$D$, and requiring the source of~$\polytope{F}$ in~$\PR[\b{A}]$ to be a simple vertex is equivalent to requiring that the transitive reduction of the subgraph of~$D$ induced by each part of~$\mu$ is a forest.
It is not difficult to see that these conditions are necessary for~$\PR[\b{A}]$ to be a lattice.

\begin{lemma}
\label{lem:criteriaPosetRegions}
If the poset of regions~$\PR[\b{A}]$ is a lattice, then~$\b{A}$ fulfills the conditions of \cref{prop:criteriaPosetRegions}.
\end{lemma}

\begin{proof}
Fix a linear hyperplane~$H$ of~$\R^n$, and let~$\polytope{F}$ be one of the two faces of~$\Zono[\b{A}]$ whose supporting hyperplanes are parallel to~$H$.
Then the restriction of~$\PR[\b{A}]$ to the vertices of~$\polytope{F}$ is an interval of~$\PR[\b{A}]$ isomorphic to the poset of region~$\PR[\b{A} \cap H]$.
Since~$\PR[\b{A}]$ is a lattice, we obtain that $\PR[\b{A} \cap H]$ is a lattice and thus that the cone generated by the vectors of~$\b{A} \cap H$ is simplicial.
\end{proof}

However, in contrast to \cref{thm:characterizationLattice} for graphical arrangements, the conditions of \cref{prop:criteriaPosetRegions} are not sufficient for~$\PR[\b{A}]$ to be a lattice.
Note that the first counter-examples arise in dimension~$4$ (since in dimension~$3$, the poset of regions is a lattice as soon as the base region is simplicial~\cite[Thm.~3.2]{BjornerEdelmanZiegler}, which is implied by the conditions of \cref{prop:criteriaPosetRegions}).
The following counter-example is an adaptation of~\cite[Exm.~3.3]{BjornerEdelmanZiegler}.

\begin{example}
\label{exm:criteriaPosetRegions}
The set~$\b{A}$ of vectors
\[
\begin{bmatrix} 1 \\ 0 \\ 0 \\ 0 \end{bmatrix} \quad
\begin{bmatrix} 0 \\ 1 \\ 0 \\ 0 \end{bmatrix} \quad
\begin{bmatrix} -1 \\ -2 \\ -1 \\ 0 \end{bmatrix} \quad
\begin{bmatrix} 2 \\ 1 \\ 1 \\ 0 \end{bmatrix} \quad
\begin{bmatrix} 0 \\ 0 \\ 1 \\ -1 \end{bmatrix} \quad
\begin{bmatrix} 0 \\ 0 \\ 0 \\ -1 \end{bmatrix}
\]
fulfills the conditions of \cref{prop:criteriaPosetRegions} but the poset of regions~$\PR[\b{A}]$ is not a lattice.
\end{example}

Even if the conditions of \cref{prop:criteriaPosetRegions} fail to characterize the collections of vectors whose poset of regions is a lattice, they might be sufficient for certain well-behaved collections of vectors, in particular for subsets of root systems of finite Coxeter groups.
It holds for root systems of rank at most~$3$ by~\cite[Thm.~3.2]{BjornerEdelmanZiegler}, for all type~$A$ root systems by \cref{thm:characterizationLattice}, and was checked by computer experiments for the type~$D_4$ root system.
However, it already fails in type~$B_4$.

\begin{example}
\label{exm:CoxeterArrangements}
The set~$\b{A}$ of vectors
\[
\begin{bmatrix} 1 \\ 0 \\ 0 \\ 0 \end{bmatrix} \quad
\begin{bmatrix} 0 \\ 1 \\ 0 \\ 0 \end{bmatrix} \quad
\begin{bmatrix} 0 \\ 0 \\ 1 \\ 0 \end{bmatrix} \quad
\begin{bmatrix} 0 \\ 0 \\ 0 \\ 1 \end{bmatrix} \quad
\begin{bmatrix} 1 \\ 1 \\ 1 \\ 0 \end{bmatrix} \quad
\begin{bmatrix} 0 \\ 1 \\ 2 \\ 2 \end{bmatrix} \quad
\]
of the type~$B_4$ root system fulfills the conditions of \cref{prop:criteriaPosetRegions} but the poset of regions~$\PR[\b{A}]$ is not a lattice.
\end{example}

In contrast to the satisfactory characterization of the poset of regions which are congruence uniform lattices~\cite{Reading-PosetRegionsChapter}, the characterization of the posets of regions which are just lattices thus remains largely open.


\addtocontents{toc}{ \vspace{.1cm} }
\section*{Acknowledgements}

I am grateful to A.~Padrol and G.~Poullot for discussions on graphical zonotopes (in particular for pointing out references for \cref{prop:simplicialGraphicalFan}), to N.~Cartier for the counter-example of \cref{exm:CoxeterArrangements}, and to the DGeCo seminar group for suggestions on a preliminary exposition of the results of this paper.
I also thank an anonymous referee for many comments and suggestions that improved the presentation of this paper.


\addtocontents{toc}{ \vspace{.1cm} }
\bibliographystyle{alpha}
\bibliography{acyclicReorientationLattices}
\label{sec:biblio}

\end{document}